\documentclass[11pt,letterpaper]{article}
\usepackage[margin=1in]{geometry}
\usepackage{times}
\usepackage{natbib}
\setcitestyle{authoryear,round,citesep={;},aysep={,},yysep={;}}
\usepackage[T1]{fontenc}
\ifdefined\XeTeXversion
  \usepackage{newunicodechar}
  \newunicodechar{–}{\textendash}
\fi
\usepackage{amsmath,amssymb,amsthm,mathtools}
\usepackage{graphicx,booktabs,tabularx,longtable,array,multirow}
\usepackage{xcolor,microtype,url,hyperref,placeins}
\usepackage{tikz,pgfplots}
\makeatletter

\def\section{\@startsection{section}{1}{\z@}
  {-0.88ex plus -0.10ex minus -.05ex}
  {0.32ex plus 0.06ex minus .05ex}
  {\large\sc\raggedright}}

\def\subsection{\@startsection{subsection}{2}{\z@}
  {-0.88ex plus -0.10ex minus -.05ex}
  {0.32ex plus 0.06ex minus .05ex}
  {\normalsize\sc\raggedright}}

\def\subsubsection{\@startsection{subsubsection}{3}{\z@}
  {-0.88ex plus -0.10ex minus -.05ex}
  {0.28ex plus 0.05ex minus .05ex}
  {\normalsize\sc\raggedright}}

\def\paragraph{\@startsection{paragraph}{4}{\z@}
  {1.12ex plus 0.10ex minus .05ex}
  {-0.70em}
  {\normalsize\bf}}

\makeatother
\pgfplotsset{compat=1.18}
\usetikzlibrary{arrows.meta,positioning,calc,fit,backgrounds}
\DeclareMathVersion{figuretimes}
\SetSymbolFont{operators}{figuretimes}{OT1}{ztmcm}{m}{n}
\SetSymbolFont{letters}{figuretimes}{OML}{ztmcm}{m}{it}
\SetSymbolFont{symbols}{figuretimes}{OMS}{ztmcm}{m}{n}
\SetSymbolFont{largesymbols}{figuretimes}{OMX}{ztmcm}{m}{n}
\SetMathAlphabet{\mathrm}{figuretimes}{OT1}{ptm}{m}{n}
\SetMathAlphabet{\mathbf}{figuretimes}{OT1}{ptm}{bx}{n}
\SetMathAlphabet{\mathit}{figuretimes}{OT1}{ptm}{m}{it}
\graphicspath{{figures/}{./}}
\hypersetup{hidelinks,pdftitle={Finance-Informed Operator Learning for Option Pricing with Quantum-Compatible Realizations},pdfauthor={Jiarui Feng, Bingyang Hu, Yu Jiang, Changhong Mou, Yeyu Zhang},pdfsubject={arXiv preprint}}
\newtheorem{proposition}{Proposition}[section]

\newcommand{\R}{\mathbb{R}}
\newcommand{\E}{\mathbb{E}}
\newcommand{\Cref}{C^{\mathrm{ref}}}
\newcommand{\Pa}{\mathcal{P}_a}
\newcommand{\Aa}{\mathcal{A}_{a,\tau}}
\newcommand{\FAL}{\mathcal{C}}
\newcommand{\norm}[1]{\left\lVert #1\right\rVert}
\newcommand{\relu}[1]{\left(#1\right)^+}
\newcommand{\diag}{\operatorname{diag}}
\newcommand{\offdiag}{\operatorname{offdiag}}
\newcommand{\sigm}{\operatorname{sigm}}
\newcommand{\spdelta}{\operatorname{sp}_{\delta}}
\newcommand{\RelL}{\mathrm{RelL2}}
\newcommand{\pP}{\widehat{V}_{\mathrm{FI}}}
\newcolumntype{Y}{>{\raggedright\arraybackslash}X}
\numberwithin{equation}{section}
\title{Finance-Informed Operator Learning\\for Option Pricing with\\Quantum-Compatible Realizations}
\author{Jiarui Feng\textsuperscript{1}, Bingyang Hu\textsuperscript{2}, Yu Jiang\textsuperscript{1},\\
Changhong Mou\textsuperscript{3}, Yeyu Zhang\textsuperscript{1}\\[0.6em]
\small \textsuperscript{1}School of Mathematics, Shanghai University of Finance and Economics, China\\
\small \textsuperscript{2}Department of Mathematics and Statistics, College of Sciences and Mathematics,\\
\small Auburn University, USA\\
\small \textsuperscript{3}Department of Mathematics and Statistics, Utah State University, USA}
\date{}
\begin{document}
\maketitle
\begin{abstract}
Pricing European options under local volatility requires solving a partial differential equation whose coefficients change with every recalibration, and practitioners repeatedly need not only prices but also their sensitivities to the underlying across spot--time surfaces for many strike and maturity configurations. Neural surrogates are an efficient alternative, approximating the solution operator directly so that model evaluation replaces repeated numerical solves. Near expiry, however, the solution loses regularity, and a network trained on prices alone can misstate the curvature needed for hedging or return prices outside the no-arbitrage bounds. We propose a finance-informed Deep Operator Network (FI-DeepONet) that decomposes the solution operator into a closed-form reference price and an additive learned correction. The framework is finance-informed in that financial structure enters the representation and the output map rather than the loss: the reference uses the strike-line integrated variance---the time integral of the squared local volatility evaluated at the strike, which captures the leading near-expiry curvature singularity, and a smooth monotone admissibility layer enforces pointwise no-arbitrage price bounds, at the cost of piecewise rather than global smoothness in the spot variable. A deterministic post-processing step then corrects the predicted grid after inference and is evaluated separately from the learned operator. We establish a short-maturity estimate and differentiated asymptotics for the exact correction, together with identities relating errors in the learned correction to pricing and sensitivity errors, as well as PDE-residual errors. On in-distribution (ID) tests, the model reduces the global relative price error by nearly an order of magnitude over vanilla and physics-informed DeepONet baselines across independent initializations. On out-of-distribution (OOD) tests, generalization is moderate when the input parameters lie outside the training range but within the same parametric class of local-volatility functions. We also test the trained model on real market data using index option quotes and find that it remains accurate without market-specific network retraining, though it does not improve on the analytic formula given the same volatility input. We further give a quantum-compatible realization of the trained operator, in which selected linear maps are either compiled exactly or approximated within a restricted diagonal–orthogonal family before supervised adaptation.
\end{abstract}
\section{Introduction}
\label{sec:intro}

Many scientific and engineering workflows solve the same family of partial differential equations (PDEs) many times, once for each new coefficient function or parameter setting. Neural operators pay off this cost by learning the map from PDE inputs to solutions, so that a single trained model replaces the numerical solver at inference time \citep{r17,r18,r19,r20}. Option pricing is a natural instance. A European call option gives its holder the right to buy an asset at a fixed price (the strike) on a fixed future date (the maturity). Under a local-volatility model, where the asset's volatility is a given function of asset price and time, the option price solves a linear parabolic PDE whose diffusion coefficient is set by that function \citep{r01,r02,r03}. Because the volatility function is recalibrated frequently and many contracts must be priced under each calibration, an operator that maps the volatility function and contract terms to the full price surface can replace a large number of repeated solves.
Accurate prices alone are not enough in this setting. Risk management also relies on the sensitivities of the price to the asset price: its first derivative, Delta, and its second derivative, Gamma. These become especially challenging to learn near maturity. The payoff has a kink at the strike (it is zero below the strike and grows linearly above it), so as maturity approaches, Gamma grows without bound at the strike and concentrates in a shrinking neighborhood around it \citep{r06,r07}. A model with small average price error can therefore still have
large Gamma errors in the near-strike, short-maturity region. A generic network must learn this near-singular behavior from data, even though its form is known in closed form.
This raises a question: where should known (prior) knowledge enter a neural operator? One common approach is physics-informed operator learning, which
adds the PDE residual and boundary conditions to the training loss~\citep{r27,r28,r29}. This encourages satisfaction of the equation but does not directly supply the singular structure of the solution, and its performance is sensitive to loss balancing, optimization, and the placement of collocation points \citep{r30,r31,r32}. An alternative, standard in classical numerical analysis but less common in operator learning, is to build the structure into the representation: subtract an analytic function that carries the singularity and let the network learn only the smoother remainder. Option pricing offers three such pieces of structure. (i) The payoff at maturity is known exactly. (ii) No-arbitrage arguments confine the price to a known interval: it cannot exceed the asset price, and it cannot fall below the larger of zero and the asset price minus the discounted strike. (iii) Near the strike at short maturities, the price is well approximated by the Black--Scholes formula parameterized by the strike-line integrated variance, that is, the time integral of the squared local volatility evaluated at the strike.

We propose the finance-informed Deep Operator Network (FI-DeepONet), which uses all three. Its core is a strike-matched carrier--residual ansatz (RA): the price is written as an analytic reference, which we call the carrier, plus a learned correction scaled by the remaining time to maturity. The carrier is the Black--Scholes price parameterized by this strike-line integrated variance. It therefore depends on the input volatility function, not only on the contract terms, and it reproduces the payoff kink exactly. 
Under standard well-posedness conditions, assuming that the local
variance is bounded above and below by positive constants and is
Lipschitz in log-moneyness, we prove that the exact correction
vanishes at least linearly in the remaining time to maturity. This motivates the rescaling: dividing the correction by the remaining time gives a bounded learning target, and when the leading-order term is nonzero, linear scaling is the only power scaling that keeps the target both bounded and nonvanishing at maturity. A Financial Admissibility Layer (FAL), a smooth monotone map into the no-arbitrage interval, then enforces the price bounds at every output point. The model is trained on reconstructed prices with a standard regression loss. We retain a standard branch--trunk DeepONet architecture and encode
the finance-informed structure through the representation and output map rather than the loss. The closest precedents are residual learning around analytic pricing approximations \citep{r72}, exact-terminal and singularity-aware constructions \citep{r47}, and maturity-gated Black--Scholes representations \citep{r70}. Our reference is distinguished by strike-line variance matching: it integrates the input local-variance field along the strike line to construct the carrier's strike-line integrated variance. We also analyze the resulting correction and its derivatives, rather than only reporting lower errors. Appendix~\ref{app:context} gives a broader comparison with related work.

Our contributions are as follows.
\begin{itemize}
    \item \textbf{A finance-informed operator representation.} We propose a strike-matched carrier--residual ansatz in which an input-dependent Black--Scholes carrier, constructed from the strike-line integrated variance, captures the terminal singularity, while a DeepONet learns only a time-rescaled correction. A Financial Admissibility Layer maps every output into the no-arbitrage interval. The backbone and training objective are those of a standard DeepONet trained by price regression, so prior knowledge enters through the representation rather than through the loss.

    \item \textbf{Short-maturity analysis of the learning target.} We prove that the exact correction vanishes at least linearly in the remaining time to maturity and characterize its limiting shape, together with its first two spot derivatives, near the strike. We show that linear rescaling is the unique power scaling that keeps the target bounded and nonzero when the leading-order term is nonzero, and we derive identities that relate errors in the learned correction to errors in price, Delta, Gamma, and the PDE residual.

    \item \textbf{Improved accuracy with controlled attribution.} On a local-volatility benchmark with finite-difference references, FI-DeepONet reduces global relative price error by about $6.2\times$ and $7.5\times$ relative to data-driven and physics-informed DeepONet baselines, respectively, and also reduces Delta and Gamma errors. Ablations attribute the gains to the change of learning target and the output layer, and quantify the trade-off between adding a PDE loss and price accuracy. The model generalizes to shifted contract and volatility parameters, while a selected final checkpoint does not accurately reproduce volatility perturbations outside the training family.

    \item \textbf{Quantum-compatible realization and market-data evaluation.} In contrast to Quantum DeepONet, which trains an orthogonally parameterized network from the start \citep{r15}, we begin from the pretrained dense model and realize selected linear maps with circuits that implement orthogonal transformations. We distinguish exact compilation of selected weight matrices from restricted circuit-compatible approximation followed by adaptation, and quantify the accuracy cost relative to classical structured approximations of equal computational cost. We further evaluate the model on index-option quotes, using implied volatilities as constant-volatility inputs and Black--Scholes prices on the same inputs as references, to assess transfer from synthetic to market-derived inputs. The learned operator remains accurate on market-derived inputs without market-specific retraining, although it does not outperform the same-input Black--Scholes reference. Neither study is intended to demonstrate a quantum speedup or a market calibration.
\end{itemize}
\section{Pricing problem and finance-informed representation}
\label{sec:model}

We consider a European call option on an asset that pays no dividends.
The holder may buy the asset at a fixed price $K>0$ (the \emph{strike})
on its expiry date $T>0$ (the \emph{maturity}).
Let $V(S,t)$ be the option value when the asset's current
(\emph{spot}) price is $S$ at time $t$.
Its terminal payoff is $V(S,T)=\max(S-K,0)$.
The local volatility $\sigma(S,t)$ specifies the scale of random
proportional price fluctuations at each asset price and time.
With constant risk-free interest rate $r$, no-arbitrage pricing gives
the backward PDE \citep{r01,r03}
\begin{equation}
\partial_t V
+\tfrac12 \sigma^2(S,t)S^2\partial_{SS}V
+rS\partial_S V-rV=0,
\qquad S>0,\ t<T.
\end{equation}
We write $a=\sigma^2$ for the local variance and $\tau=T-t$
for the remaining time.
At the operator level, the learning task is
$(a,K,T,r)\mapsto V(\cdot,\cdot)$,
with $(S,t)$ specifying where to evaluate the price surface.
Delta, $\Delta=\partial_S V$, measures price sensitivity to the asset
price; Gamma, $\Gamma=\partial_{SS}V$, measures its curvature.
A financial glossary is provided in Appendix~\ref{app:finance}.
No-arbitrage arguments also give the pointwise price bounds
\begin{equation}
L(S,\tau)=\relu{S-Ke^{-r\tau}}
\le V(S,T-\tau)\le U(S,\tau)=S.
\label{eq:bounds}
\end{equation}

\paragraph{Strike-matched carrier and learned correction.}
Instead of learning the full price directly, we separate an analytic
reference (the carrier) from a learned correction.
The reference fixes the spatial argument of the local variance at
$S=K$ while retaining its time dependence.
The resulting strike-line integrated variance is
\begin{equation}
A_0(\tau)=\int_0^\tau a(K,T-u)\,du.
\label{eq:strikeline}
\end{equation}
The resulting Black--Scholes carrier is
$\Cref=S\Phi(d_+)-Ke^{-r\tau}\Phi(d_-)$,
where $d_\pm=[\log(S/K)+r\tau\pm A_0/2]/\sqrt{A_0}$
and $\Phi$ is the standard normal distribution function.
It exactly prices this spatially frozen model and equals the
payoff at $\tau=0$ (Appendix~\ref{app:ra}).
The benchmark directly evaluates $a(K,\cdot)$ from the coefficient
generator, without sensor interpolation (Appendix~\ref{app:solver}).
Sensor-only inputs would require approximate strike-line evaluation;
Appendix~\ref{app:ra} treats the resulting reference error.
The strike-matched carrier--residual ansatz (RA) reconstructs the pre-FAL price as
\begin{equation}
z_\theta=\Cref+\chi R_\theta,
\qquad \chi=\tau/T.
\label{eq:ra}
\end{equation}
Here $\chi$ is the remaining-life fraction and $R_\theta$ is the
learned time-rescaled correction in price units.
Under the value-estimate assumptions of Section~\ref{sec:theory},
including positive boundedness and log-moneyness Lipschitz regularity,
$W=V-\Cref=O(\tau)$, so
$R^\ast=W/\chi$ is bounded.
Differentiated profiles require stronger local smoothness and allow
spot derivatives to grow near expiry.
Training supervises reconstructed prices, not correction labels.

\begin{figure}[!htbp]
\centering
\resizebox{\linewidth}{!}{\input{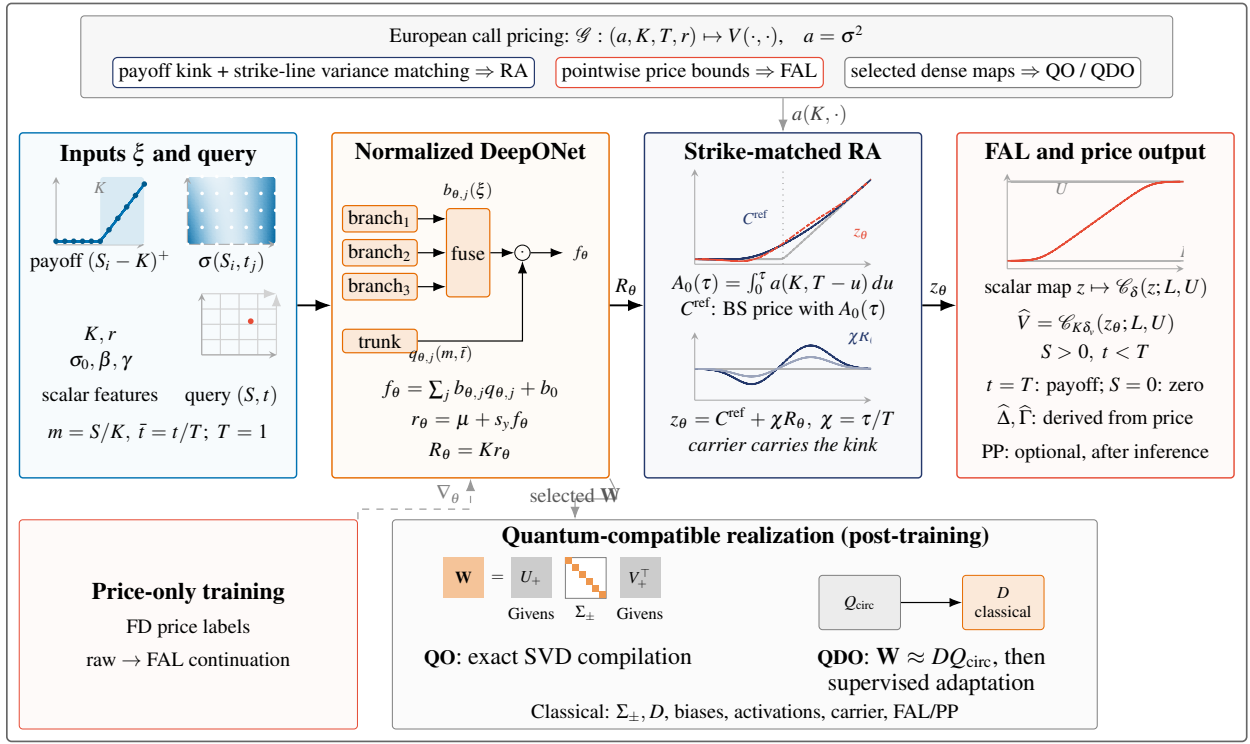}}
\caption{FI-DeepONet overview.
The carrier uses strike-line integrated variance; DeepONet predicts
$R_\theta$, RA reconstructs $z_\theta$, and FAL enforces pointwise
price bounds.
Training uses reconstructed-price regression on finite-difference (FD) labels.
Post-training QO exactly compiles selected maps; QDO approximates
and adapts them. Plots are schematic.}
\label{fig:mechanism}
\end{figure}
\FloatBarrier

\paragraph{DeepONet architecture.}
The branch input $\xi$ combines sampled payoff values,
sampled volatilities $\sigma(S_i,t_j)$, and scalar features.
Separate encoders are fused and combined with a trunk evaluated at
$m=S/K$ (asset price relative to the strike) and
$\bar t=t/T$ (normalized elapsed time).
The branch--trunk readout gives
$r_\theta=\mu+s_y f_\theta$ and $R_\theta=Kr_\theta$,
where $\mu,s_y$ are fixed training-normalization constants
(Appendix~\ref{app:training}).
The reported benchmark fixes $T=1$ and uses scalar features
$(K,r,\sigma_0,\beta,\gamma)$ for the contract and volatility family
(Appendix~\ref{app:data}).
Retaining absolute payoff samples and $K$ with fixed standardizers
makes the implementation scale-aware, not exactly scale-equivariant
(Appendix~\ref{app:scale}).

\paragraph{Financial Admissibility Layer and output stage.}
FAL maps the raw price $z_\theta$ into the interval in
\eqref{eq:bounds}.
Using the softplus
$\operatorname{sp}_\delta(q)=\delta\log(1+e^{q/\delta})$,
a smooth approximation of $\max(q,0)$, define
\begin{equation}
\FAL_\delta(z;L,U)
=L+\operatorname{sp}_\delta(z-L)-\operatorname{sp}_\delta(z-U).
\label{eq:mainfal}
\end{equation}
For fixed bounds, the map is smooth, increasing, and nonexpansive,
but introduces smoothing bias (Appendix~\ref{app:fal}).
For $S>0,t<T$, the prediction is
$\pP=\FAL_{K\delta_v}(z_\theta;L,U)$ with $\delta_v=0.002$.
Exact payoff and zero-spot assignments were added after training
with unchanged weights and no held-out outcomes.
They do not remove the fixed-width near-terminal limit bias.
FAL enforces pointwise bounds, not spot monotonicity or convexity;
the kinked lower bound makes the price only piecewise smooth,
so Greek and PDE identities apply on smooth pieces.

FI denotes this output before optional grid post-processing (PP).
PP repairs selected discrete price-bound, terminal, time-monotonicity,
and spot-slope violations after inference, without changing weights;
it is evaluated separately and provides no continuous-surface
arbitrage certificate (Appendix~\ref{app:pp}).
The complete OOD study retains an earlier output convention
(Appendix~\ref{app:ood}).
\section{Theoretical Analysis of Strike-Line Variance Matching}
\label{sec:theory}
We write $x=\log(S/K)$ for log-moneyness and
$W(S,t)=V(S,t)-\Cref(S,t)$ for the exact correction, using the calendar-time convention of Section~\ref{sec:model}.
The value estimate assumes local variance bounded above and below by positive constants and uniformly Lipschitz in $x$. The differentiated profiles require the stronger local smoothness specified in Appendix~\ref{app:ra}, which states the full assumptions.

\begin{proposition}[Correction cancellation]
\label{prop:value}
Under the value-estimate assumptions in Appendix~\ref{app:ra}, there is a constant $C$, independent of $\tau\in(0,T]$, such that
\begin{equation}
\sup_{x\in\R}\frac{|W(Ke^x,T-\tau)|}{K}\le C\tau,
\label{eq:valuebound}
\end{equation}
uniformly over the compact class of contracts stated there.
\end{proposition}
The correction has zero terminal data and source
$\tfrac12[a(S,T-\tau)-a(K,T-\tau)]S^2\partial_{SS}\Cref$.
In the $O(\sqrt\tau)$ near-strike band, the Lipschitz coefficient
difference is $O(\sqrt\tau)$ and carrier curvature is $O(K/\sqrt\tau)$;
Gaussian decay controls the tails.
The source is therefore bounded, giving \eqref{eq:valuebound}
after time integration.
A persistent strike-line mismatch can instead introduce an
$O(\sqrt\tau)$ contribution (Appendix~\ref{app:ra}).
\begin{proposition}[First-order skew profile and critical scaling]
\label{prop:profile}
Under the differentiated-profile assumptions in
Appendix~\ref{app:ra}, let $a_*=a(K,T)$ be the local variance at the strike and maturity, $b_*=K\partial_S a(K,T)$ its log-moneyness slope,
and $g_{a_*}$ the centered Gaussian density with variance $a_*$.
To resolve the shrinking $O(\sqrt{\tau})$ near-strike layer, we keep
the scaled log-moneyness $y=x/\sqrt{\tau}$ fixed.
For fixed $y$ and $S_\tau=Ke^{y\sqrt{\tau}}$,
\begin{equation}
\frac{W(S_\tau,T-\tau)}{K\tau}\longrightarrow \frac{b_*}{4}\,y\,g_{a_*}(y)\qquad(\tau\to0).
\label{eq:mainprofile}
\end{equation}
Along these paths, if $b_*y\ne0$, then $W/\chi^\alpha$ with
$\chi=\tau/T$ has a finite nonzero limit if and only if $\alpha=1$.
The corresponding profiles of the Delta and Gamma corrections,
$\partial_S W$ and $\partial_{SS}W$, are given in
\eqref{eq:profiles}. Thus the linear gate $\chi=\tau/T$ in (2.4) is the critical value-level rescaling near expiry.
\end{proposition}
Boundedness of $R^\ast$ does not imply bounded spot derivatives:
along these paths, $\partial_S R^\ast=O(\tau^{-1/2})$ and
$\partial_{SS}R^\ast=O(\tau^{-1})$.
Manufactured-solution checks and the exponent ablation distinguish
these asymptotics from finite-budget accuracy
(Appendix~\ref{app:verification}; Section~\ref{sec:ablation}).

\paragraph{Error transfer through the output stages.}
Let $\rho_\theta=R_\theta-R^\ast$ be the correction error. On the
nonterminal domain $\Omega$ under consideration, define
$N_\rho:=\sup_{\Omega}\bigl(
|\rho_\theta|/K
+\sqrt{\chi}\,|\partial_S\rho_\theta|
+\chi K|\partial_{SS}\rho_\theta|
\bigr)$.
Since $z_\theta-V=\chi\rho_\theta$, a finite $N_\rho$ gives the
pre-FAL bounds
\begin{equation}
\frac{|z_\theta-V|}{K}\le\chi N_\rho,\qquad
|\partial_S z_\theta-\Delta|\le\sqrt{\chi}\,N_\rho,\qquad
K|\partial_{SS}z_\theta-\Gamma|\le N_\rho,
\label{eq:transfer}
\end{equation}
where $\Delta=\partial_S V$ and $\Gamma=\partial_{SS}V$.
These conditional bounds give $O(\chi)$ normalized price error
and $O(\sqrt{\chi})$ Delta error; they bound the $K$-scaled Gamma error
without requiring it to vanish.
These bounds are conditional on finite $N_\rho$, which finite-grid checks do not establish.
After FAL, price error includes smoothing bias, while Greek errors also
involve derivatives of the bounds and output map
(Section~\ref{sec:model}; Appendix~\ref{app:stages}).
Writing
$P_a:=\partial_\tau-\frac12 a(S,T-\tau)S^2\partial_{SS}
-rS\partial_S+r$
for the time-to-maturity pricing operator, the exact pre-FAL identity is
$P_a z_\theta=\chi P_a\rho_\theta+\rho_\theta/T$.
This connects correction error to the governing-equation defect without
requiring a PDE loss. The reported finite-grid residual RMS is insufficient to certify the required continuum regularity.
\section{Numerical Experiments}
\label{sec:experiments}
Principal comparisons use pre-PP outputs (FI after FAL), FD Greeks
and PDE diagnostics, and five-seed means; sample standard deviations
are reported in Appendix~\ref{app:attribution}.
Mechanism ablations use pre-FAL outputs; the PDE-loss factorial
uses automatic-differentiation (AD) diagnostics.

\subsection{Benchmark and metrics}
\label{sec:metrics}
\paragraph{Data and models.}
The fixed-$T=1$ family varies volatility level $\sigma_0$, spot
dependence $\beta$, time dependence $\gamma$, strike $K$, and rate $r$.
Fixed splits contain 512/32/32 training/validation/ID surfaces and
320 parameter-support OOD surfaces.
Sensors, inputs, and direct carrier access follow
Section~\ref{sec:model} and Appendix~\ref{app:data}.
All principal models have width 128, 258,305 parameters, and
validation-price-based checkpoint selection.
DeepONet is the unnormalized direct-price baseline; FI uses RA/FAL;
PI adds data, PDE, terminal, and boundary losses without RA/FAL.
PI's baseline collocation omits the shortest maturities and has no
near-strike enrichment (Appendix~\ref{app:training}).

\paragraph{Metrics.}
Price $\RelL$ pools the test grid, including maturity; price p95
is the 95th percentile of absolute error near the strike.
Near-strike means $|S/K-1|\le0.05$; ATM-short further restricts
$0<\tau\le0.05$.
FD Delta/Gamma p95 use nonterminal near-strike points on the
24-surface higher-resolution L3 subset.
The common dimensionless FD PDE-residual RMS uses interior points,
with terminal prices entering its last time stencil.
Refining the reference preserves the FI--PI ordering, although
shortest-maturity Gamma remains resolution-sensitive
(Appendices~\ref{app:metrics}, \ref{app:verification}).

\subsection{Main results}
\label{sec:main_results}

FI-DeepONet achieves six- to eightfold lower mean global price
error than the vanilla and PI baselines, with lower FD Greek
p95 errors (Table~\ref{tab:classical}A; Figure~\ref{fig:classical}a).
Despite price-only supervision, it also has the lowest common FD
PDE-residual RMS.
Lower residuals alone do not predict better prices: PI improves
this diagnostic over DeepONet but worsens price error.
The FI--PI gap persists under the predeclared collocation and
boundary-target sensitivity studies
(Appendices~\ref{app:pifairness}, \ref{app:piboundary}). Figures~\ref{fig:classical}c--e show the corresponding
shortest-maturity error profiles and local error ratios,
complementing these pooled metrics.

\begin{figure}[!htbp]
\centering
\includegraphics[width=\linewidth]{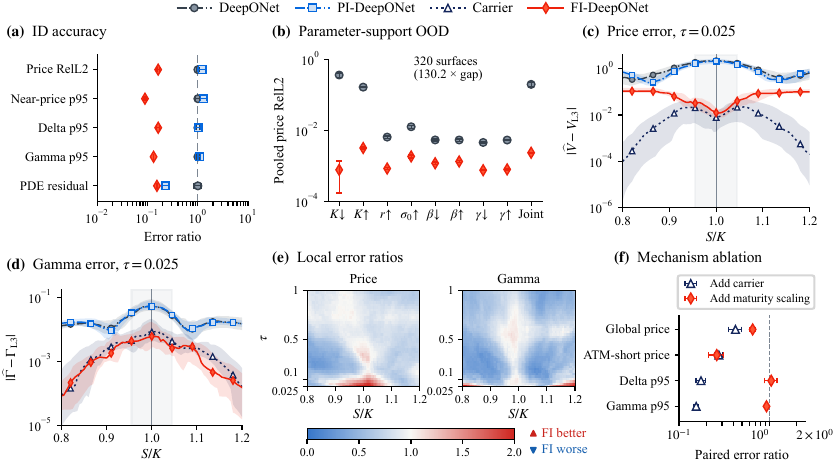}
\caption{Classical evaluation before PP.
(a) ID error ratios to DeepONet. (b) Complete parameter-support OOD set.
(c,d) Price/Gamma errors at $\tau=0.025$: medians and interquartile
ranges over 24 L3 surfaces and five seeds.
(e) Local error ratios; red favors FI.
(f) Pre-FAL paired carrier (B0$\to$A0) and rescaling (A0$\to$A1) effects.
Protocols: Appendix~\ref{app:metrics}.}
\label{fig:classical}
\end{figure}

\begin{table}[!t]
\caption{
Principal pre-PP results.
A: ID (32 surfaces); B: component ablation;
C: parameter-shift OOD (320 surfaces).
Five-seed means are shown unless deterministic; SDs are reported in
Appendices~\ref{app:attribution} and~\ref{app:ood}.
Bold: lowest mean.
}
\label{tab:classical}
\centering
\scriptsize
\setlength{\tabcolsep}{1.2pt}
\renewcommand{\arraystretch}{0.90}

\begin{minipage}[t]{0.56\linewidth}
\vspace{0pt}
\begin{tabular*}{\linewidth}{@{\extracolsep{\fill}}lrrr@{}}
\toprule
\multicolumn{4}{@{}l}{\textbf{A. In-distribution (32 surfaces)}}\\
Metric & DeepONet & PI & FI\\
\midrule
Price ($\times10^{-3}$)
& 5.99
& 7.24
& \textbf{0.965}\\
Near p95
& 1.80
& 2.33
& \textbf{0.158}\\
Delta p95
& 0.152
& 0.155
& \textbf{0.0245}\\
Gamma p95
& 0.0332
& 0.0355
& \textbf{0.00434}\\
PDE RMS
& 0.0461
& 0.0105
& \textbf{0.00703}\\
\bottomrule
\end{tabular*}
\end{minipage}
\hfill
\begin{minipage}[t]{0.42\linewidth}
\vspace{0pt}
\begin{tabular*}{\linewidth}{@{\extracolsep{\fill}}lrr@{}}
\toprule
\multicolumn{3}{@{}l}{\textbf{B. Component ablation}}\\
Configuration & Price ($\times10^{-3}$) & Near p95\\
\midrule
Normalized
& 7.6
& 2.26\\
+ soft penalties
& 7.5
& 2.21\\
+ FAL
& 2.4
& 0.528\\
FI (RA+FAL)
& \textbf{0.965}
& \textbf{0.158}\\
Carrier only
& 6.89
& 0.393\\
\bottomrule
\end{tabular*}
\end{minipage}

\vspace{1pt}

\begin{tabular*}{\linewidth}{@{\extracolsep{\fill}}lrr@{}}
\toprule
\multicolumn{3}{@{}l}{\textbf{C. Parameter-shift OOD (320 surfaces)}}\\
& DeepONet & FI\\
\midrule
Price ($\times10^{-3}$)
& 197
& \textbf{1.52}\\
\bottomrule
\end{tabular*}

\end{table}

\subsection{Ablation study}
\label{sec:ablation}

\paragraph{Carrier and rescaling.}
FAL improves Normalized DeepONet, and full FI further improves both
price metrics (Table~\ref{tab:classical}B).
A five-seed, 7,000-update matched study compares B0 (direct-price
control), A0 (carrier without rescaling), and A1 (full RA).
Before FAL, adding the carrier lowers global price and near-strike
Delta/Gamma errors in every paired seed.
Linear rescaling then improves global and ATM-short price accuracy,
cutting the latter error by about three quarters, but slightly
worsens near-strike Delta p95
(Figure~\ref{fig:classical}f; Appendix~\ref{app:mechanism_ablation}).

\paragraph{Exponent and representation controls.}
For the rescaling exponent $\alpha$ in the generalized
reconstruction
$z_{\theta,\alpha}=C^{\rm ref}+g_\alpha(\chi)R_\theta$,
with $g_0(\chi)=1$ and $g_\alpha(\chi)=\chi^\alpha$ for
$\alpha>0$, we test
$\alpha\in\{0,\frac12,1,\frac32\}$.
The lowest mean post-FAL Gamma p95 occurs at $\alpha=1$,
while price error favors $\alpha=\frac12$.
Thus asymptotic criticality does not imply metric-wise training
optimality.
A separate three-seed study favors RA over the tested coordinate
and sampling controls, but global low-rank reconstruction is worse
for both correction targets than for full prices.
The evidence supports local cancellation of the near-expiry singular
structure; rank diagnostics show no reduction in global linear
complexity (Appendices~\ref{app:mechanism_ablation}, \ref{app:attribution}).

\FloatBarrier
\subsection{Effect of PDE loss and limited data}
\label{sec:physics}
The five-seed normalized-RA factorial separates the PDE-loss intervention from representation.
Without FAL, adding the PDE loss lowers AD residual RMS and both Greek MAEs in every seed, but improves price error in only three.
With FAL, it lowers residuals while worsening price error in all five
seeds: optimizing the PDE loss and price accuracy need not improve together.
Separate limited-data studies show benefits from PDE loss
at 25\% data and from FAL at every tested fraction;
financial penalties add modest gains to FAL at the smallest fractions.
These are development results, not changes to the final price-only
FI objective (Appendix~\ref{app:objectives}).

\subsection{Out-of-distribution generalization}
\label{sec:ood}
Across all 320 parameter-shift surfaces, FI's pooled price error is
about two orders of magnitude below DeepONet's by the ratio of
five-seed means (Table~\ref{tab:classical}C; Figure~\ref{fig:classical}b).
This tests the complete predictor within the declared functional family.
The historical OOD output convention precludes a same-protocol
ID/OOD degradation ratio (Appendix~\ref{app:ood}).

To test beyond this coefficient family, we use carrier-null
off-family perturbations that vanish at the strike, leaving
the carrier unchanged.
The selected FI checkpoint responds but does not accurately reproduce
the reference directional price response or its spot derivatives,
despite agreement between AD and FD of its own predictions.
A separate three-seed multi-family study improves these responses
with derivative supervision at a cost in price accuracy
(Appendix~\ref{app:functional}).
\FloatBarrier
\section{Quantum-Compatible Realizations and Market Evaluation}
\label{sec:downstream}

We study quantum compatibility and market-derived inputs as downstream
evaluations of the trained FI operator, with the RA/FAL construction fixed.

\subsection{Quantum-Compatible Realizations}
\label{sec:quantumrealizations}
\paragraph{Setup.}
Unlike Quantum DeepONet \citep{r15}, we start from a pretrained dense
FI parent.
QO exactly factors selected maps into orthogonal rotations and a
classical signed core, without adaptation.
QDO approximates eligible maps by $\mathbf W\approx DQ_{\rm circ}$,
with diagonal $D$ and restricted orthogonal
$Q_{\rm circ}\in\mathcal Q_{\rm circ}$.
Full QDO adapts all trainable parameters using pre-FAL price regression;
a fixed composite score selects checkpoints and coverage.
The 410/102 fitting/selection split comes entirely from the original
synthetic training set, and all candidates share one parent.
Cores, scalings, biases, nonlinearities, carrier, and FAL remain
classical (Appendices~\ref{app:qo}--\ref{app:qselection}).

\vspace{-10pt}
\begin{table}[!htbp]
\caption{
Selected C94 realization versus its fixed FI parent.
A: MAC coverage and update scope;
B: separate 32-surface AD synthetic evaluation;
C: 2,586 held-out-IV quotes.
Single selected instance, not seed averages; bold: lower error.
}
\label{tab:quantum}
\centering
\small

\begin{tabularx}{\linewidth}{Yrr}
\toprule
\multicolumn{3}{l}{\textbf{A. Realization scope}}\\
\midrule
Property & QO & QDO (C94)\\
Dense-map MAC coverage & $49.6\%$ & $92.8\%$\\
Parameters updated & None & All trainable\\

\midrule
\multicolumn{3}{l}{\textbf{B. Paired synthetic accuracy}}\\
\midrule
Metric & Parent FI-DeepONet & QDO (C94)\\
Global price $\RelL$ ($\times10^{-3}$)
& $\mathbf{0.837}$ & $1.77$\\
Near-strike price p95 error
& $\mathbf{0.1315}$ & $0.1324$\\
Global AD Delta p95 error
& $\mathbf{0.0127}$ & $0.0148$\\
Global AD Gamma p95 error
& $0.00244$ & $\mathbf{0.00224}$\\

\midrule
\multicolumn{3}{l}{\textbf{C. Market quotes with held-out implied volatility}}\\
\midrule
Metric & Parent FI-DeepONet & QDO (C94)\\
Market RMSE (quoted price units)
& $19.54$ & $\mathbf{14.92}$\\
Operator $\RelL$ ($\times10^{-3}$)
& $13.18$ & $\mathbf{5.48}$\\
\bottomrule
\end{tabularx}
\end{table}
\vspace{5pt}

\paragraph{Fidelity and adaptation.}
QO reproduces parent predictions to numerical precision in idealized
evaluation.
At the selected C94 coverage, QDO roughly doubles global price error;
near-strike price p95 changes little, Delta p95 worsens, and Gamma
p95 improves (Table~\ref{tab:quantum}A,B).
The selected C94 instance fails the fixed ID/OOD accuracy-preservation criteria (Appendix~D.5).
These synthetic evaluations use AD Greeks under a protocol separate from the principal FD study.
In the three-seed matched C94 study, updating only substituted maps
and their biases achieves almost all frozen-to-full error recovery
(Figure~\ref{fig:quantum}a).
Across paired C49--C99 configurations, mean price errors remain near
$1.8\times10^{-3}$ without a monotone coverage trend
(Figure~\ref{fig:quantum}b; Appendix~\ref{app:qfrontier}).

\paragraph{Resource scope.}
QDO's lower listed errors relative to historical classical structured
approximations of comparable computational cost are descriptive because
their adaptation protocols are not matched
(Appendix~\ref{app:qfidelity}); the comparison does not isolate the effect
of the DQ parameterization.
Computational cost here follows the arithmetic/storage accounting convention
for such comparisons, rather than measured runtime.
Whole-dense-map coverage $\omega$ counts reference dense-layer
multiply--accumulate operations (MACs), while eligible-map coverage $\rho$
uses a different denominator; both are static analytical coverage measures,
and the saved partial implementation computes both dense and structured outputs.
Figure~\ref{fig:quantum}c reports the resulting conditional cost scenarios
analytically, without measuring hardware speedup.
A complete end-to-end assessment would further require accounting for
preparation, measurement, precision, and the retained classical workload
\citep{r13} (Appendix~\ref{app:qresources}).

\begin{figure}[!htbp]
\centering
\includegraphics[width=\linewidth]{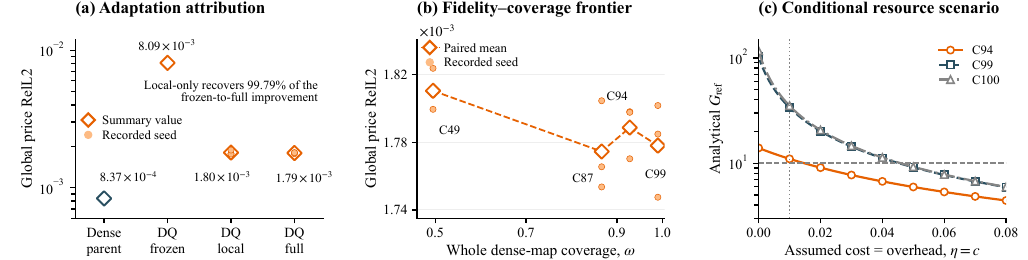}
\caption{
Quantum-compatible evaluation.
(a) C94 adaptation attribution over three paired seeds.
(b) Paired C49--C99 fidelity--coverage frontier; zoomed vertical scale.
(c) Conditional
$G_{\rm ref}=[1-\omega(1-\eta)+c]^{-1}$ with $\eta=c$.
No hardware speedup is measured.
}
\label{fig:quantum}
\end{figure}

\FloatBarrier

\subsection{Market Evaluation with Constant-Volatility Inputs}
\label{sec:marketinputs}

\paragraph{Setup.}
We evaluate Nasdaq-100 (NDX) call quotes from 13 May to 7 August 2026.
Implied volatility (IV) is a constant-volatility parameter inferred
from a quote under a pricing model.
Own-IV uses the target's supplied IV; held-out-IV reconstructs it
from other strikes on the same date and expiry, excluding the target
from the fit.
Inputs set $\beta=\gamma=0$ and use constant-volatility sensors.
Under the effective-spot mapping in Appendix~E,
$\beta=\gamma=0$ makes the supplied volatility field constant in
spot and time, so the same-input exact solution is the
Black--Scholes carrier:
$V=C^{\rm ref}$ and $W=R^*=0$ for $\tau>0$.
The fixed-$T=1$ network retains actual remaining maturity
through the effective-spot and contract-scaling map in
Appendix~\ref{app:market}.
Operator error uses the same-input Black--Scholes reference;
market error uses the bid--offer midpoint.
Weights and normalizers remain fixed, and all predictions precede PP.

\paragraph{Results.}
FI improves over both normalized neural baselines in both IV settings,
but not the same-input Black--Scholes reference
(Table~\ref{tab:market}).
On held-out IV, its five-seed mean market relative RMSE is about
thirteen times lower than Normalized DeepONet's and three times
lower than the FAL-only baseline's.
Separate diagnostics place FI's median market error at about
one half-spread, with larger tail errors.
Figure~\ref{fig:market}a visualizes market quotes against the
frozen FI surface on one held-out date, while
operator-error checks remain stable across the tested IV
reconstructions and late dates (Figure~\ref{fig:market}b).

Despite the synthetic fidelity loss above, across C49--C99,
all twelve paired coverage--adaptation--seed comparisons
reduce market RMSE against the same fixed FI parent.
Full-sample geometric means indicate roughly $20$--$25\%$ reductions; all four coverages also improve on the sealed late-date subset.
The reported date/chain bootstrap intervals lie below one conditional
on the trained realizations
(Figure~\ref{fig:market}c; Appendix~\ref{app:marketpaired}).
The error decomposition supports sample-dependent price-error alignment
as the explanation for these reductions, without evidence of a
quantum-specific regularization effect.
The study evaluates constant-volatility inputs; local-volatility
calibration is outside its scope, and incomplete upstream rate/dividend
and contract records limit source-to-model reproduction
(Appendix~\ref{app:marketprovenance}).
\begin{figure}[!htbp]
\centering
\includegraphics[width=\linewidth]{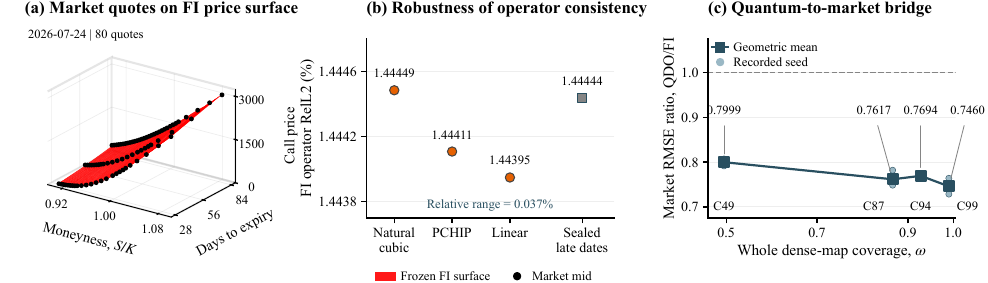}
\caption{
Frozen-network market evaluation before PP.
(a) FI price surface on 2026-07-24 with 80 market mid-quotes.
(b) Same-input Black--Scholes operator error under IV reconstruction
and late-date checks; zoomed vertical scale.
(c) Paired C49--C99 QDO/FI market-RMSE ratios;
squares show three-seed geometric means and circles recorded seeds.
Protocols: Appendix~E.
}
\label{fig:market}
\end{figure}
\FloatBarrier
\enlargethispage{1.2\baselineskip}
\vspace{-16pt}
\section{Limitations and Conclusion}
\label{sec:conclusion}
\vspace{-2pt}
FI-DeepONet shows that known structure of a parabolic pricing problem can
enter an operator network through its representation and output map, rather
than through additional loss terms. A carrier matched to the strike-line integrated variance captures the
leading near-expiry singular structure. The network learns a time-rescaled correction whose exact target is
bounded under the stated assumptions; FAL enforces pointwise price bounds.
Trained on prices alone, the model lowers global price error six- to
eightfold relative to vanilla and physics-informed DeepONet, and it also
reduces errors in Delta and Gamma, the first and second spot derivatives
of the price, and in the PDE residual. Ablations attribute these gains to the
carrier and the rescaling, with trade-offs that depend on the metric.
Selected layers can be compiled exactly into orthogonal circuits, or
approximated within a restricted family and then adapted. The frozen model
also stays consistent on market-derived constant-volatility inputs.
These conclusions are limited to one low-dimensional coefficient family,
fixed sensors, and a fixed horizon; off-family tests leave general
functional extrapolation unresolved.
Theory and matched ablations support local cancellation of the leading
near-expiry singular structure, while rank diagnostics show no reduction
in global linear complexity.
Fixed-width FAL retains near-terminal bias and a spot kink; PP repairs
only a discrete grid. Neither step certifies an arbitrage-free continuous
surface.
Quantum results are classical or idealized, so hardware speedup remains
unmeasured.
Market tests concern constant-volatility inputs; same-input
Black--Scholes remains more accurate, and source-to-model provenance is
incomplete.
Future work includes terminal-scaled FAL, broader coefficient families,
and end-to-end resource measurements.
\label{maintext-end}
\clearpage
\subsection*{AI use statement}
Generative AI tools were used to assist literature and source discovery;
research ideation and experimental-methodology refinement; implementation,
debugging, and execution of numerical analysis and evaluation workflows;
mathematical derivation, proof drafting, and proof checking; analysis and
interpretation of results; manuscript organization, section drafting, language
editing, and consistency checks; and development of the LaTeX code for
scientific figures. This includes assistance with mathematical claims and their
proofs, rather than language polishing alone. The synthetic numerical datasets
and reference labels used in the experiments were generated by analytical
formulas or finite-difference solvers under specified parameterized protocols;
generative AI was not used to synthesize observations, numerical labels, or
replacement benchmark data. AI tools also assisted preparation and auditing of
the submission. AI-assisted calculations and consistency checks are not
independent validation of a proof or an empirical claim. The authors retain
responsibility for the final text, mathematical arguments, experimental design,
numerical results, citations, code, and data-use permissions, including all
AI-assisted content.

\subsection*{Reproducibility statement}
Appendix~\ref{app:theory} states the assumptions, proofs, notation, and
output-stage qualifications. Appendix~\ref{app:protocols} specifies the fixed
data, inputs, architecture, optimization, checkpoint selection, masks, and
metric units. Appendix~\ref{app:results} reports additional numerical results,
reference checks, ablations, and sensitivity studies.
Appendix~\ref{app:quantum} distinguishes exact selected-map compilation from
adapted diagonal--orthogonal approximation and conditional resource accounting.
Appendix~\ref{app:market} defines the market-input mapping, target-excluded IV
reconstruction, evaluation populations, and remaining provenance limitations.
Appendix~\ref{app:context} summarizes the replay scope across the
historical synthetic records, the separate five-seed E0 refit, and the
complete-population OOD replay.
Reproducibility claims concern fixed realized data, documented selected
states, and explicitly identified refits and replays; a fresh retraining
of every historical principal run has not been performed.
Some split-generation, circuit-topology, and market-feature provenance
remains incomplete, and licensed quote-level market records cannot be
redistributed.

\subsection*{Ethics statement}
The models are research approximations and have not been validated for
deployment in trading or risk management. Pointwise price bounds, discrete
grid repairs, and finite numerical checks do not establish reliable hedging
or continuous-surface absence of arbitrage. The market study evaluates
constant-volatility inputs and does not establish local-volatility calibration
or superiority to the same-input Black--Scholes reference. The quantum-compatible
studies do not demonstrate measured hardware speedup. The market data are
licensed; any artifact release must respect applicable access and redistribution
permissions. The manuscript and figure-source package does not confer rights
to redistribute the underlying market records.
\bibliographystyle{plainnat}
\bibliography{references}
\clearpage
\appendix
\FloatBarrier
\section{Analytical details and proofs}
\label{app:theory}
\usepgfplotslibrary{groupplots}
\definecolor{appGray}{HTML}{536273}
\definecolor{appBlue}{HTML}{0066CC}
\definecolor{appNavy}{HTML}{24386B}
\definecolor{appRed}{HTML}{F02B18}
\definecolor{appOrange}{HTML}{D87A18}
\pgfplotsset{appaxis/.style={
  font=\footnotesize,axis lines=left,tick align=outside,
  axis line style={black,line width=0.5pt},
  tick style={black,line width=0.4pt},
  label style={font=\footnotesize},
  title style={font=\footnotesize,align=left,at={(0,1.08)},anchor=south west},
  legend style={font=\scriptsize,draw=gray!30,fill=white,inner sep=3pt},
  legend cell align=left,
  every axis plot/.append style={line width=1pt,mark size=2pt},
  scaled ticks=false,clip=true,
  grid=major,grid style={gray!15,line width=0.3pt},
  tick label style={font=\scriptsize},
},appsmall/.style={appaxis,width=3.05cm,height=3.6cm,scale only axis},
apphalf/.style={appaxis,width=5.25cm,height=4cm,scale only axis}}
\tikzset{appbox/.style={draw=appNavy!70,fill=appNavy!3,rounded corners=2pt,
  line width=0.6pt,align=center,inner sep=5pt,font=\footnotesize},
appflow/.style={-{Latex[length=2mm]},line width=0.75pt,draw=appGray},
appnote/.style={font=\scriptsize,align=center,text=appGray}}

\raggedbottom
\makeatletter
\setlength{\@fptop}{0pt}
\setlength{\@fpsep}{12pt}
\setlength{\@fpbot}{0pt plus 1fil}
\makeatother
The analysis separates properties of the pricing operator from properties of the implemented predictor. Appendices~\ref{app:scale} and~\ref{app:fal} establish scale symmetry of the exact pricing operator and pointwise admissibility of the scalar output map; Appendices~\ref{app:ra} and~\ref{app:stages} analyze the correction and its transfer through the output stages. Appendix~\ref{app:scaledfal} gives a terminal-scaled smoothing construction whose empirical performance is not evaluated here.
\setcounter{subsection}{-1}
\subsection{Financial terminology}
\label{app:finance}

The terms below follow the pricing conventions used in this paper
\citep{r01,r02,r03}.
They distinguish the financial contract, its PDE inputs,
and the market-data quantities.

\begingroup
\footnotesize
\setlength{\tabcolsep}{4pt}
\renewcommand{\arraystretch}{1.12}

\begin{tabularx}{\linewidth}{@{}>{\raggedright\arraybackslash}p{0.25\linewidth}Y@{}}
\toprule
Term & Meaning in this paper \\
\midrule

Underlying / spot $S$
&
The asset on which the option is written;
$S$ is its current price.
\\

European call
&
A contract giving the holder the right, but not the obligation,
to buy the asset at price $K$ only at expiry $T$.
\\

Strike $K$
&
The purchase price fixed by the option contract.
\\

Maturity $T$ / remaining time $\tau$
&
$T$ is the expiry time;
$\tau=T-t$ is the time left from valuation time $t$.
\\

Payoff
&
The cash value at expiry:
$(S-K)^+=\max(S-K,0)$ for the call considered here.
\\

Risk-free rate $r$
&
The constant continuously compounded rate used for discounting;
the discounted strike is $Ke^{-r\tau}$.
\\

Local volatility $\sigma$ / variance $a$
&
$\sigma(S,t)$ sets the instantaneous scale of proportional price
fluctuations; $a(S,t)=\sigma^2(S,t)$.
\\

Strike-line integrated variance $A_0$
&
The time integral of the input local variance evaluated at $S=K$,
as in \eqref{eq:strikeline}; not the integral of $\sigma$.
\\

Delta $\Delta$ / Gamma $\Gamma$
&
The first and second price derivatives in $S$:
$\partial_S V$ and $\partial_{SS}V$.
They describe sensitivity and curvature.
\\

Hedging
&
Taking offsetting positions to reduce financial risk;
accurate price sensitivities are relevant to this task.
\\

Moneyness / at the money (ATM)
&
The synthetic benchmark uses $m=S/K$, with ATM near $m=1$.
Market tests use their stated forward-based convention.
\\

No-arbitrage bounds
&
Necessary pointwise price restrictions \eqref{eq:bounds};
satisfying them alone does not ensure cross-spot or
cross-contract consistency.
\\

Implied volatility (IV)
&
The constant volatility that matches an observed option price
in a specified pricing formula; it is not a local-volatility field.
\\

Calibration
&
Inferring pricing-model parameters or a volatility function
from market quotes; distinct from evaluating a trained
pricing surrogate.
\\

Dividend yield $q$ / forward $F$
&
The market-input convention uses $F=Se^{(r-q)\tau}$ and
effective spot $Se^{-q\tau}$;
see Appendix~\ref{app:market}.
\\

Bid, offer, mid, spread
&
Bid is the price quoted by a buyer; offer is the price quoted
by a seller. Mid is their average, and spread is offer minus bid.
\\

\bottomrule
\end{tabularx}

\endgroup
\paragraph{Notation shared with the main text.}
We reserve $a=\sigma^2$ for local variance, $x=\log(S/K)$ for
log-moneyness, $m=S/K$ for moneyness, $\bar t=t/T$ for normalized
elapsed time, and $\chi=\tau/T=1-\bar t$ for remaining-life fraction.
$W=V-\Cref$ is the exact price correction; $R^*=W/\chi$ and
$R_\theta=Kr_\theta$ have price units, while
$r_\theta=\mu+s_yf_\theta$ is dimensionless.
$z_\theta$ denotes the reconstructed pre-FAL price and
$\pP$ the terminal-corrected FI price before optional PP.
The Greek errors refer to physical-spot derivatives.
Bold $\mathbf W$ in Appendix~\ref{app:quantum} denotes a weight matrix,
not the scalar correction $W$.
The exponent $\alpha$ always denotes maturity rescaling; coverage uses
$\rho$ and $\omega$.
The main-text $P_a$ and appendix $\Pa$ denote the same pricing operator
in \eqref{eq:pde}.

\FloatBarrier
\subsection{Pricing, normalization, and exact scale symmetry}
\label{app:scale}
For a zero-dividend asset, the risk-neutral dynamics are
$dS_u=rS_u\,du+S_u\sigma(S_u,u)\,dW_u$, with $a=\sigma^2$.
Risk-neutral valuation is
\begin{equation}
V_a^K(S,t)=e^{-r(T-t)}\E^{\mathbb Q}\big[(S_T-K)^+\mid S_t=S\big].
\label{eq:riskneutral}
\end{equation}
The backward equation is $V_t+\tfrac12a(S,t)S^2V_{SS}+rSV_S-rV=0$.
With $v(S,\tau)=V(S,T-\tau)$, define
\begin{equation}
\Pa=\partial_\tau-\Aa,\qquad
\Aa=\tfrac12a(S,T-\tau)S^2\partial_{SS}+rS\partial_S-r.
\label{eq:pde}
\end{equation}
Then $\Pa v=0$ and $v(S,0)=\relu{S-K}$.
Whenever $\Pa$ acts on $V$, $\Cref$, or their corrections, these functions are understood in time-reversed coordinates.
Delta and Gamma are $\Delta=V_S$ and $\Gamma=V_{SS}$.
Classical Gamma at the terminal strike is not an evaluation target.
Cross-strike and calendar-spread restrictions require consistent contract families; they are distinct from the pointwise interval in \eqref{eq:bounds}.
The operator-learning target is $(a,K,T,r)\mapsto V_a^K(\cdot,\cdot)$, with function-valued $a$.
The finite-dimensional sampled family in Appendix~\ref{app:data} is narrower than this mathematical input class.
These pricing formulations follow the classical risk-neutral and local-volatility constructions \citep{r01,r02,r03,r04}.

For $m=S/K$, $\bar t=t/T$, $c(m,\bar t)=V(Km,T\bar t)/K$, the equation and sensitivities are
\begin{equation}
c_{\bar t}+T\left[\tfrac12a(Km,T\bar t)m^2c_{mm}+rmc_m-rc\right]=0,
\qquad \Delta=c_m,\quad \Gamma=K^{-1}c_{mm}.
\label{eq:normalpde}
\end{equation}
The variable $\chi=\tau/T=1-\bar t$ is remaining, not elapsed, normalized time. All reported synthetic training uses $T=1$; normalization alone is not an unrestricted maturity-invariance result.

\begin{proposition}[Joint spot--strike--field symmetry]
For $\lambda>0$, set $a_\lambda(S,u)=a(S/\lambda,u)$. Then
\begin{equation}
V_{a_\lambda}^{\lambda K}(\lambda S,t)=\lambda V_a^K(S,t),\quad
\Delta_{a_\lambda}^{\lambda K}(\lambda S,t)=\Delta_a^K(S,t),\quad
\Gamma_{a_\lambda}^{\lambda K}(\lambda S,t)=\lambda^{-1}\Gamma_a^K(S,t),
\label{eq:scale}
\end{equation}
where the derivatives exist.
\end{proposition}
\begin{proof}
If $S_u$ solves the original SDE, then $\widetilde S_u=\lambda S_u$ solves the transformed model because $a_\lambda(\widetilde S_u,u)=a(S_u,u)$. Its payoff is $(\widetilde S_T-\lambda K)^+=\lambda(S_T-K)^+$. Discounted expectation and uniqueness give the price identity. Differentiating with respect to physical spot gives the Delta and Gamma identities.
\end{proof}
A trained network inherits this identity only if its inputs and normalization statistics transform compatibly. The implementation retains absolute payoff samples and the unnormalized $K$, with fixed training standardization. If $\operatorname{Std}(y)=(y-\mu_{\mathrm{tr}})/s_{\mathrm{tr}}$, then
\begin{equation}
\operatorname{Std}(\lambda y)-\operatorname{Std}(y)=(\lambda-1)y/s_{\mathrm{tr}}.
\end{equation}
Thus the implemented model is scale-aware, not exactly equivariant. Equation~\eqref{eq:scale} is an operator property, not a learned OOD theorem.

\FloatBarrier
\subsection{Admissibility, stability, and units}
\label{app:fal}
For $\delta>0$ and $\ell<u$, define the softplus map and the Financial Admissibility Layer using stable softplus evaluation:
\begin{equation}
\spdelta(z)=\delta\log(1+e^{z/\delta}),\qquad
\FAL_\delta(z;\ell,u)=\ell+\spdelta(z-\ell)-\spdelta(z-u).
\label{eq:fal}
\end{equation}
We use dimensionless smoothing $\delta_v=0.002$, corresponding to $K\delta_v$ in price units. The complete FI output is
\begin{equation}
\pP(S,t)=
\begin{cases}
0,&S=0,\\
(S-K)^+,&t=T,\\
\FAL_{K\delta_v}(z_\theta;L,U),&S>0,\ t<T.
\end{cases}
\label{eq:fi}
\end{equation}
This is the post-FAL, terminal-corrected output before PP.
Historical continuation training applied ordinary FAL even at terminal nodes; \eqref{eq:fi} was evaluated later as a parameter-free analytic assignment using unchanged selected weights, no held-out outcomes, and no retraining (Appendix~\ref{app:training}).
Appendix~\ref{app:scaledfal} studies a time-scaled smoothing alternative, but it was not trained here.
\begin{proposition}[Scalar FAL properties]
For $\ell<u$, finite $z$, and $\delta>0$, the map in \eqref{eq:fal} satisfies
\begin{gather}
\ell<\FAL_\delta(z;\ell,u)<u,\qquad 0<\partial_z\FAL_\delta<1,\label{eq:falrange}\\
|\FAL_\delta(z_1;\ell,u)-\FAL_\delta(z_2;\ell,u)|\le|z_1-z_2|,\label{eq:fallip}\\
|\FAL_\delta(z;\ell,u)-P_{[\ell,u]}(z)|\le\delta\log2,\qquad
\FAL_{\lambda\delta}(\lambda z;\lambda\ell,\lambda u)=\lambda\FAL_\delta(z;\ell,u).
\label{eq:falproject}
\end{gather}
For $v\in[\ell,u]$, $|\FAL_\delta(z;\ell,u)-v|\le|z-v|+\delta\log2$.
\end{proposition}
\begin{proof}
Put $A=\sigm((z-\ell)/\delta)$, $B=\sigm((z-u)/\delta)$, where $\sigm(s)=(1+e^{-s})^{-1}$. Since $0<B<A<1$,
\begin{equation}
\partial_z\FAL=A-B,\qquad \partial_\ell\FAL=1-A,\qquad \partial_u\FAL=B.
\label{eq:falgradient}
\end{equation}
These derivatives are positive and sum to one. The limits as $z\to-\infty,+\infty$ are $\ell,u$, respectively, so strict monotonicity proves the interval result. The mean-value theorem proves scalar nonexpansiveness; integrating the joint gradient also gives one-Lipschitz continuity in the maximum norm along admissible triples. Since
$P_{[\ell,u]}(z)=\ell+(z-\ell)^+-(z-u)^+$ and $0\le\spdelta(q)-q^+\le\delta\log2$, their difference is the difference of two numbers in $[0,\delta\log2]$, not the sum of their absolute upper bounds. This proves the single-$\delta\log2$ constant. Projection nonexpansiveness gives the error bound relative to $v$. Finally, $\operatorname{sp}_{\lambda\delta}(\lambda q)=\lambda\spdelta(q)$ proves homogeneity.
\end{proof}
As a function of the independent scalar arguments $(z,\ell,u)$, FAL is a smooth approximation to projection, not a retraction that fixes every interior price. This scalar smoothness does not imply global smoothness of the composed price as a function of spot when $\ell=L(S,\tau)$ is kinked. The dimensionless parameter $\delta_v$ acting on $V/K$ corresponds to $K\delta_v$ in price units. At maturity, the payoff equals $L$ and
\begin{equation}
\FAL_\delta(L;L,U)-L=\delta\log2-\delta\log\!\left(1+e^{-(U-L)/\delta}\right)>0\quad(L<U).
\label{eq:terminalbias}
\end{equation}
The explicit terminal assignment in \eqref{eq:fi} does not remove this nonterminal limit bias for fixed $\delta$. At $S=0$, $L=U=0$, and the output is assigned zero directly rather than invoking the strict-interval proposition.

On a margin region $z-\ell\ge m_0>0$, $u-z\ge m_0$,
\begin{equation}
|\FAL_\delta(z)-z|\le2\delta e^{-m_0/\delta},\quad
1-\partial_z\FAL\le2e^{-m_0/\delta},\quad
\partial_\ell\FAL,\partial_u\FAL\le e^{-m_0/\delta}.
\label{eq:falmargin}
\end{equation}
For $q=(z,\ell,u)$, its Hessian is
\begin{equation}
\nabla_q^2\FAL=h_\ell v_\ell v_\ell^\top-h_u v_u v_u^\top,\quad
v_\ell=(1,-1,0)^\top,\quad v_u=(1,0,-1)^\top,
\label{eq:falhessian}
\end{equation}
where $h_\ell=A(1-A)/\delta$, $h_u=B(1-B)/\delta$, each bounded by $e^{-m_0/\delta}/\delta$ on that region. Small scalar distortion therefore need not imply small derivative distortion near an active bound.

Figure~\ref{fig:app_fal} isolates the scalar smoothing effect and its
terminal implication; neither panel is a trained-model experiment.
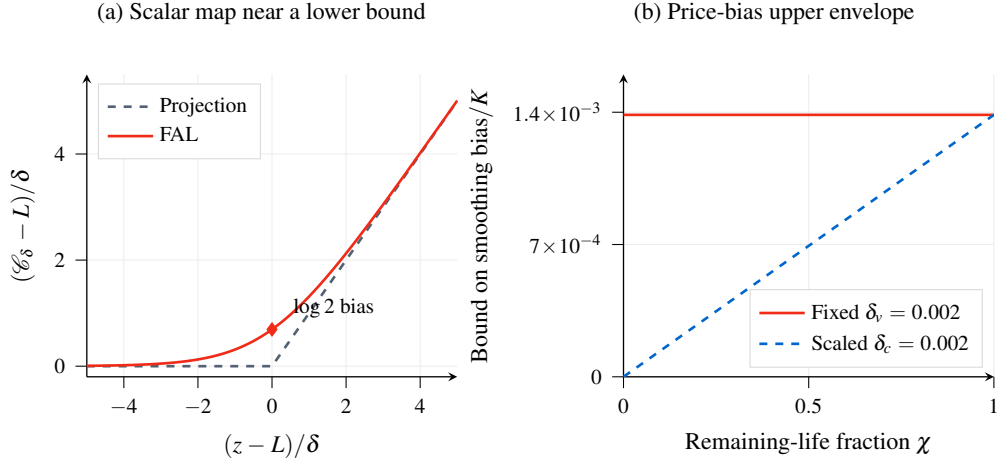
\begin{figure}[!htbp]
\centering
\begingroup\fontfamily{ptm}\selectfont\mathversion{figuretimes}

\begin{tikzpicture}
\begin{groupplot}[apphalf,width=4.9cm,group style={group size=2 by 1,horizontal sep=2.2cm}]
\nextgroupplot[title={(a) Scalar map near a lower bound},xlabel={$(z-L)/\delta$},ylabel={$(\mathcal C_\delta-L)/\delta$},
 xmin=-5,xmax=5,ymin=-.2,ymax=5.5,xtick={-4,-2,0,2,4},ytick={0,2,4},legend pos=north west]
\addplot[appGray,dashed,domain=-5:5,samples=101,no marks] {max(x,0)};\addlegendentry{Projection}
\addplot[appRed,domain=-5:5,samples=101,no marks] {max(x,0)+ln(1+exp(-abs(x)))-ln(1+exp(x-20))};\addlegendentry{FAL}
\addplot[only marks,appRed,mark=diamond*] coordinates {(0,.693147)};
\node[font=\scriptsize,anchor=west] at (axis cs:.3,1.1) {$\log 2$ bias};
\nextgroupplot[title={(b) Price-bias upper envelope},xlabel={Remaining-life fraction $\chi$},ylabel={Bound on smoothing bias$/K$},
 xmin=0,xmax=1,ymin=0,ymax=.0016,xtick={0,.5,1},ytick={0,.0007,.0014},
 yticklabels={$0$,$7\!\times\!10^{-4}$,$1.4\!\times\!10^{-3}$},legend pos=south east]
\addplot[appRed,no marks,domain=0:1] {.002*ln(2)};\addlegendentry{Fixed $\delta_v=0.002$}
\addplot[appBlue,dashed,no marks,domain=0:1] {.002*x*ln(2)};\addlegendentry{Scaled $\delta_c=0.002$}
\end{groupplot}
\end{tikzpicture}

\endgroup
\caption{Analytic FAL diagnostics. (a) The scalar map and hard projection on an
illustrative interval with $(U-L)/\delta=20$; the lower-bound bias approaches
$\delta\log2$. This local coordinate plot is not a price-versus-spot surface.
(b) The bounds $\delta_v\log2$ and $\chi\delta_c\log2$ on normalized price
smoothing error, with equal illustrative coefficients $\delta_c=\delta_v=0.002$.
The fixed-width curve concerns the nonterminal output: an exact assignment
at $\chi=0$ does not remove its one-sided limit bias.
Terminal-scaled FAL is the untested theoretical refinement of
Appendix~\ref{app:scaledfal}, not a new trained model.}
\label{fig:app_fal}
\end{figure}

\FloatBarrier
\subsection{Correction equation, cancellation, and differentiated profiles}
\label{app:ra}
With $A_0$ from \eqref{eq:strikeline}, the strike-matched carrier for $S>0$ and $A_0(\tau)>0$ is
\begin{equation}
\begin{gathered}
\Cref(S,t)=S\Phi(d_+)-Ke^{-r\tau}\Phi(d_-),\\
d_+=\frac{\log(S/K)+r\tau+A_0(\tau)/2}{\sqrt{A_0(\tau)}},\qquad
d_-=d_+-\sqrt{A_0(\tau)}.
\end{gathered}
\label{eq:carrier}
\end{equation}
Here $\Phi$ is the standard normal CDF.
The carrier exactly prices the spatially frozen, time-dependent variance $a(K,t)$.
Its limiting assignments are $\Cref(S,T)=\relu{S-K}$ and $\Cref(0,t)=0$.
\paragraph{Assumptions and correction equation.}
Set $x=\log(S/K)$ and $\bar a(x,\tau)=a(Ke^x,T-\tau)$.
Contract parameters range over a compact class with $K,T$ bounded away from zero.
For value estimates, assume a full-line extension with $0<a_-\le\bar a\le a_+<\infty$, uniform spatial Lipschitz continuity, and unique pricing/correction solutions with comparison in an at-most-exponential-growth class.
For differentiated profiles, additionally assume a smooth extension with $\bar a\in C_b^{3+\zeta,1+\zeta/2}$, $0<\zeta<1$.
A smooth neighborhood of the strike suffices by localization; clipping elsewhere need not be differentiable, as detailed below.
Constants below are uniform over the stated compact contract class.
The stronger differentiated assumptions concern the smooth extension, not every clipped production coefficient globally.

The exact correction $W=V-\Cref$ satisfies, wherever classical derivatives exist,
\begin{equation}
\begin{gathered}
\Pa W=f,\qquad W(S,T)=0,\\
f=\tfrac12[a(S,T-\tau)-a(K,T-\tau)]S^2\Cref_{SS}.
\end{gathered}
\label{eq:forcing}
\end{equation}

\paragraph{Proof of Proposition~\ref{prop:value}.}
The carrier solves the equation with $a(S,t)$ replaced by $a(K,t)$. Subtracting its forward equation from the full price equation gives \eqref{eq:forcing}. Set $c^{\mathrm{ref}}(x,\tau)=\Cref(Ke^x,T-\tau)/K$ and $w=W(Ke^x,T-\tau)/K$. Then
\begin{align}
(\partial_{xx}-\partial_x)c^{\mathrm{ref}}&=\frac{e^x\phi(d_+)}{\sqrt{A_0(\tau)}},\label{eq:carrierdensity}\\
\partial_\tau w&=\tfrac12\bar a\,w_{xx}+(r-\bar a/2)w_x-rw+\widetilde f,\nonumber\\
\widetilde f&=\tfrac12[\bar a(x,\tau)-\bar a(0,\tau)](\partial_{xx}-\partial_x)c^{\mathrm{ref}}.
\label{eq:logcorrection}
\end{align}
Here $\phi=\Phi'$. Uniform ellipticity gives $a_-\tau\le A_0(\tau)\le a_+\tau$. Completing the square in $e^x\phi(d_+)$, with $r\tau+A_0/2=O(\tau)$, yields
\begin{equation}
0\le(\partial_{xx}-\partial_x)c^{\mathrm{ref}}\le C\tau^{-1/2}e^{-c_0x^2/\tau}.
\end{equation}
Lipschitz continuity implies $|\widetilde f|\le C|x|\tau^{-1/2}e^{-c_0x^2/\tau}\le C_f$. Duhamel/comparison with zero initial correction gives
\begin{equation}
\norm{w(\cdot,\tau)}_\infty\le\int_0^\tau e^{|r|(\tau-s)}C_f\,ds\le C\tau.
\end{equation}
The full-line estimate concerns the declared extension. A bounded-domain analogue must also include differences in lateral boundary data. This proves the stated value bound.\hfill$\square$

\paragraph{Alternative coupling estimate.}
For initial spots in a fixed compact interval containing $K$, assume bounded spot-Lipschitz $\sigma$, a globally Lipschitz extension of $s\mapsto s\sigma(s,u)$, and finite required fourth moments, uniformly over the contract class. Drive the true process $X$ and strike-frozen process $Y$ from common initial $S$ with the same Brownian motion. The latter has discounted call value $\Cref$. Its geometric representation gives, for $q=2,4$,
\begin{equation}
\E\sup_{t\le v\le u}|Y_v-K|^q\le C_q\big(|S-K|^q+(u-t)^{q/2}\big).
\end{equation}
For $D=X-Y$, split the diffusion discrepancy as
\begin{align}
X_u\sigma(X_u,u)-Y_u\sigma(K,u)
={}&[X_u\sigma(X_u,u)-Y_u\sigma(Y_u,u)]\nonumber\\
&+Y_u[\sigma(Y_u,u)-\sigma(K,u)].
\end{align}
Stochastic integral estimates and Gronwall reduce the mean-square difference to the integral of $\E[Y_u^2|Y_u-K|^2]$. H\"older and the fourth-moment bounds give
$\E\sup_{t\le u\le T}|D_u|^2\le C(\tau|S-K|^2+\tau^2)$.
The payoff is one-Lipschitz, hence
\begin{equation}
|V(S,t)-\Cref(S,t)|\le C\big(\sqrt\tau|S-K|+\tau\big).
\label{eq:coupling}
\end{equation}
This is an $O(\tau)$ value estimate on a parabolic strike neighborhood; it does not permit derivative estimates by differentiation of its right-hand side.

\paragraph{Approximate strike-line integration.}
Let $\widetilde A_0$ be absolutely continuous with $a_-\tau\le\widetilde A_0(\tau)\le a_+\tau$ and $\widetilde a_K=\partial_\tau\widetilde A_0$ almost everywhere. Its carrier $\widetilde C^{\mathrm{ref}}$ satisfies
\begin{equation}
\Pa(V-\widetilde C^{\mathrm{ref}})=\tfrac12[a(S,T-\tau)-\widetilde a_K(\tau)]S^2\widetilde C^{\mathrm{ref}}_{SS}.
\end{equation}
Separate spatial mismatch from line error $\varepsilon_K(q)=|\widetilde a_K(q)-a(K,T-q)|$. The preceding Gaussian bound yields
\begin{equation}
\sup_x\frac{|V-\widetilde C^{\mathrm{ref}}|}{K}
\le C\tau+C\int_0^\tau\varepsilon_K(q)q^{-1/2}\,dq
\le C(\tau+\overline\varepsilon_K\sqrt\tau)
\label{eq:quadrature}
\end{equation}
when $\varepsilon_K\le\overline\varepsilon_K$. A fixed nonzero effective line error can contaminate the remainder at $O(\sqrt\tau)$. Generator-based strike evaluation avoids sensor interpolation error but is not a guarantee for arbitrary approximate quadrature.

\paragraph{Full statement of Proposition~\ref{prop:profile}.}
Under the differentiated assumptions, put $a_* =\bar a(0,0)$, $b_* =\partial_x\bar a(0,0)$, and let $g_{a_*}$ be the centered Gaussian density of variance $a_*$.
In spot coordinates, $b_*=K\partial_S a(K,T)$, so its nonvanishing is the nonzero terminal variance-slope condition used in the main text.
Along $S_\tau=Ke^{y\sqrt\tau}$, for fixed $y$,
\begin{align}
\frac{W(S_\tau,T-\tau)}{K\tau}&\longrightarrow G_0(y)=\tfrac{b_*}{4}y g_{a_*}(y),\nonumber\\
\frac{W_S(S_\tau,T-\tau)}{\sqrt\tau}&\longrightarrow G_1(y)=\tfrac{b_*}{4}(1-y^2/a_*)g_{a_*}(y),\label{eq:profiles}\\
KW_{SS}(S_\tau,T-\tau)&\longrightarrow G_2(y)=\tfrac{b_*}{4}(y^3/a_*^2-3y/a_*)g_{a_*}(y).\nonumber
\end{align}
Along the same rays, $\sqrt\tau K\Cref_{SS}\to g_{a_*}$.
For $b_*y\ne0$, $W/(\tau/T)^\alpha$ has a bounded nonzero limit precisely at $\alpha=1$.
The criticality concerns the value target, not global unweighted $C^2$ regularity.
Generically $R^*=O(1)$, $R^*_S=O(\tau^{-1/2})$, and $KR^*_{SS}=O(\tau^{-1})$ on these rays.
If $b_*=0$, the uniqueness claim for this exponent does not apply.
Manufactured reference calculations support the scaled profiles, not a theorem about training (Appendix~\ref{app:verification}).

\paragraph{Proof of Proposition~\ref{prop:profile}.}
Put $W_\varepsilon(y,s)=\varepsilon^{-2}w(\varepsilon y,\varepsilon^2s)$. Equation~\eqref{eq:logcorrection} becomes
\begin{align}
\partial_sW_\varepsilon={}&\tfrac12\bar a(\varepsilon y,\varepsilon^2s)\partial_{yy}W_\varepsilon
+\varepsilon\big(r-\bar a(\varepsilon y,\varepsilon^2s)/2\big)\partial_yW_\varepsilon\nonumber\\
&-\varepsilon^2rW_\varepsilon+\widetilde f(\varepsilon y,\varepsilon^2s).
\end{align}
Taylor expansion gives $\bar a(\varepsilon y,\varepsilon^2s)-\bar a(0,\varepsilon^2s)=\varepsilon b_*y+o(\varepsilon)$ and $A_0(\varepsilon^2s)/\varepsilon^2\to a_*s$. Consequently
$\widetilde f(\varepsilon y,\varepsilon^2s)\to(b_*/2)y g_{a_*s}(y)$ in a local parabolic H\"older norm away from $s=0$. The value bound gives $|W_\varepsilon(y,s)|\le Cs$. Interior Schauder estimates and compactness produce locally $C^{2,1}$-convergent subsequences on $|y|\le M$, $s_0\le s\le1$. Each limit has zero initial trace and solves
\begin{equation}
\partial_sW_0=\tfrac12a_*\partial_{yy}W_0+\tfrac12b_*y g_{a_*s}(y),\qquad W_0(y,0)=0.
\end{equation}
Uniqueness in the bounded mild-solution class identifies the whole limit. The Gaussian convolution identity
\begin{equation}
\int_\R g_{a_*(s-q)}(y-z)\,z g_{a_*q}(z)\,dz=\frac qs\,y g_{a_*s}(y)
\end{equation}
evaluates Duhamel's integral as $W_0(y,s)=(b_*/4)syg_{a_*s}(y)$. Taking $s=1$, $\varepsilon=\sqrt\tau$, and using
\begin{equation}
W_S=e^{-x}w_x,\qquad KW_{SS}=e^{-2x}(w_{xx}-w_x)
\end{equation}
proves all three profiles in \eqref{eq:profiles}. The carrier limit follows directly from $A_0(\tau)/\tau\to a_*$. Finally,
\begin{equation}
\frac{W(S_\tau,T-\tau)}{(\tau/T)^\alpha}=KT^\alpha\tau^{1-\alpha}\big(G_0(y)+o(1)\big).
\end{equation}
For $b_*y\ne0$, $G_0(y)\ne0$, which gives a vanishing, finite nonzero, or divergent limit according as $\alpha<1$, $\alpha=1$, or $\alpha>1$.\hfill$\square$

Figure~\ref{fig:app_profiles} visualizes the universal dimensionless
shapes obtained from \eqref{eq:profiles}. It separates bounded correction
values from the derivative scales that remain singular after rescaling.
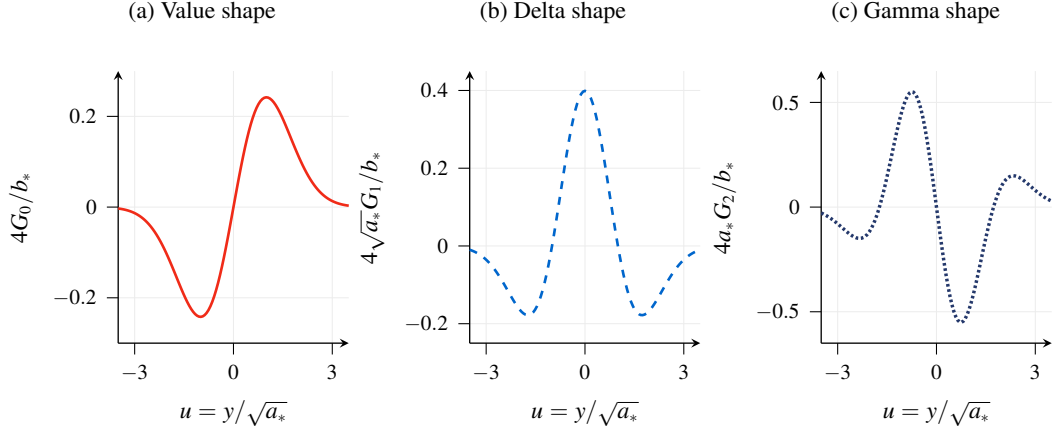
\begin{figure}[!htbp]
\centering
\begingroup\fontfamily{ptm}\selectfont\mathversion{figuretimes}

\begin{tikzpicture}
\begin{groupplot}[appsmall,group style={group size=3 by 1,horizontal sep=1.6cm},
 xmin=-3.5,xmax=3.5,xtick={-3,0,3},xlabel={$u=y/\sqrt{a_*}$},domain=-3.5:3.5,samples=121]
\nextgroupplot[title={(a) Value shape},ylabel={$4G_0/b_*$},ymin=-.3,ymax=.3,ytick={-.2,0,.2}]
\addplot[appRed,no marks] {x*exp(-x*x/2)/sqrt(2*pi)};
\nextgroupplot[title={(b) Delta shape},ylabel={$4\sqrt{a_*}G_1/b_*$},ymin=-.25,ymax=.45,ytick={-.2,0,.2,.4}]
\addplot[appBlue,dashed,no marks] {(1-x*x)*exp(-x*x/2)/sqrt(2*pi)};
\nextgroupplot[title={(c) Gamma shape},ylabel={$4a_*G_2/b_*$},ymin=-.65,ymax=.65,ytick={-.5,0,.5}]
\addplot[appNavy,densely dotted,line width=1.3pt,no marks] {(x*x*x-3*x)*exp(-x*x/2)/sqrt(2*pi)};
\end{groupplot}
\end{tikzpicture}

\endgroup
\caption{Dimensionless first-order skew profiles from \eqref{eq:profiles}, for $b_*\ne0$.
With $u=y/\sqrt{a_*}$ and standard normal density $\phi$, the three curves
are $u\phi(u)$, $(1-u^2)\phi(u)$, and $(u^3-3u)\phi(u)$.
They correspond to the limits of $W/(K\tau)$, $W_S/\sqrt\tau$, and $KW_{SS}$
after the displayed amplitude normalizations. These are exact analytic
shapes, not fitted curves or empirical convergence measurements.
Nonzero-target criticality requires $b_*y\ne0$; the value profile vanishes at $y=0$.}
\label{fig:app_profiles}
\end{figure}

\paragraph{Localization near an unclipped strike.}
Suppose the value assumptions hold globally and the coefficient is smooth on $|x|<d$ for small remaining times. Choose a smooth, bounded, uniformly elliptic extension agreeing there. Its strike-line carrier is identical. The difference of the two normalized prices solves the same homogeneous parabolic equation inside the neighborhood, with zero initial data and bounded lateral data. A bounded-drift, bounded-volatility log diffusion started in $|x|\le d/2$ exits $|x|<d$ before $\tau$ with probability at most $Ce^{-c/\tau}$. The stopped representation bounds the value difference by the same scale. Interior derivative estimates add only powers of $\tau^{-1}$, so the first two derivative differences vanish faster than the scales in \eqref{eq:profiles} on $x=y\sqrt\tau$. Thus the local profiles survive clipping away from the strike. This is not a discretization or finite-domain error estimate.

\FloatBarrier
\subsection{Error transfer, full FAL derivatives, and PDE residuals}
\label{app:stages}
Here PDE residual refers to the governing-equation defect $\Pa u$; it is
distinct from the carrier-relative correction learned by the network. For correction error $\rho_\theta=R_\theta-R^*$, define the weighted norm on the nonterminal domain $\Omega$ under consideration:
\begin{equation}
N_\rho=\sup_\Omega\left(\frac{|\rho_\theta|}{K}+\sqrt\chi|\rho_{\theta,S}|+\chi K|\rho_{\theta,SS}|\right),\quad \chi>0.
\label{eq:weighted}
\end{equation}
Equation~\eqref{eq:transfer} follows from $z_\theta-V=\chi\rho_\theta$ and the independence of $\chi$ from $S$.
Each term in \eqref{eq:weighted} bounds the corresponding correction error, giving the three inequalities separately when $N_\rho<\infty$.
Finite checks of these identities cannot establish a uniform bound on the learned norm.
In particular, the maturity gate does not force Gamma error to vanish: for a fixed smooth decoder, $\chi KR_{\theta,SS}$ can vanish while $KW_{SS}$ has a nonzero limit.
At nonterminal points, FAL gives the global price estimate
$|\pP-V|\le\chi|\rho_\theta|+K\delta_v\log2$.
Its Greek derivatives contain bound-derivative and Hessian terms, so pre-FAL Greek guarantees do not transfer globally.

Before FAL, the exact PDE-residual identity is
\begin{equation}
\Pa z_\theta=-f+\chi\Pa R_\theta+R_\theta/T
=\chi\Pa\rho_\theta+\rho_\theta/T.
\label{eq:residual}
\end{equation}
This identity connects the correction target to physics information without a PDE loss; it does not guarantee successful optimization.
Post-FAL differentiation adds the curvature and smoothing terms derived below.

For fixed price-unit smoothing $\delta$, define $F=\FAL_\delta(z_\theta;L,U)$, $A=\sigm((z_\theta-L)/\delta)$, $B=\sigm((z_\theta-U)/\delta)$, $p=A-B$, $q=1-A$, $s=B$, $h_L=A(1-A)/\delta$, and $h_U=B(1-B)/\delta$. On a smooth region,
\begin{align}
F_S&=p z_{\theta,S}+qL_S+sU_S,\label{eq:fullfalDelta}\\
F_{SS}&=p z_{\theta,SS}+qL_{SS}+sU_{SS}
+h_L(z_{\theta,S}-L_S)^2-h_U(z_{\theta,S}-U_S)^2.
\label{eq:fullfalGamma}
\end{align}
These follow from \eqref{eq:falgradient}--\eqref{eq:falhessian}. The composition need not be globally smooth in spot. For fixed $\tau>0$, let $S_b=Ke^{-r\tau}$ be the kink of the lower bound. If $z_\theta(\cdot,t)$ is smooth and finite at $S_b$, then $L_S$ jumps from $0$ to $1$ while $p,q,s$ and the other smooth factors remain continuous, so
\begin{equation}
F_S(S_b^+,t)-F_S(S_b^-,t)
=1-\sigm\!\left(\frac{z_\theta(S_b,t)}{\delta}\right)>0.
\label{eq:faldeltajump}
\end{equation}
Thus the post-FAL predictor is generally only piecewise $C^1$ in spot; this does not invalidate grid-based error metrics, but it prevents interpreting scalar softplus smoothness as a global Greek-smoothness guarantee. On a margin region $z_\theta-L\ge m_0$, $U-z_\theta\ge m_0$,
\begin{align}
|F_S-V_S|\le{}&\chi|\rho_{\theta,S}|+e^{-m_0/\delta}(|L_S-V_S|+|U_S-V_S|),\label{eq:falDeltabound}\\
|F_{SS}-V_{SS}|\le{}&\chi|\rho_{\theta,SS}|+e^{-m_0/\delta}(|L_{SS}-V_{SS}|+|U_{SS}-V_{SS}|)\nonumber\\
&+\frac{e^{-m_0/\delta}}{\delta}\big[(z_{\theta,S}-L_S)^2+(z_{\theta,S}-U_S)^2\big].
\label{eq:falGammabound}
\end{align}
Subtract $V_S=(p+q+s)V_S$ and its second-derivative analogue, then use $p\le1$, $q,s\le e^{-m_0/\delta}$. These are local, not global, Greek bounds: $L$ is not $C^2$ at $S=Ke^{-r\tau}$ and no positive margin is uniform at maturity.

For time-dependent $\delta_e(\tau)>0$ and $b=\tfrac12aS^2$, the full post-FAL PDE residual on smooth nonterminal regions is
\begin{align}
\Pa F={}&p\Pa z_\theta-b\big[h_L(z_{\theta,S}-L_S)^2-h_U(z_{\theta,S}-U_S)^2\big]\nonumber\\
&+\big(\delta'_e(\tau)+r\delta_e\big)\partial_\delta\FAL_{\delta_e}(z_\theta;L,U).
\label{eq:postfalresidual}
\end{align}
Indeed, for $q_i\in\{z_\theta,L,U\}$ and corresponding weights $p_i$, differentiation gives $F_\tau=\sum_i p_iq_{i,\tau}+\delta'_e\partial_\delta\FAL$ and $F_{SS}=\sum_i p_iq_{i,SS}+q_S^\top\nabla_q^2\FAL q_S$. Homogeneity yields
$F-\sum_i p_iq_i=\delta_e\partial_\delta\FAL$.
On either smooth branch of $L$, $\Pa L=\Pa U=0$, proving \eqref{eq:postfalresidual}. For the implemented constant $\delta_e=K\delta_v$, the $r\delta_e$ term remains. Product differentiation of $\Cref+\chi R_\theta$ gives \eqref{eq:residual} because $\Pa\Cref=-f$ and $\chi'=1/T$. Discrete PP is outside these differentiation identities.

\paragraph{Interior recovery requires a continuum norm.}
Let $u(x,\tau)=V(Ke^x,T-\tau)/K$ and $\widehat u$ be a prediction at one specified output stage. On nested parabolic cylinders $Q'\Subset Q$, separated from maturity, boundaries, and output kinks, assume uniform ellipticity, $C^{\zeta,\zeta/2}$ coefficients, bounded drift/reaction, and $u,\widehat u\in C^{2+\zeta,1+\zeta/2}(Q)$. The interior estimate \citep{r62,r63} is
\begin{equation}
\norm{\widehat u-u}_{C^{2+\zeta,1+\zeta/2}(Q')}
\le C\left(\norm{\widehat u-u}_{C^0(Q)}+\norm{\mathcal L_a\widehat u}_{C^{\zeta,\zeta/2}(Q)}\right),
\label{eq:schauder}
\end{equation}
where $\mathcal L_a=\partial_\tau-\tfrac12\bar a\partial_{xx}-(r-\bar a/2)\partial_x+r$.
This follows by applying the parabolic estimate to $e=\widehat u-u$, which satisfies $\mathcal L_ae=\mathcal L_a\widehat u$. Since $\Delta=e^{-x}u_x$ and $K\Gamma=e^{-2x}(u_{xx}-u_x)$, bounded log cylinders also give
\begin{equation}
\sup_{Q'}\left(\frac{|\widehat V-V|}{K}+|\widehat\Delta-\Delta|+K|\widehat\Gamma-\Gamma|\right)
\le C\left(\norm e_{C^0(Q)}+\norm{\mathcal L_a\widehat u}_{C^{\zeta,\zeta/2}(Q)}\right).
\end{equation}
A scalar FD RMS on one grid controls neither the continuum supremum nor the H\"older seminorm, and it also includes truncation error. It cannot certify this premise.

\FloatBarrier
\subsection{Terminal-scaled FAL: a theoretical refinement, not a tested model}
\label{app:scaledfal}
For a dimensionless $\delta_c>0$ and $\chi>0$, define
\begin{equation}
L_R=(L-\Cref)/\chi,\quad U_R=(U-\Cref)/\chi,\qquad
\widetilde V_\theta=\Cref+\chi\FAL_{K\delta_c}(R_\theta;L_R,U_R).
\label{eq:refinedfal}
\end{equation}
Translation and homogeneity show
$\widetilde V_\theta=\FAL_{K\chi\delta_c}(z_\theta;L,U)$.
Assign the payoff at $\chi=0$ and zero at $S=0$.
\begin{proposition}[Conditional value-level closure]
For $\chi>0$ and $L<U$, $L<\widetilde V_\theta<U$, and
\begin{equation}
\frac{|\widetilde V_\theta-V|}{K}\le\chi\left(\frac{|\rho_\theta|}{K}+\delta_c\log2\right).
\end{equation}
If $R_\theta/K$ is locally bounded near maturity, $\widetilde V_\theta$ converges to the payoff in value. If in addition
\begin{equation}
\begin{aligned}
R_\theta^{K,T}(S,t)&=K\mathcal R_\theta\big(\mathfrak a_{K,T},rT,\text{scale-invariant inputs};S/K,\tau/T\big),\\
\mathfrak a_{K,T}(m,s)&=Ta(Km,T(1-s)).
\end{aligned}
\end{equation}
with dimensionless sensors and compatible statistics, the composed predictor inherits \eqref{eq:scale}.
\end{proposition}
\begin{proof}
Move $\chi$ through softplus to obtain the equivalent smoothing $K\chi\delta_c$, then apply \eqref{eq:falrange}--\eqref{eq:falproject}. Boundedness of the decoder, $\chi\to0$, and carrier convergence give the terminal limit. On the joint scaling orbit, dimensionless fields, contract combinations, and sensor data do not change; the decoder, carrier, bounds, and smoothing all scale by $\lambda$. Homogeneity proves price scaling, and differentiation gives the Greek identities where defined.
\end{proof}
This is not a theorem for the implemented dimensional branch or fixed $0.002K$ smoothing. It does not transfer pre-FAL weighted Greek convergence through the refined map: its Hessian grows like $(K\chi\delta_c)^{-1}$ near active bounds. Empirical claims would require a separate matched experiment.

\clearpage
\section{Implementation, data, and evaluation protocols}
\label{app:protocols}
\FloatBarrier
\subsection{Model vocabulary and analytical scope}
Figure~\ref{fig:mechanism} summarizes the FI construction.
Table~\ref{tab:vocabulary} defines the model names used throughout the paper,
and Table~\ref{tab:theoryscope} separates mathematical assumptions from
properties established for the implemented predictor.

\begin{table}[htbp]
\caption{Model names and components. RA, FAL, and PP have distinct roles.}
\label{tab:vocabulary}
\centering\small
\begin{tabularx}{\linewidth}{lY}
\toprule
Name & Meaning\\
\midrule
DeepONet & Deep Operator Network\\
Normalized DeepONet & DeepONet with contract normalization\\
FAL & Financial Admissibility Layer; smooth-in-raw-price pointwise call-price-bound map; spot composition is piecewise smooth\\
RA & Strike-matched carrier--residual ansatz\\
PP & Deterministic post-processing of a predicted grid\\
PI-DeepONet & Physics-informed DeepONet; canonical comparator has no RA/FAL/PP\\
FI-DeepONet & Normalization + RA + FAL; principal ID output uses terminal correction\\
FI-DeepONet-PP & The same selected FI checkpoint followed by PP\\
QO-FI-DeepONet & Exact compilation of selected frozen maps with orthogonal factors and a classical signed core\\
QDO-FI-DeepONet & Restricted diagonal--orthogonal map approximation followed by supervised adaptation\\
\bottomrule
\end{tabularx}
\end{table}

\begin{table}[htbp]
\caption{Analytical statements and applicability. Finite verification does not establish a uniform theorem premise.}
\label{tab:theoryscope}
\centering\small
\begin{tabularx}{\linewidth}{p{0.22\linewidth}YY}
\toprule
Result & Required regime & Implemented-model interpretation\\
\midrule
Scale symmetry & Joint spot--strike--field transformation & Exact for the operator; learned model only scale-aware\\
FAL interval & Finite values, common units, $L<U$ & Pointwise only; scalar map smooth, but spot composition inherits lower-bound kink\\
RA cancellation & Value-level assumptions or stated coupling hypotheses & Continuous canonical-field value bound; reference errors separate\\
First-order skew profile & Strong local smoothness; $b_*y\ne0$ for criticality & Local near the strike; clipping away from it is handled by localization\\
Weighted transfer & Finite $N_\rho$ in \eqref{eq:weighted}; pre-FAL stage & Exact identities; uniform learned norm unverified\\
Interior recovery & Same-stage PDE residual in H\"older norm & Not certified by scalar grid RMS\\
Terminal-scaled FAL & Smoothing $K\chi\delta_c$ & Theoretical refinement, not a principal experiment\\
Post-FAL PDE residual & Smooth region and full chain rule & Excludes kinks, terminal line, and PP\\
\bottomrule
\end{tabularx}
\end{table}

\FloatBarrier
\subsection{Fixed coefficient family, sensors, and split support}
\label{app:data}
We use European calls with $T=1$ and zero dividends. The local-volatility family is
\begin{equation}
\sigma_{\mathrm{loc}}(S,t)=\operatorname{clip}_{[0.05,1]}
\left\{\sigma_0\left[1+\beta\tanh\!\left(2\frac{S-K}{K}\right)\right]
\left[1+\gamma\left(\frac{T-t}{T}-\frac12\right)\right]\right\}.
\label{eq:family}
\end{equation}
The raw coefficient in \eqref{eq:family} is clipped to $[0.05,1.00]$ before squaring. Implementation safeguards replace $K,T$ in denominators by $\max(K,10^{-8})$, $\max(T,10^{-8})$; these are inactive on the reported support. The 101 payoff sensors are equally spaced on $[0,220]$. The 231 volatility values use 21 spot positions on $[0,220]$ and 11 calendar-time positions on $[0,1]$, stored time-major. The scalars are $(K,r,\sigma_0,\beta,\gamma)$. The reported generalization experiments evaluate held-out and parameter-support performance within this declared coefficient family; off-family functional response is assessed separately.

Table~\ref{tab:data} lists the fixed support and split sizes;
Table~\ref{tab:ooddesign} decomposes the 320-surface OOD population into
eight single-axis shifts and the joint-shift group.

\begin{table}[htbp]
\caption{Fixed synthetic data support. All intervals are half-open. Validation and held-out ID share the training support.}
\label{tab:data}
\centering\small
\begin{tabularx}{\linewidth}{p{0.25\linewidth}Y}
\toprule
Item & Specification\\
\midrule
Surface counts & 512 training; 32 validation; 32 held-out ID; 320 strict parameter-support OOD\\
Training/ID support & $K\in[55,135)$, $r\in[0,0.13)$, $\sigma_0\in[0.10,0.45)$, $\beta\in[-0.90,0.60)$, $\gamma\in[-0.45,0.50)$\\
Strike shifts & Low $K\in[40,52)$; high $K\in[138,160)$\\
Rate / volatility & High $r\in[0.14,0.18)$; high $\sigma_0\in[0.48,0.60)$\\
Skew shifts & $\beta\in[-1.15,-0.95)$ or $[0.65,0.80)$\\
Term shifts & $\gamma\in[-0.65,-0.50)$ or $[0.55,0.70)$\\
Joint shifts & Two or three shifted coordinates under the recorded design\\
Unshifted coordinates & $K\in[70,120)$, $r\in[0.03,0.08)$, $\sigma_0\in[0.20,0.30)$, $\beta\in[-0.30,0.10)$, $\gamma\in[-0.10,0.15)$\\
\bottomrule
\end{tabularx}
\end{table}

\begin{table}[htbp]
\caption{Disjoint parameter-support OOD design. Excursions are farthest proposal distances beyond the training boundary, normalized by its support width.}
\label{tab:ooddesign}
\centering\small
\begin{tabular}{lrlr}
\toprule
Regime & Count & Shifted support & Max. excursion\\
\midrule
Low $K$ &32&$[40,52)$&18.75\%\\
High $K$ &32&$[138,160)$&31.25\%\\
High $r$ &32&$[0.14,0.18)$&38.46\%\\
High $\sigma_0$ &32&$[0.48,0.60)$&42.86\%\\
Negative $\beta$ &32&$[-1.15,-0.95)$&16.67\%\\
Positive $\beta$ &32&$[0.65,0.80)$&13.33\%\\
Negative $\gamma$ &32&$[-0.65,-0.50)$&21.05\%\\
Positive $\gamma$ &32&$[0.55,0.70)$&21.05\%\\
Joint &64&Two or three coordinates&Multi-axis\\
\midrule
Total &320&Eight single-axis + joint&---\\
\bottomrule
\end{tabular}
\end{table}
Parameter rows and their ordering are recorded in the project, but the original split-generation seed and complete sampling distribution are not established. Comparisons therefore use fixed realized data without asserting a uniform sampling law. The single-axis design shifts $K,\beta,\gamma$ in both directions but $r,\sigma_0$ only upward. A separate 64-surface high-strike diagnostic lies \emph{inside} training support and is not part of the parameter-support OOD evidence. Realized excursions need not reach proposal endpoints (Table~\ref{tab:realizedsupport}).

The data-flow schematic in Figure~\ref{fig:app_dataflow} distinguishes
the sampled network input from the carrier's direct coefficient access.
The 410/102 realization split is nested within the original 512 training
surfaces; it does not use the 32-surface development-validation or ID sets.
\begin{figure}[!htbp]
\centering
\begingroup\fontfamily{ptm}\selectfont\mathversion{figuretimes}

\begin{tikzpicture}[x=1cm,y=1cm]
\node[appbox,text width=12.7cm,minimum height=.7cm] (family) at (6.6,0)
{\textbf{Fixed coefficient family and contracts}\quad $T=1$, $(K,r,\sigma_0,\beta,\gamma)$; clip $\sigma$ before squaring};
\node[appbox,draw=appBlue,text width=3.7cm,minimum height=1.8cm] (sensor) at (2,-1.8)
{\textbf{Sampled branch input}\\101 payoff values\\231 volatility values\\Five scalar features};
\node[appbox,draw=appNavy,text width=3.7cm,minimum height=1.8cm] (carrier) at (6.6,-1.8)
{\textbf{Direct carrier access}\\$a(K,T-u)\ \longrightarrow\ A_0$\\64-point quadrature};
\node[appbox,draw=appGray,text width=3.7cm,minimum height=1.8cm] (label) at (11.2,-1.8)
{\textbf{FD price labels}\\Rannacher startup\\Crank--Nicolson steps\\Pointwise coefficient};
\draw[appflow] (family.south) -- ++(0,-.25) -| (sensor.north);
\draw[appflow] (family) -- (carrier);
\draw[appflow] (family.south) -- ++(0,-.25) -| (label.north);
\node[appbox,text width=3.7cm,minimum height=1.3cm] (train) at (2,-3.7)
{\textbf{512 training surfaces}\\410 realization fitting\\+ 102 realization selection};
\node[appbox,text width=3.7cm,minimum height=1.3cm] (val) at (6.6,-3.7)
{\textbf{32 validation surfaces}\\Price-based checkpoint\\selection};
\node[appbox,draw=appRed,text width=3.7cm,minimum height=1.3cm] (test) at (11.2,-3.7)
{\textbf{Held-out evaluation}\\32 ID; 320 parameter OOD\\No fitting or selection};
\node[appnote,text width=12.7cm] at (6.6,-5)
{Top arrows denote data construction. Bottom boxes denote distinct split roles.\\Carrier coefficient access is richer than the sampled sensor path alone.};
\end{tikzpicture}

\endgroup
\caption{Data access and split roles. The carrier evaluates the generating
coefficient on the strike line; it does not infer that line from the 231
volatility sensors. The original 512 training surfaces also supply the disjoint
410/102 fitting/selection split for quantum-compatible realization, whereas
the 32-surface validation set selects the principal neural checkpoints.
The bottom boxes specify roles, not a random resampling algorithm:
the realized split rows are fixed and the original generation stream is incomplete.
The OOD solver and historical output adapter are separately specified in
Appendix~\ref{app:ood}.}
\label{fig:app_dataflow}
\end{figure}

\FloatBarrier
\subsection{Finite-difference reference and analytic carrier evaluation}
\label{app:solver}
For the principal training and ID data, the reference advances $v_\tau=\Aa v$ directly in spot on $[0,605]$, with $N_S=880$ spatial intervals, $\Delta S=0.6875$, $N_t=320$ time intervals, and $\Delta\tau=0.003125$, in double precision. Coefficients are evaluated at time-step midpoints. Four backward-Euler half steps replace the first two full time intervals (Rannacher startup), followed by Crank--Nicolson steps with tridiagonal Thomas solves \citep{r05,r06,r07,r08}. The conditions are
\begin{equation}
v(0,\tau)=0,\quad v(605,\tau)=\max(605-Ke^{-r\tau},0),\quad v(S,0)=(S-K)^+.
\label{eq:fdbc}
\end{equation}
Reference Delta and Gamma use centered second-order spot differences. The strike is not shifted onto a grid node or midpoint. The output is restricted to 41 calendar times and 81 spots on $[0,1]\times[0,220]$ with coordinate tolerance $2\times10^{-13}$; arrays are ordered (surface,time,spot). The smallest positive remaining maturity is $0.025$, then $0.05$. The query edge 220 is not the computational boundary. The complete OOD evaluation uses a separate archived solver configuration described in Appendix~\ref{app:ood}.

The boundary in \eqref{eq:fdbc} is an asymptotic finite-domain approximation. Refining a fixed domain does not test truncation, and Rannacher startup alone does not guarantee accurate shortest-maturity Gamma for every strike \citep{r06}. The distinct checks in Appendix~\ref{app:verification} address these issues empirically.

The carrier integral $A_0$ is evaluated from the coefficient generator, not interpolated sensors, with 64-point Gauss--Legendre quadrature. When strike-line clipping is inactive,
\begin{equation}
A_0(\tau)=\sigma_0^2\left[(1-\gamma/2)^2\tau+(1-\gamma/2)\gamma\frac{\tau^2}{T}+\frac{\gamma^2\tau^3}{3T^2}\right].
\label{eq:analyticA0}
\end{equation}
This is regular at $\gamma=0$. The quadrature integrates the polynomial exactly in exact arithmetic, subject to floating-point rounding in numerical evaluation. Equation~\eqref{eq:quadrature} treats a general effective line-integration error.

\FloatBarrier
\subsection{Architecture, training, and checkpoint selection}
\label{app:training}
\paragraph{DeepONet architecture and normalized coordinates.}
For the normalized FI parameterization, let $\xi$ collect the sampled payoff, sampled volatility, and contract scalars.
The fused branch features $b_{\theta,j}(\xi)$ and query-trunk features $q_{\theta,j}$ produce the scalar logit
\begin{equation}
f_\theta(\xi;m,\bar t)=\sum_{j=1}^{p}b_{\theta,j}(\xi)q_{\theta,j}(m,\bar t)+b_0,
\quad m=S/K,\quad \bar t=t/T,\quad c=V/K.
\label{eq:deeponet}
\end{equation}
Here $p$ is the latent width and $b_0$ is a learned scalar bias.
The relation between $f_\theta$ and the correction is specified in the FI training protocol below; the canonical E0 baseline uses physical query coordinates as described separately.
The principal models have hidden/latent width 128, tanh activations, and 258,305 trainable parameters. The network combines separate payoff, volatility, and scalar branches with a fusion map and a query trunk, as in \eqref{eq:deeponet}. Each of the payoff, volatility, scalar, fusion, and trunk MLPs has three biased linear layers, with tanh after the first two. Their input dimensions are respectively $101,231,5,384,2$; hidden and output dimensions are 128. The branch features are concatenated before fusion; the fusion--trunk dot product has a separate learned scalar bias.

\paragraph{DeepONet: canonical unnormalized baseline (E0).}
E0 uses physical $(S,t)$ queries and price $V$, with frozen statistical standardization of branch inputs, query coordinates, and output. The E0 output normalizer is fit on the final Rannacher training labels only, using population standard deviation plus $10^{-8}$. The resulting mean and scale are $40.613327773255371$ and $45.26865379343748$, respectively. Validation, test, and OOD labels are excluded from the fit. Training uses 5,000 Adam updates at $10^{-3}$, reloads the best raw-phase validation checkpoint, and uses a new Adam optimizer for 2,000 updates at $10^{-4}$. Adam has betas $(0.9,0.999)$, epsilon $10^{-8}$, zero weight decay, and no AMSGrad. For each principal seed, \texttt{default\_rng(seed)} first draws all $7{,}000\times32$ surface indices and then all $7{,}000\times512$ query indices with replacement; each update shares its query vector across surfaces. Training is float32 on one CPU thread with deterministic Torch algorithms and seeded Python, NumPy, and Torch initialization. Validation pools standardized-price MSE over all 32 validation surfaces at each phase start, its first update, and every 100 updates: 74 evaluations. Strict improvement selects the checkpoint, ties retain the earlier candidate, and E0 retains the raw-phase best score during continuation. There is no early stopping, gradient clipping, scheduler, warmup, auxiliary loss, RA, FAL, or PP in this baseline.

\paragraph{Canonical PI-DeepONet.}
The PDE residual of a destandardized dimensionless price prediction $\widehat c$ is
\begin{equation}
R_{\rm el}[\widehat c]=\widehat c_{\bar t}+\frac T2a(Km,T\bar t)m^2\widehat c_{mm}+Trm\widehat c_m-Tr\widehat c,
\quad \Pa(K\widehat c)=-(K/T)R_{\rm el}[\widehat c].
\label{eq:piresidual}
\end{equation}
The PDE loss $\mathcal L_{\rm PDE}$ is the mean squared value of
$R_{\rm el}[\widehat c]$ at the sampled collocation points.
The objective is $\mathcal L_{\rm data}+\mathcal L_{\rm PDE}+\mathcal L_{\rm terminal}+\mathcal L_{\rm boundary}$ with four unit weights. Each of the five principal seeds runs 5,000 Adam updates at $10^{-3}$, reloads the lowest-validation-error raw-phase checkpoint, creates a \emph{new} Adam optimizer, and continues for 2,000 updates at $10^{-4}$. Adam has betas $(0.9,0.999)$, epsilon $10^{-8}$, and zero weight decay.

Each update samples 32 training surfaces and a shared vector of 512 supervised grid-query indices, with replacement. Independently for each surface, it samples 512 PDE points
$m=0.01+U(220/K-0.01)$, $\bar t=0.975U$;
512 terminal points $m=(220/K)U$, $\bar t=1$;
and 512 times $U\in[0,1)$ shared by the boundaries $m=0$ and $m=220/K$. Uniform draws are independent apart from the stated sharing. There is no near-strike enrichment. Supervised draws use \texttt{default\_rng(seed)}; float32 collocation coordinates use \texttt{default\_rng(seed+5051005)}. Each loss has its own mean-square reduction; the boundary loss averages concatenated lower and upper samples. Four sequential backward calls accumulate gradients before one optimizer update.

Derivatives are with respect to raw $m=S/K$ and elapsed $\bar t$, with input standardization and output destandardization inside the differentiation graph. Volatility is bilinearly interpolated from reconstructed sensors, clamping coordinates to the sensor rectangle, and \emph{then squared}. Validation is standardized price MSE on complete 32-surface development-validation grids at each phase start and every 100 updates (72 evaluations in total). Continuation resets the best-validation state; strict improvement selects, ties retain the earlier checkpoint. ID/OOD results do not enter selection. This comparator has no RA, FAL, or PP.

\paragraph{Historical FI price-loss algebra and later inference rule.}
Let $f_\theta$ be the standardized scalar network logit, with frozen dimensionless-price normalization
\begin{equation}
\mu=0.5017557622056907,\qquad s_y=0.6236731671283599.
\end{equation}
The historical E5 training forward is
\begin{align}
r_\theta&=\mu+s_yf_\theta,\quad z_v=\Cref/K+\chi r_\theta,
\quad y_n=(V_{\rm FD}/K-\mu)/s_y,\nonumber\\
\mathcal L_{\rm raw}&=\operatorname{mean}\left[\left((z_v-\mu)/s_y-y_n\right)^2\right],\label{eq:filoss}\\
\mathcal L_{\rm cont}&=\operatorname{mean}\left[\left((\FAL_{\delta_v}(z_v;L/K,U/K)-\mu)/s_y-y_n\right)^2\right],\nonumber
\end{align}
with $\delta_v=0.002$. For $\chi>0$, define the numerical-label correction
$R_{\rm FD}^*=(V_{\rm FD}-C^{\mathrm{ref}})/\chi$.
The pointwise raw-stage loss in \eqref{eq:filoss} then equals
$\chi^2(R_\theta-R_{\rm FD}^*)^2/(K^2s_y^2)$.
Thus reconstructed-price supervision implicitly weights
squared numerical-correction error by $\chi^2$, without an
additional user-specified maturity weight.
The numerical target differs from the exact correction by
$R_{\rm FD}^*-R^*=(V_{\rm FD}-V)/\chi$; comparison with the
continuous target therefore also involves reference error.
This identity applies before FAL and does not carry unchanged
through the nonlinear continuation loss. E5 uses 5,000 Adam updates at $10^{-3}$ and 2,000 at $10^{-4}$, with 32-surface and 512-query batches and validation-price selection. There is no direct correction-label MSE, extra user-specified $\chi$ point weight, near-strike loss weight, Greek loss, PDE loss, or financial penalty in final E5. Terminal nodes remain eligible; $\chi=0$ removes their learned correction. Ordinary FAL was applied during continuation \emph{including terminal nodes}. Consequently, at terminal nodes the continuation-stage FAL output is parameter-independent, so these samples contribute no gradient to the learned network despite remaining in the batch. The terminal-preserving output in \eqref{eq:fi} is a later evaluation change with frozen selected weights, not a retrained architecture. FAL-only E3 interprets the destandardized network logit directly as normalized price and also applies FAL in continuation.

The principal seeds are $\{2026,3407,5201,7713,9109\}$. The separate kink-treatment seeds are $\{20260719,20260720,20260721\}$. Training specifics of development studies or historical baselines are not assumed to be identical merely because an update count is shared. Alternative kink treatments include gradient-balanced soft financial penalties, diffusion-distance coordinate enlargement based on $\log(S/K)/\sqrt{A_0}$ with near-terminal regularization, and near-strike oversampling. The coordinate intervention is motivated by adaptive-scale inputs \citep{r59}; it is not a full reproduction of every such method.
\paragraph{Matched carrier--rescaling mechanism ablation.}
To separate the analytic carrier from maturity rescaling, we use the generalized reconstruction
\begin{equation}
z_{\theta,\alpha}
=
C^{\mathrm{ref}}+g_\alpha(\chi)R_\theta,
\qquad
g_\alpha(\chi)=
\begin{cases}
1, & \alpha=0,\\
\chi^\alpha, & \alpha>0,
\end{cases}
\label{eq:alpha_ablation}
\end{equation}
with
\[
\alpha\in\{0,\tfrac12,1,\tfrac32\}.
\]
The learned arms are denoted A0, A05, A1, and A15, respectively. B0 is the normalized direct-price FAL control. All arms use the same frozen H1 data, width-128 architecture, input/output statistics, five principal seeds, surface/query sampling schedules, $5{,}000+2{,}000$ Adam-update budget, validation-price checkpoint rule, and FAL continuation protocol. No PDE loss, Greek loss, near-strike weighting, oversampling, scheduler, or PP is introduced.

A0, A05, and A15 were trained as fifteen new paired runs. Historical A1/E5 and B0/E3 checkpoints were reused only after exact generalized-path replay, initialization, data, checkpoint, and schedule audits passed. No held-out metric was used to alter the exponent set or training protocol. The primary mechanism analysis uses pre-FAL nonterminal outputs; a second evaluation applies the same terminal-preserving post-FAL convention to all arms. This distinction prevents FAL or terminal assignment from being attributed to the residual parameterization itself.
\paragraph{Selected checkpoints and replay scope.}
In the principal seed order above, the selected global updates are $(6900,7000,6900,6900,6900)$ for refitted E0, $(7000,6900,7000,7000,6600)$ for canonical PI, and $(6800,7000,6800,6600,6600)$ for FI. The E0 results in Tables~\ref{tab:classical} and~\ref{tab:principal} use the output normalizer fitted to the final Rannacher training labels. Their configurations, training logs, checkpoint hashes, and selected-state replay are recorded separately from the historical synthetic-core archive. That archive retains the earlier E0 states and output normalizer, so its replay does not reproduce the current E0 column. Appendix~\ref{app:context} distinguishes these versions and the separate OOD replay extension.

\FloatBarrier
\subsection{Deterministic grid post-processing}
\label{app:pp}
For each float64 array with shape $N\times41\times81$, ordered (sample,time,spot), PP performs exactly two cycles. Each cycle applies: call-price lower/upper clipping and $S=0$ overwrite; terminal payoff overwrite; reverse-calendar-time cumulative maximum; unit-weight pooled-adjacent-violators (PAVA) repair of forward spot slopes; slope clipping to $[0,1]$; left-anchored reconstruction; and bounds/zero-spot enforcement. After two cycles, apply a final terminal overwrite, time-monotonicity pass, and bounds/boundary enforcement. There is no convergence test or repair tolerance; $10^{-8}$ is the post hoc violation threshold.

Only the slope-PAVA step is an isotonic projection. The whole noncommuting sequence is order-dependent, not a simultaneous metric projection onto the intersection of all constraints. It acts after inference, outside backpropagation, and changes no parameters, optimizer state, checkpoint selection, or ranking. The recorded five-model re-evaluation reproduces the reported PP outputs. It does not give continuous-surface no-arbitrage certification or independently differentiated AD Greeks. In a time-inhomogeneous pricing problem, a backward surface indexed by valuation time must also be distinguished from a family of contract maturities at a fixed valuation date; the time-repair operation is not used as a general calendar-arbitrage theorem.

Inference forms normalized queries, evaluates $A_0$ and $\Cref$, predicts $R_\theta$, reconstructs $z_\theta$, applies unit-consistent FAL, and assigns terminal/zero-spot values. PP is an optional final grid transformation. Eligible linear-map replacements leave this pricing pipeline unchanged.

\FloatBarrier
\subsection{Metric definitions, masks, and coefficient conventions}
\label{app:metrics}
For a specified reference vector $\mathbf V$ and prediction $\widehat{\mathbf V}$,
\begin{equation}
\RelL=\frac{\norm{\widehat{\mathbf V}-\mathbf V}_2}{\norm{\mathbf V}_2},\quad
\operatorname{RMSE}=\sqrt{N^{-1}\sum_i|\widehat V_i-V_i|^2},\quad
\operatorname{MAE}=N^{-1}\sum_i|\widehat V_i-V_i|.
\end{equation}
Price p95 is a pooled absolute-error quantile, not relative price error. Principal price $\RelL$ pools $32\times41\times81=106,272$ points before computing one norm ratio per seed. Near-strike price uses $|S/K-1|\le0.05+8\epsilon_{64}$, retains terminal nodes, and has no explicit spatial-edge exclusion. Seed means and sample SDs are computed afterward.

The principal Greek predictions are physical-spot \texttt{np.gradient(edge\_order=2)} of saved prices, applied twice for Gamma. References come from the fixed 24-surface L3 ID/test subset. The historical near-strike mask intersected with $\tau>0$ retains 3,160 finite prediction/reference pairs per Greek and seed, excluding 79 terminal pairs. The historical including-terminal mask has 3,239 pairs and is confined to Table~\ref{tab:historical} and its attribution discussion. The principal PDE residual RMS pools dimensionless residuals computed from the
saved pre-PP output with the common interior-grid FD evaluator
(98,592 points per seed). The terminal row itself is not scored, but the recorded centered time stencil at the last interior level $\bar t=0.975$ uses the $\bar t=1$ terminal price; consequently the near-maturity PDE residual RMS depends on the terminal-value
convention. This principal FD evaluation is distinct from the PDE loss used in PI training
and from the AD evaluation used in the controlled PDE-loss
factorial. Price RelL2 is dimensionless.
Absolute price errors, including RMSE, MAE, and absolute-error
p95, have price units; Delta errors are dimensionless, and
Gamma errors have inverse-price units.

The archived $\tau=0.025$ profile analysis aggregates all 24 L3 surfaces and five seeds: absolute native errors are interpolated to the fixed common moneyness grid, then equally weighted medians/IQRs are formed over 120 surface--seed values per coordinate. Both price and Gamma there use the L3 population, unlike the principal 32-surface price table. IQRs are dispersion, not confidence intervals; no surface is selected by its model error.

Reference refinement in Table~\ref{tab:refinement} excludes spatial edge columns, uses $|S/K-1|\le0.05$, and defines ATM-short by $0<\tau\le0.05$. Carrier-only regional price comparison uses the near-strike mask and $0<\tau\le0.05$ for pooled regional $\RelL$ per seed. The shortest-time check uses only $\tau=0.025$. The PI coefficient-sensitivity mask instead has $0<\tau\le0.10$ within the legal sensor rectangle. Market ATM-short uses $|\log(K/F)|\le0.03$ and $7\le\mathrm{DTE}\le30$. These masks and populations are not interchangeable.

For a constraint $g_j\ge0$, violation frequency at threshold $\varepsilon_j$ counts $g_j<-\varepsilon_j$; mean severity averages $(-g_j)^+$. The principal financial magnitude sums seven native-array mean severities: lower and upper price-bound violations, negative price, negative forward spot slope, positive elapsed-time price increment, negative FD Gamma, and absolute terminal discrepancy. Time increments are not divided by step size. The factorial averages nine rule means, adding lower/upper FD-Delta violations. The threshold $10^{-8}$ defines frequencies, not magnitudes. These mixed-unit scores are protocol-specific, not a calibrated arbitrage scale.

\paragraph{Output-stage and population registry.}
Figure~\ref{fig:app_protocols} provides a reading guide to the principal
protocols. ``Pre-PP'' alone is not a complete output specification:
FI is post-FAL, mechanism comparisons may be pre-FAL, and the complete
OOD study retains its historical FI adapter.
The canonical PI baseline, the normalized-RA PDE-loss factorial,
and the selected-parent realization study are different experiments.
Neither AD versus FD nor nonterminal versus including-terminal Greek
statistics are interchangeable. Unless a caption explicitly states a
paired estimator, a quoted improvement computed from a summary table
is a ratio of reported means, not a mean within-seed improvement.
\begin{figure}[!htbp]
\centering
\begingroup\fontfamily{ptm}\selectfont\mathversion{figuretimes}

\begin{tikzpicture}[x=1cm,y=1cm]
\node[appbox,text width=3.7cm,draw=appNavy] (raw) at (2,0)
{\textbf{Pre-FAL}\\$z_\theta=\Cref+\chi R_\theta$};
\node[appbox,text width=3.7cm,draw=appRed] (fi) at (6.6,0)
{\textbf{FI before PP}\\FAL + terminal / zero-spot rule};
\node[appbox,text width=3.7cm,draw=appGray] (pp) at (11.2,0)
{\textbf{Optional grid PP}\\Deterministic, weights fixed};
\draw[appflow] (raw) -- (fi);\draw[appflow] (fi) -- (pp);
\node[anchor=west,font=\footnotesize\bfseries] at (0,-1.05) {Study};
\node[anchor=west,font=\footnotesize\bfseries] at (4,-1.05) {Population / stage};
\node[anchor=west,font=\footnotesize\bfseries] at (9.15,-1.05) {Estimator / scope};
\draw[appGray,line width=.6pt] (0,-1.35)--(13.25,-1.35);
\foreach \yy in {-2.1,-2.9,-3.7,-4.5} {\draw[gray!30] (0,\yy)--(13.25,\yy);}
\node[anchor=west,font=\footnotesize] at (0,-1.7) {Principal ID};
\node[anchor=west,font=\scriptsize,align=left] at (4,-1.7) {32 price surfaces; 24 L3 for Greeks\\FI after FAL; before PP};
\node[anchor=west,font=\scriptsize,align=left] at (9.15,-1.7) {Five-seed mean $\pm$ sample SD\\FD Greeks, $\tau>0$};
\node[anchor=west,font=\footnotesize] at (0,-2.5) {Mechanism ablation};
\node[anchor=west,font=\scriptsize,align=left] at (4,-2.5) {B0 / A0 / A05 / A1 / A15\\Pre-FAL and post-FAL reported separately};
\node[anchor=west,font=\scriptsize,align=left] at (9.15,-2.5) {Five paired seeds\\Stage-dependent rankings};
\node[anchor=west,font=\footnotesize] at (0,-3.3) {Complete OOD};
\node[anchor=west,font=\scriptsize,align=left] at (4,-3.3) {320 canonical-clipped surfaces\\Historical FI output adapter};
\node[anchor=west,font=\scriptsize,align=left] at (9.15,-3.3) {Pooled price norm per seed\\No ID/OOD degradation ratio};
\node[anchor=west,font=\footnotesize] at (0,-4.1) {PDE-loss factorial};
\node[anchor=west,font=\scriptsize,align=left] at (4,-4.1) {Normalized RA; paired objectives\\FAL is an output-map factor};
\node[anchor=west,font=\scriptsize,align=left] at (9.15,-4.1) {AD PDE residual and Greek MAE\\Distinct from principal PI};
\node[anchor=west,font=\footnotesize] at (0,-4.9) {Realization / market};
\node[anchor=west,font=\scriptsize,align=left] at (4,-4.9) {One fixed dense parent; selected QDO\\Market uses constant-volatility inputs};
\node[anchor=west,font=\scriptsize,align=left] at (9.15,-4.9) {Separate synthetic AD Greeks\\Market mid $\ne$ operator reference};
\end{tikzpicture}

\endgroup
\caption{Output-stage and protocol map. Arrows show the FI inference order,
not a guarantee that errors decrease at every stage. DeepONet and canonical
PI are direct-price comparators and do not follow the FI carrier--FAL path.
The table identifies the principal distinctions needed to read the appendix:
FD and AD, nonterminal and including-terminal masks, per-seed summaries and
selected-parent comparisons, and market versus same-input operator references.
See Appendix~\ref{app:metrics} for exact masks and aggregation.}
\label{fig:app_protocols}
\end{figure}
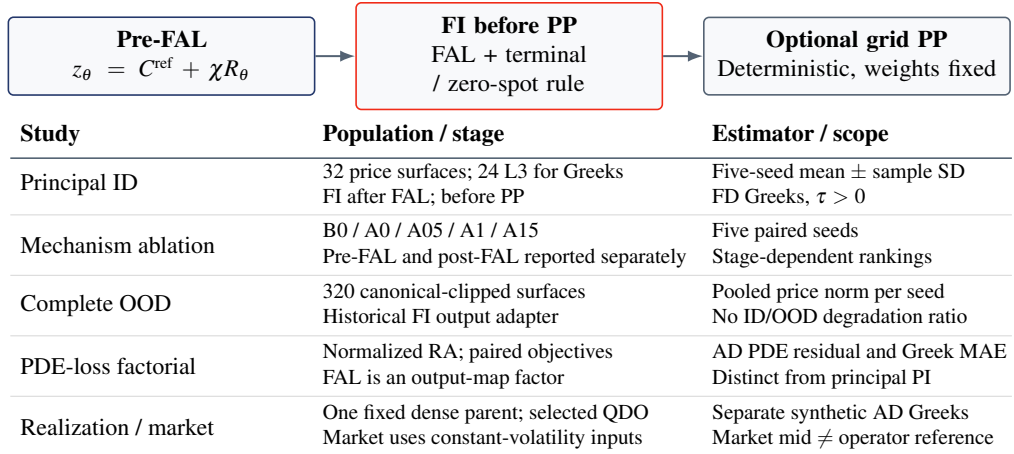

\paragraph{PI coefficient-path sensitivity.}
PI interpolates clipped $21\times11$ volatility sensors before squaring, whereas the reference evaluates and clips the generator pointwise. For fixed selected weights,
\begin{equation}
q_{\rm coef}=\frac{\operatorname{RMS}(R_{\rm sensors}-R_{\rm generator})}{\operatorname{RMS}(R_{\rm sensors})}
\label{eq:coefperturb}
\end{equation}
lies in $[0.0103,0.0148]$ on the admissible validation region and $[0.00163,0.00210]$ on this check's ATM-short subset. The direct PDE-residual perturbation is small on these fixed predictions, but it does not bound the price effect of retraining with a different coefficient path. FI's generator-based carrier access is also distinct from PI's
sensor-interpolated PDE residual.

\paragraph{Finite-precision admissibility.}
An initial normalized float32 evaluation reported apparent lower-bound violations at about 2.0\% of FI and 11.0\% of FAL-only predictions. Double/higher-precision re-evaluation found no genuine interval violation at those points. This verifies those outputs, not every numerical protocol or cross-point shape condition.

\clearpage
\section{Additional Numerical Experiments}
\label{app:results}
The additional experiments examine reference accuracy, representation choices, supervision, and generalization. Each study retains its own model population, output stage, and error estimator; the formal comparisons and development studies are identified below.
\FloatBarrier
\subsection{Reference accuracy and analytical checks}
\label{app:verification}

The canonical clipped variance obeys $0.05^2\le a\le1$. In log-moneyness, the derivative of the unclipped spatial factor contains $2e^x\operatorname{sech}^2(2(e^x-1))$, which is bounded on $\R$. Clipping is Lipschitz; squaring preserves this property on the bounded nonnegative range. Thus the coefficient part of the value assumptions holds. At the strike, when clipping is inactive,
\begin{equation}
\sigma(K,T-u)=\sigma_0(1-\gamma/2+\gamma u/T).
\end{equation}
On the closure of training support this lies in $[0.075,0.5625]$ for $0\le u\le T$, strictly inside $[0.05,1]$. A smooth neighborhood of the strike therefore exists. The carrier is independent of $\beta$, while
\begin{equation}
a_* =\sigma_0^2(1-\gamma/2)^2,\qquad b_*=4\beta a_*.
\label{eq:abstar}
\end{equation}
Local profiles depend on skew even though the carrier does not. These facts concern the canonical continuous field, not every historical coefficient convention.

Table~\ref{tab:refinement} reports the primary check, which refines $(N_S,N_t)=(880,320)$ to $(1760,640)$ with the 32 ID surfaces and selected models fixed. It excludes the 82-surface parameter-support OOD archive with partly model-error-based selection. FI--PI ordering is preserved in all five aggregate metrics and all 25 seed--metric comparisons. Joint refinement does not independently identify space/time order; main results retain the base reference.

\begin{table}[htbp]
\caption{Reference-solver sensitivity on the fixed 32 ID surfaces: p95 absolute base-to-doubled-resolution discrepancies in the corresponding physical units. ATM-short uses $0<\tau\le0.05$. This is a separate 32-surface sensitivity population, not the 24-surface Greek population. Discrepancies are not certified error bounds.}
\label{tab:refinement}
\centering\small
\begin{tabular}{lrrr}
\toprule
Slice & Price & Delta & Gamma\\
\midrule
All admissible points & $4.776101\!\times\!10^{-4}$ & $5.962378\!\times\!10^{-5}$ & $6.385583\!\times\!10^{-6}$\\
Near strike & $2.145785\!\times\!10^{-3}$ & $3.096243\!\times\!10^{-4}$ & $1.103966\!\times\!10^{-4}$\\
ATM-short & $4.730651\!\times\!10^{-3}$ & $1.654274\!\times\!10^{-3}$ & $1.508516\!\times\!10^{-3}$\\
\bottomrule
\end{tabular}
\end{table}

A separate domain check moves the upper boundary from 605 to 1210 while preserving $\Delta S=0.6875$ and the time step. Near-strike price/Delta/Gamma p95 discrepancies are $5.3701\times10^{-12}$, $9.0501\times10^{-13}$, and $1.3680\times10^{-13}$; the largest checked absolute discrepancy is $4.5276\times10^{-9}$. At $\tau=0.025$, independent spatial refinement gives price/Delta/Gamma p95 differences $0.00337994$, $0.00187478$, and $0.00159277$; independent temporal refinement gives $0.00286804$, $0.000573062$, and $0.000454783$. The combined Gamma discrepancy is $0.00204970$, with no reversal of the nine checked five-seed p95 rankings. The matched refined-reference FI--PI gap for this slice is $0.06301548116729162$, so the unrounded discrepancy $0.002049703724169065$ is approximately $3.25\%$ of that gap. This is not the principal table's different Greek statistic or a continuum error bound.

A smooth uniformly elliptic manufactured first-order-skew family uses a $3\times3$ product of independent spatial/temporal refinements. Observed spatial orders for price, Delta, and Gamma are $1.9976$, $1.9989$, and $2.0039$; temporal orders are $2.0007$, $2.0002$, and $2.0018$. All 18 tested nonzero self-similar rays per quantity support the predicted scales. Median relative discrepancies of scaled profiles from their limits are $0.00256248$, $0.01000658$, and $0.0152407$. These manufactured convergence orders are not transferred to every shortest-maturity production-grid Gamma.

The pre-FAL weighted identities pass on 575 reliable points per verification path. Direct post-FAL differentiation agrees with the complete chain-rule reconstruction at 360 points, with maximum absolute discrepancy $1.7833\times10^{-15}$. Such finite identity checks do not supply the continuum norm in \eqref{eq:schauder}. The archived manufactured-field formula, all refinement grids, and finite-check point lists are needed for independent replay; they are not reconstructed from the aggregate observations.

Figure~\ref{fig:app_reference} compares the reported refinement
discrepancies by region and the separately measured manufactured-solution
orders. Different panels retain their own units and numerical populations.
\begin{figure}[!htbp]
\centering
\begingroup\fontfamily{ptm}\selectfont\mathversion{figuretimes}

\begin{tikzpicture}
\begin{groupplot}[appsmall,group style={group size=3 by 1,horizontal sep=1.6cm},
 xtick={1,2,3},xticklabels={All,Near,Short},xmin=.6,xmax=3.4,ymode=log,
 xlabel={Evaluation region},ymin=0.000003,ymax=.01]
\nextgroupplot[title={(a) Price},ylabel={p95 discrepancy (price)}]
\addplot[appRed,mark=diamond*] coordinates {(1,.0004776101)(2,.002145785)(3,.004730651)};
\nextgroupplot[title={(b) Delta},ylabel={p95 discrepancy (unitless)}]
\addplot[appBlue,mark=square*,dashed] coordinates {(1,.00005962378)(2,.0003096243)(3,.001654274)};
\nextgroupplot[title={(c) Gamma},ylabel={p95 discrepancy (1/price)}]
\addplot[appNavy,mark=triangle*,densely dotted] coordinates {(1,.000006385583)(2,.0001103966)(3,.001508516)};
\end{groupplot}
\node[font=\scriptsize,align=center,text=appGray] at (6,-1.6)
{Separate manufactured check: spatial orders $1.9976,\ 1.9989,\ 2.0039$; temporal orders $2.0007,\ 2.0002,\ 2.0018$.};
\end{tikzpicture}

\endgroup
\caption{Reference-solver refinement on the fixed 32 ID surfaces, from
$(N_S,N_t)=(880,320)$ to $(1760,640)$, using Table~\ref{tab:refinement}.
``All'' denotes admissible points, ``Near'' the near-strike region, and
``Short'' the ATM-short subset $0<\tau\le0.05$. Each panel uses its own
physical units and a logarithmic vertical axis. Points are pooled p95
base-to-refined discrepancies, not sample means; no confidence intervals
are available from these summaries. Lines connect categories only.
The annotation reports a separate manufactured-solution experiment in
price/Delta/Gamma order. Neither check certifies continuum error. }
\label{fig:app_reference}
\end{figure}

\FloatBarrier
\subsection{Principal comparisons, component ablations, and representation controls}
\label{app:attribution}
Table~\ref{tab:classical} reports means only for compact presentation.
Tables~\ref{tab:principal} and~\ref{tab:ablation} retain the corresponding
five-seed sample standard deviations and full reported precision.
The repeated FI row uses the principal-comparison values; carrier-only
price results are stated below. The E0 column uses the refitted baseline;
the version-specific replay scope is described in
Appendices~\ref{app:training} and~\ref{app:context}.

\begin{table}[htbp]
\caption{Principal five-seed comparison, mean $\pm$ sample SD. Price: 32 held-out ID surfaces. Greeks: nonterminal near-strike FD errors on the fixed 24-surface L3 subset. PDE residual RMS uses the common interior FD evaluator
(98,592 points/seed). PP is excluded. Lower is better; bold marks the lowest mean in each row.}
\label{tab:principal}
\centering\small
\begin{tabular}{lrrr}
\toprule
Metric & DeepONet & PI-DeepONet & FI-DeepONet\\
\midrule
Price $\RelL$ & $0.0059854\pm0.00043$ & $0.0072437\pm0.0012$ & $\mathbf{0.0009645}\pm0.000049$\\
Near price p95 & $1.7967\pm0.169$ & $2.3338\pm0.33$ & $\mathbf{0.15794}\pm0.0073$\\
Delta p95 & $0.15183\pm0.0092$ & $0.15474\pm0.021$ & $\mathbf{0.024518}\pm0.0018$\\
Gamma p95 & $0.033225\pm0.0015$ & $0.035549\pm0.0021$ & $\mathbf{0.0043378}\pm0.000035$\\
PDE Residual RMS & $0.046102\pm0.0075$ & $0.010494\pm0.0015$ & $\mathbf{0.0070251}\pm0.00022$\\
\bottomrule
\end{tabular}
\end{table}

\begin{table}[htbp]
\caption{Principal component ablation (five seeds, mean $\pm$ SD). ``Normalized'' denotes contract normalization; RA is the carrier-residual representation. The complete FI predictor has normalization, RA, and FAL.}
\label{tab:ablation}
\centering\small
\begin{tabular}{lrr}
\toprule
Configuration & Global $\RelL$ & Near-strike price p95\\
\midrule
DeepONet & $0.0059853982\pm0.000433$ & $1.7967389\pm0.1692$\\
Normalized DeepONet & $0.0076\pm0.0013$ & $2.2553\pm0.3486$\\
Normalized + soft financial penalties & $0.0075\pm0.0013$ & $2.2125\pm0.3474$\\
Normalized + FAL & $0.0024\pm0.000488$ & $0.5277\pm0.0581$\\
Normalized + FAL + PP & $0.0023\pm0.000498$ & $0.4581\pm0.0688$\\
FI-DeepONet & $0.0009645\pm0.0000488$ & $0.15794\pm0.0073$\\
FI-DeepONet-PP & $0.001043\pm0.0000455$ & $0.1579\pm0.0073$\\
\bottomrule
\end{tabular}
\end{table}
\begin{table}[htbp]
\caption{Payoff-kink treatments on a common three-seed population before PP, distinct from the principal five-seed population. Entries are means. Lower is better; bold marks the lowest mean in each column.}
\label{tab:kink}
\centering\small
\begin{tabular}{lrrrr}
\toprule
Treatment & Global $\RelL$ & Near price p95 & Near Delta p95 & Near Gamma p95\\
\midrule
Coordinate enlargement &0.001795&0.348724&0.065799&0.025100\\
Near-strike oversampling &0.001811&0.364922&0.085414&0.036080\\
RA & \textbf{0.0009050} & \textbf{0.145524} & \textbf{0.020482} & \textbf{0.003367} \\
\bottomrule
\end{tabular}
\end{table}
The kink intervention in Table~\ref{tab:kink} compares price representation with alternative query representation and sampling. Its results concern the tested controls and do not establish superiority over adaptive-coordinate or adaptive-sampling architectures in general.

On the principal 32-surface ID price population the deterministic carrier alone has global $\RelL=0.00688952$ and near-strike price p95 $0.392665$. The full predictor improves these by $86.00\%$ and $59.78\%$. Carrier-only is a diagnostic rather than a trained baseline; the matched learned-model ablation below separately identifies the carrier and maturity-rescaling contributions. On the experiment-specific ATM-short price mask, the analytic carrier is already highly accurate and the learned correction raises $\RelL$ by approximately $40.46\%$; the improvement is therefore not uniform over every maturity slice. Together with the global-rank result below, these observations prevent interpreting the asymptotic construction as proof that the learned correction is globally simpler or that $\alpha=1$ is training-optimal. Historical including-terminal Greek attribution is retained below rather than mixed into nonterminal results.

\subsubsection{Matched carrier--rescaling mechanism ablation}
\label{app:mechanism_ablation}

Within the carrier arms, the pre-registered five-seed ablation varies the reconstruction exponent in
\[
z_{\theta,\alpha}=C^{\mathrm{ref}}+g_\alpha(\chi)R_\theta,
\]
with $g_0(\chi)=1$ and $g_\alpha(\chi)=\chi^\alpha$ for $\alpha>0$, as in \eqref{eq:alpha_ablation}. B0 is a separate direct-price control without the carrier. Table~\ref{tab:alpha_prefal} reports the pre-FAL nonterminal mechanism metrics.

\begin{table}[htbp]
\caption{Matched carrier--rescaling ablation before FAL; five-seed mean $\pm$ sample SD. B0 is the normalized direct-price control, and A0/A05/A1/A15 use the same carrier with $\alpha=0,\frac12,1,\frac32$. Lower is better.}
\label{tab:alpha_prefal}
\centering\small
\resizebox{\linewidth}{!}{
\begin{tabular}{lrrrr}
\toprule
Arm &
Global price $\RelL$ &
ATM-short price $\RelL$ &
Near Delta p95 &
Near Gamma p95\\
\midrule
B0
& $0.04680\pm0.00996$
& $0.1752\pm0.0228$
& $0.1362\pm0.0194$
& $0.03690\pm0.00274$\\
A0
& $0.01975\pm0.00117$
& $0.04846\pm0.01045$
& $0.02388\pm0.00335$
& $0.005722\pm0.000121$\\
A05
& $0.01345\pm0.00094$
& $0.01974\pm0.00350$
& $\mathbf{0.02052}\pm0.00233$
& $0.005416\pm0.0000547$\\
A1
& $0.01292\pm0.000520$
& $0.01270\pm0.000747$
& $0.02508\pm0.00174$
& $0.005331\pm0.0000529$\\
A15
& $\mathbf{0.01248}\pm0.000744$
& $\mathbf{0.01028}\pm0.000216$
& $0.03025\pm0.00130$
& $\mathbf{0.005232}\pm0.00000563$\\
\bottomrule
\end{tabular}}
\end{table}

Adding the strike-matched carrier without maturity rescaling (B0$\rightarrow$A0) lowers the mean pre-FAL global price, Delta, and Gamma errors by approximately $57.8\%$, $82.5\%$, and $84.5\%$, respectively; the corresponding paired directions improve in all five seeds. Thus the carrier has an empirical contribution beyond normalization and FAL even before the maturity gate is introduced.

Adding the linear gate (A0$\rightarrow$A1) lowers global price $\RelL$ by approximately $34.6\%$ and ATM-short price $\RelL$ by $73.8\%$, while near-strike Gamma p95 decreases by about $6.8\%$. Near-strike Delta p95 instead increases by about $5.0\%$. Hence maturity rescaling has a clear but metric-dependent finite-budget effect.

For $\alpha>0$, the raw correction vanishes at maturity by construction. The pre-FAL terminal maximum error is $5.03\pm0.99$ for A0 but approximately $2.26\times10^{-5}$ for A05, A1, and A15. This terminal behavior is reported separately from the nonterminal mechanism metrics.

Table~\ref{tab:alpha_postfal} evaluates the same checkpoints after the common FAL and terminal-preserving output rule.

\begin{table}[htbp]
\caption{Matched carrier--rescaling ablation after FAL and the common terminal-preserving inference rule; five-seed mean $\pm$ sample SD. These values use a common evaluation stage for all arms and are distinct from the historical component table above. Lower is better.}
\label{tab:alpha_postfal}
\centering\small
\resizebox{\linewidth}{!}{
\begin{tabular}{lrrrrr}
\toprule
Arm &
Price $\RelL$ &
Near price p95 &
Delta p95 &
Gamma p95 &
FD PDE RMS\\
\midrule
B0
& $0.002336\pm0.000493$
& $0.4581\pm0.0688$
& $0.08459\pm0.01166$
& $0.02535\pm0.00265$
& $0.02577\pm0.00291$\\
A0
& $0.001179\pm0.000234$
& $0.1616\pm0.0168$
& $0.02289\pm0.00394$
& $0.004527\pm0.000125$
& $0.008158\pm0.00131$\\
A05
& $\mathbf{0.0008339}\pm0.0000657$
& $\mathbf{0.1424}\pm0.0162$
& $\mathbf{0.01959}\pm0.00249$
& $0.004354\pm0.0000506$
& $\mathbf{0.006339}\pm0.000232$\\
A1
& $0.0009645\pm0.0000488$
& $0.1579\pm0.00730$
& $0.02452\pm0.00178$
& $\mathbf{0.004338}\pm0.0000353$
& $0.007025\pm0.000224$\\
A15
& $0.001192\pm0.0000705$
& $0.1677\pm0.00768$
& $0.02982\pm0.00132$
& $0.004413\pm0.0000859$
& $0.007912\pm0.000220$\\
\bottomrule
\end{tabular}}
\end{table}

No single exponent minimizes every finite-budget metric. A15 has the lowest pre-FAL global and ATM-short price errors and the lowest pre-FAL Gamma p95, whereas A05 has the lowest pre-FAL Delta p95 and the lowest post-FAL price, near-price, Delta, and PDE errors; A1 has the lowest post-FAL Gamma p95. These observations do not contradict Proposition~\ref{prop:profile}: $\alpha=1$ is critical for a bounded, nondegenerate value target in the stated asymptotic regime, not a theorem of metric-wise training optimality.

Several arm rankings change between the pre- and post-FAL stages. FAL therefore interacts nontrivially with the residual parameterization, and post-FAL differences are not attributed to the exponent alone.

Figure~\ref{fig:app_exponents} shows the post-FAL metric trade-offs
with the reported sample standard deviations. These error bars describe
variation across five seeds; they are not confidence intervals and do not
establish significance of the small differences between exponents.
\begin{figure}[!htbp]
\centering
\begingroup\fontfamily{ptm}\selectfont\mathversion{figuretimes}

\begin{tikzpicture}
\begin{groupplot}[appsmall,group style={group size=3 by 1,horizontal sep=1.6cm},
 xmin=-.12,xmax=1.62,xtick={0,.5,1,1.5},xticklabels={$0$,$1/2$,$1$,$3/2$},xlabel={Exponent $\alpha$},
 error bars/y dir=both,error bars/y explicit]
\nextgroupplot[title={(a) Global price},ylabel={Price $\RelL$ ($\times10^{-3}$)},ymin=.6,ymax=1.55,ytick={.8,1,1.2,1.4}]
\addplot[appRed,mark=diamond*,error bars/.cd,y dir=both,y explicit] coordinates
{(0,1.179)+-(0,.234)(.5,.8339)+-(0,.0657)(1,.9645)+-(0,.0488)(1.5,1.192)+-(0,.0705)};
\nextgroupplot[title={(b) Delta},ylabel={Near-strike FD p95},ymin=.014,ymax=.034,ytick={.015,.02,.025,.03},yticklabels={.015,.020,.025,.030}]
\addplot[appBlue,mark=square*,dashed,error bars/.cd,y dir=both,y explicit] coordinates
{(0,.02289)+-(0,.00394)(.5,.01959)+-(0,.00249)(1,.02452)+-(0,.00178)(1.5,.02982)+-(0,.00132)};
\nextgroupplot[title={(c) Gamma},ylabel={FD p95 ($\times10^{-3}$/price)},ymin=4.2,ymax=4.75,ytick={4.2,4.4,4.6}]
\addplot[appNavy,mark=triangle*,densely dotted,error bars/.cd,y dir=both,y explicit] coordinates
{(0,4.527)+-(0,.125)(.5,4.354)+-(0,.0506)(1,4.338)+-(0,.0353)(1.5,4.413)+-(0,.0859)};
\end{groupplot}
\end{tikzpicture}

\endgroup
\caption{Post-FAL exponent ablation with the common terminal-preserving output
rule, from Table~\ref{tab:alpha_postfal}. Points and error bars are five-seed
means and sample standard deviations for A0, A05, A1, and A15.
The lowest mean price and Delta errors occur at $\alpha=1/2$, while the
lowest mean Gamma error occurs at $\alpha=1$. Vertical axes are zoomed;
Greek units differ. Connecting lines are visual guides between the four tested
exponents, not an interpolated optimization landscape. The critical
asymptotic value scaling at $\alpha=1$ is not a claim of finite-budget
optimality for every metric.}
\label{fig:app_exponents}
\end{figure}
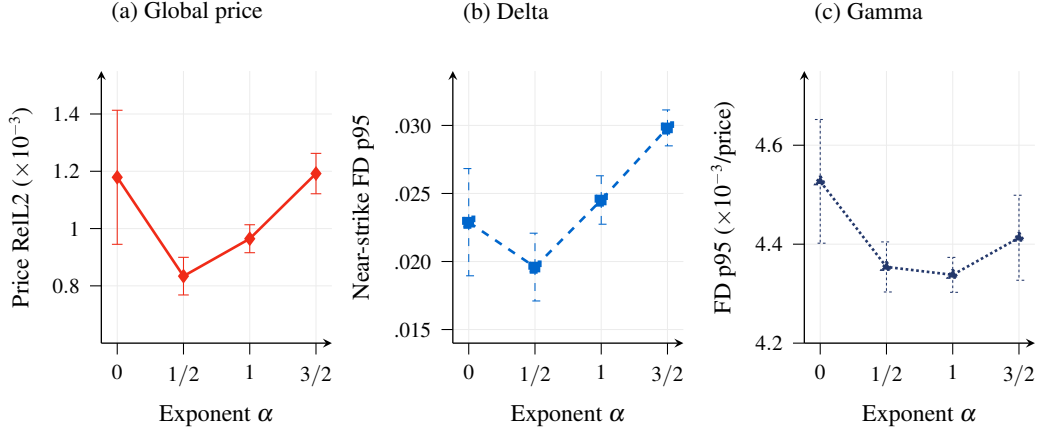

The global linear-reconstruction experiment fits the mean and basis using only 512 training surfaces, with a common 3,240-point nonterminal output grid and uniform evaluator weights. It compares $v_{\rm full}=V/K$, $w_{\rm corr}=(V-\Cref)/K$, and $r_{\rm corr}=w_{\rm corr}/\chi$. It does not replace this common-coordinate problem with irregular contract-dependent ATM subsets.

\begin{table}[htbp]
\caption{Global linear reconstruction: energy ranks and held-out error ratios relative to normalized full price at the same rank. Ratios below one would mean easier reconstruction.}
\label{tab:rank}
\centering\small
\begin{tabular}{lrrrrr}
\toprule
Representation & 99\% rank & 99.9\% rank & $p=8$ & $p=32$ & $p=128$\\
\midrule
Full normalized price $v_{\rm full}$ &1&3&1&1&1\\
Pricing correction $w_{\rm corr}$ &8&16&31.666&37.484&60.457\\
Rescaled correction $r_{\rm corr}$ &13&29&62.013&133.118&291.398\\
\bottomrule
\end{tabular}
\end{table}
For a square-integrable centered Hilbert-space random output $R$, with covariance eigenvalues $\lambda_1\ge\lambda_2\ge\cdots$, the optimal rank-$p$ mean-square reconstruction error is
\begin{equation}
\inf_{\dim E=p}\E\norm{R-P_ER}^2=\sum_{j>p}\lambda_j.
\label{eq:covariance}
\end{equation}
Indeed, the error is $\operatorname{tr}(M_R)-\operatorname{tr}(P_EM_R)$, and the latter trace is maximized by leading eigenvectors. This standard identity \citep{r64} does not order the spectra of full price and correction. Table~\ref{tab:rank} shows that both correction representations require more modes to reach the reported variance thresholds and have higher held-out reconstruction error at the tested ranks. This does not support reduced global linear complexity as the explanation for the observed gains. It does not contradict the exact correction equation or local short-time cancellation; the oracle linear reconstruction error also does not determine the error of a trained branch--trunk network.

\FloatBarrier
\subsection{PDE loss and limited data}
\label{app:objectives}

The five-seed factorial uses a common normalized carrier-residual parameterization.
The price-only arm and the arm with an added PDE loss do not use FAL, whereas the corresponding FAL arms activate the output map
during continuation. The financial factor therefore changes the output
map without adding a soft financial-penalty loss. The data-only control is not the principal
unnormalized DeepONet, and none of these arms redefines the principal FI model.
Separate development studies vary PDE-loss weights, training-gradient calibration,
late learning-rate instability, and registered-mixture, uniform, and PDE-residual-adaptive collocation.

For the formal factorial, all five seeds use paired initialization and supervised schedules, with validation normalized-price MSE checkpoint selection. The PDE-loss arms share 1,024 collocation pairs per update, uniformly sampled over $m\in[0.01,2.75)$ and $\bar t\in[0,0.975)$. Their final PDE-loss coefficient is $0.865302726925$: the first 10\% of updates have zero weight, the next 20\% ramp linearly, and the remaining 70\% use the fixed coefficient. All arms reduce the learning rate tenfold at update 3,800, retaining that multiplier in continuation. These settings differ from the canonical comparator.

The separate stabilized collocation study compares a registered random mixture, uniform draws, and a PDE-residual-adaptive bank under the carrier-residual parameterization using development data. The registered mixture allocates probabilities $0.50/0.25/0.25$ to global, near-strike ($m\in[0.95,1.05)$), and late-time ($\bar t\in[0.80,0.975)$) draws. All three schemes retain the cutoff $\bar t<0.975$; none resolves the omitted shortest-maturity PDE region of the canonical comparator. The full scheme comparison uses one seed; the selected uniform scheme has a separate second-seed development check. These results are not a matched collocation-only rerun of canonical PI.

\begin{table}[htbp]
\caption{Paired effects in the normalized carrier-residual factorial. FAL is an output-map change; it adds no soft financial loss. Entries are mean \emph{within-seed} percentage changes, with improved seeds out of five in parentheses. Negative is better. PDE residual RMS and Greeks are evaluated by AD, whereas
Table~\ref{tab:principal} uses the principal FD evaluation for the PDE metric
and FD p95 for Greeks.}
\label{tab:factorial}
\centering\small
\begin{tabular}{lrrrr}
\toprule
Comparison & Price $\RelL$ & PDE residual RMS & Delta MAE & Gamma MAE\\
\midrule
FAL vs price control & $-33.95\ (5)$ & $-8.38\ (5)$ & $-37.69\ (5)$ & $-4.80\ (5)$\\
PDE loss vs price control & $-2.34\ (3)$ & $-34.18\ (5)$ & $-7.87\ (5)$ & $-5.85\ (5)$\\
PDE loss + FAL vs FAL & $+34.38\ (0)$ & $-30.67\ (5)$ & $-3.83\ (4)$ & $-3.33\ (4)$\\
\bottomrule
\end{tabular}
\end{table}

Table~\ref{tab:objectivemeans} reports the corresponding arm means,
which are distinct from the paired percentage-change estimator above.

\begin{table}[htbp]
\caption{Five-seed means for the normalized carrier-residual factorial: $C$ uses price supervision, $F$ adds FAL, $P$ adds the PDE loss, and $FP$ combines them. No arm adds a soft financial loss. Lower is better; bold marks the lowest mean in each column. These estimators differ from the principal comparison. PDE residual RMS is evaluated by AD in this factorial.}
\label{tab:objectivemeans}
\centering\small\setlength{\tabcolsep}{5pt}
\begin{tabular}{lrrrrr}
\toprule
Objective & Price $\RelL$ & PDE residual RMS & Delta MAE & Gamma MAE & Financial\\
\midrule
$C$ &0.001576342&0.006682306&0.005316740&0.000777047&0.002907677\\
$F$ &$\mathbf{0.001036124}$&0.006115663&0.003305758&0.000739836&$\mathbf{0.000087879}$\\
$P$ &0.001525491&0.004392578&0.004876109&0.000731263&0.003298899\\
$FP$ &0.001389139&$\mathbf{0.004239599}$&$\mathbf{0.003172085}$&$\mathbf{0.000714711}$&0.000093171\\
\bottomrule
\end{tabular}
\end{table}

The ratio of means and the mean of paired within-seed changes are different estimators. For $FP$ versus $F$, the former gives a $34.07\%$ price increase and a $30.68\%$ PDE-residual RMS decrease; the latter gives $34.38\%$ and $30.67\%$ in Table~\ref{tab:factorial}. The qualitative trade-off agrees.

In the separate development study with 25\% data, the control has price $\RelL=0.0013530935$, versus $0.0011967161$ with the PDE loss; PDE residual RMS is $0.0059653614$ versus $0.0037237662$. The reported improvements are $11.557\%$ for price, $37.577\%$ for PDE residual RMS, $11.039\%$ for Delta MAE, and $12.052\%$ for Gamma MAE. Improved-seed counts for price, PDE residual, Delta, and Gamma are 4/5, 5/5, 4/5, and 5/5. These observations were used during development and are not independent confirmation. Another adaptation with a reinitialized optimizer lowered PDE residual RMS while increasing price and financial errors and did not change the selected final model.

Table~\ref{tab:scarcity} reports global price errors across six data fractions; Table~\ref{tab:scarcitytails} records the corresponding near-strike price and Gamma tails.

\begin{table}[htbp]
\caption{Six-level limited-data study: mean global $\RelL$ over five seeds, before PP. Additions refer to Normalized DeepONet; this study does not redefine the final RA model. Bold marks the lowest mean within each data fraction.}
\label{tab:scarcity}
\centering\small
\begin{tabular}{lrrrr}
\toprule
Data retained & Normalized & + balanced penalties & + FAL & + FAL + penalties\\
\midrule
100\% & 0.007066 & 0.008253 & \textbf{0.002648} & 0.002684 \\
50\% & 0.007299 & 0.008733 & \textbf{0.002752} & 0.002883 \\
25\% & 0.007526 & 0.008868 & \textbf{0.003306} & 0.003421 \\
12.5\% & 0.009351 & 0.010642 & 0.004504 & \textbf{0.004489} \\
5\% & 0.016960 & 0.017343 & 0.008249 & \textbf{0.007894} \\
1\% & 0.061610 & 0.061183 & 0.030917 & \textbf{0.030302} \\
\bottomrule
\end{tabular}
\end{table}
Twenty deterministic training batches calibrate the aggregate penalty/data gradient ratio to $[0.8,1.2]$; weights are then fixed. This differs from arbitrary hand weighting or a continuously adaptive scheme. FAL benefits all six fractions. Standalone balanced penalties are unfavorable at most fractions; adding them to FAL gives modest global gains at the smallest fractions and improves near-strike tails.

\begin{table}[htbp]
\caption{Near-strike tails under limited data. Each pair compares FAL with FAL plus balanced penalties on Normalized DeepONet. Bold marks the lower value within each metric pair and data fraction.}
\label{tab:scarcitytails}
\centering\small
\begin{tabular}{lrrrr}
\toprule
&\multicolumn{2}{c}{Near price p95}&\multicolumn{2}{c}{Near Gamma p95}\\
Data retained & FAL & FAL + penalties & FAL & FAL + penalties\\
\midrule
100\% & 0.601981 & \textbf{0.566507} & 0.027989 & \textbf{0.026772} \\
50\% & 0.616515 & \textbf{0.561270} & 0.028161 & \textbf{0.025659} \\
25\% & 0.836214 & \textbf{0.796390} & 0.028038 & \textbf{0.026413} \\
12.5\% & 1.182752 & \textbf{1.104384} & 0.031705 & \textbf{0.028233} \\
5\% & 2.246798 & \textbf{2.138333} & 0.035645 & \textbf{0.031917} \\
1\% & 6.899662 & \textbf{6.523885} & 0.071334 & \textbf{0.069443} \\
\bottomrule
\end{tabular}
\end{table}
At both 5\% and 1\% data, FAL-plus-penalty global error improves
in four of five seeds, not all. This supports a benefit from the added
penalties under limited data, not a recommendation that all FI models
should contain soft penalties.

Table~\ref{tab:financial} retains the separate financial aggregates and their paired changes.

The threshold for violation frequency is not used as a mean-severity threshold. Small aggregate values after PP are not continuous-surface certificates. The QDO selection score in Appendix~\ref{app:qselection} is a third, distinct composite and is not substituted for these financial magnitudes.

Figure~\ref{fig:app_supervision} keeps the paired PDE-loss effects
and limited-data study in separate panels. Its displayed points summarize
existing experiments; the connecting lines in the limited-data panel only
guide the eye between tested fractions.
\begin{figure}[!htbp]
\centering
\begingroup\fontfamily{ptm}\selectfont\mathversion{figuretimes}

\begin{tikzpicture}
\begin{groupplot}[apphalf,group style={group size=2 by 1,horizontal sep=1.55cm}]
\nextgroupplot[title={(a) Adding the PDE loss},xmin=-45,xmax=45,xtick={-40,-20,0,20,40},
 xlabel={Mean within-seed change (\%)},ymin=.5,ymax=4.5,ytick={1,2,3,4},
 yticklabels={Gamma MAE,Delta MAE,PDE RMS,Price $\RelL$},legend style={at={(.5,-.31)},anchor=north,legend columns=1}]
\addplot[gray,densely dashed,no marks,forget plot] coordinates {(0,.5)(0,4.5)};
\addplot[only marks,appBlue,mark=square] coordinates {(-5.85,1.08)(-7.87,2.08)(-34.18,3.08)(-2.34,4.08)};\addlegendentry{PDE loss vs price control}
\addplot[only marks,appRed,mark=diamond*] coordinates {(-3.33,.92)(-3.83,1.92)(-30.67,2.92)(34.38,3.92)};\addlegendentry{PDE loss + FAL vs FAL}
\nextgroupplot[title={(b) Separate limited-data study},xmode=log,ymode=log,xmin=.8,xmax=125,ymin=.002,ymax=.09,
 xtick={1,5,25,100},xticklabels={1,5,25,100},xlabel={Training data retained (\%)},ylabel={Global price $\RelL$},
 legend style={at={(.5,-.31)},anchor=north,legend columns=1}]
\addplot[appGray,dashdotted,mark=*] coordinates {(1,.061610)(5,.016960)(12.5,.009351)(25,.007526)(50,.007299)(100,.007066)};\addlegendentry{Normalized DeepONet}
\addplot[appBlue,dashed,mark=square] coordinates {(1,.030917)(5,.008249)(12.5,.004504)(25,.003306)(50,.002752)(100,.002648)};\addlegendentry{Normalized + FAL}
\addplot[appRed,mark=diamond*] coordinates {(1,.030302)(5,.007894)(12.5,.004489)(25,.003421)(50,.002883)(100,.002684)};\addlegendentry{FAL + balanced penalties}
\end{groupplot}
\end{tikzpicture}

\endgroup
\caption{Distinct PDE-loss and limited-data studies. (a) Mean paired
within-seed percentage changes from Table~\ref{tab:factorial}; negative is
better. The normalized-RA factorial uses AD PDE residuals and Greek MAEs.
Adding the PDE loss to FAL lowers the residual but raises price error.
(b) Five-seed mean price errors from Table~\ref{tab:scarcity}; all three
curves start from Normalized DeepONet, not the final RA model.
Both axes in (b) are logarithmic. The source tables do not provide
uncertainty for these plotted summaries, so no error bars are inferred.
Connecting lines in (b) only join tested fractions. The final FI objective
remains price-only.}
\label{fig:app_supervision}
\end{figure}
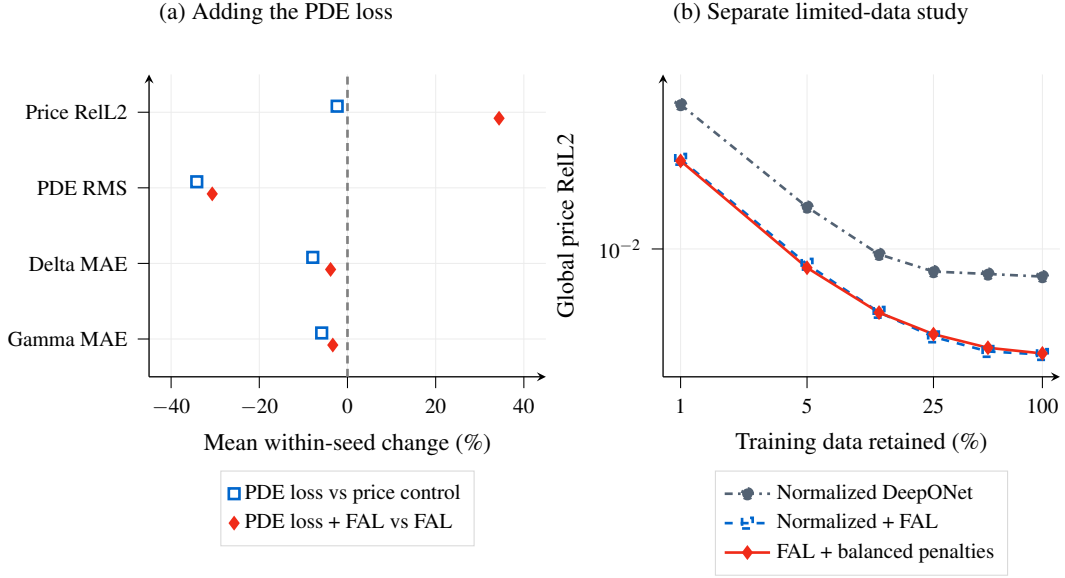

\begin{table}[htbp]
\caption{Protocol-specific financial aggregates retained for traceability. The principal seven-rule sum and factorial nine-rule mean have different definitions and mixed units.}
\label{tab:financial}
\centering\small
\begin{tabularx}{\linewidth}{Yr}
\toprule
Configuration / comparison & Value\\
\midrule
Principal PI & $0.7917\pm0.1426$\\
Principal FI & $(7.236\pm0.237)\times10^{-4}$\\
Principal FI-DeepONet-PP & $(8.088\pm2.314)\times10^{-10}$\\
Factorial data control &0.002907677\\
Factorial FAL &0.000087879\\
Factorial PDE loss &0.003298899\\
Factorial combined &0.000093171\\
FAL vs control & $-96.94\%$; 5/5 improve\\
PDE loss vs control & $+14.52\%$; 1/5 improve\\
Combined vs FAL & $+6.03\%$; 0/5 improve\\
\bottomrule
\end{tabularx}
\end{table}

\FloatBarrier
\subsection{Matched PI collocation sensitivity}
\label{app:pifairness}
We predeclared and hashed a single collocation-only sensitivity before training. It uses the five canonical PI seeds, 512 PDE points per surface per update, and a fixed equal mixture of global samples with full nonterminal time support and enriched near-strike/short-maturity samples satisfying $|S/K-1|\le 0.05$ and $0<\tau/T\le 0.10$, with float32 endpoint guards fixed in advance. Architecture, normalization, supervised/terminal/boundary draws, the four unit-weight losses, Adam settings, the $5{,}000+2{,}000$ update budget, validation population, and phase-wise validation checkpoint selection are unchanged. No held-out ID/OOD outcome is used to tune the sampler or select checkpoints. All five runs completed the locked 7,000-update protocol before held-out evaluation.

Table~\ref{tab:pifairness} reports the same five metrics and populations as Table~\ref{tab:principal}. Relative to canonical PI, the sensitivity lowers mean price $\RelL$ by $1.59\%$, near-strike price p95 by $5.67\%$, Delta p95 by $7.61\%$, and Gamma p95 by $0.40\%$, while increasing PDE residual RMS by $18.26\%$. The corresponding per-seed improvement counts are $4/5$, $3/5$, $4/5$, $1/5$, and $1/5$, respectively, so the small mean Gamma change in particular is not a seed-stable improvement. The sensitivity-to-FI mean ratios are $7.39$, $13.94$, $5.83$, $8.16$, and $1.77$ for those five metrics. No importance weights map the enriched collocation samples back to the canonical uniform collocation measure; therefore the intervention changes the empirical measure underlying the PDE loss and its spatial/temporal coverage, even though the four explicit loss coefficients remain at unit weight. The reported PDE residual RMS, however, uses the same principal FD evaluator as Table~\ref{tab:principal}. The FI--PI gap persists under this fixed predeclared
near-strike/short-maturity collocation intervention.
A separate upper-boundary target sensitivity produces the same
qualitative conclusion. Because this is one sampler and reuses the held-out populations, it is robustness evidence rather than independent confirmation, a universal PI-fairness result, or a characterization of all possible physics-informed collocation objectives.

\begin{table}[htbp]
\caption{Predeclared five-seed PI collocation sensitivity. Entries are mean $\pm$ sample SD. The historical canonical PI and FI columns are the same models/populations as Table~\ref{tab:principal}. Price p95 includes terminal points; Delta/Gamma p95 exclude the terminal line. Lower is better; bold marks the lowest mean in each row.}
\label{tab:pifairness}
\centering\small
\begin{tabular}{lrrr}
\toprule
Metric & Canonical PI & PI collocation sensitivity & FI-DeepONet\\
\midrule
Price $\RelL$ &$0.0072437\pm0.00123$&$0.0071289\pm0.00138$&$\mathbf{0.00096450}\pm0.0000488$\\
Near price p95 &$2.3338\pm0.327$&$2.2014\pm0.298$&$\mathbf{0.15794}\pm0.00730$\\
Delta p95 ($\tau>0$) &$0.15474\pm0.0205$&$0.14296\pm0.0162$&$\mathbf{0.024518}\pm0.00178$\\
Gamma p95 ($\tau>0$) &$0.035549\pm0.00205$&$0.035406\pm0.00275$&$\mathbf{0.0043378}\pm0.0000353$\\
PDE residual RMS &$0.010494\pm0.00151$&$0.012410\pm0.00223$&$\mathbf{0.0070251}\pm0.000224$\\
\bottomrule
\end{tabular}
\end{table}

Table~\ref{tab:pifairnessseeds} exposes the five selected-state results behind the collocation-sensitivity summary.

\begin{table}[htbp]
\caption{Per-seed PI collocation-sensitivity metrics under the locked intervention. These are the selected checkpoints after the unchanged phase-wise validation rule; lower is better.}
\label{tab:pifairnessseeds}
\centering\scriptsize
\begin{tabular}{rrrrrr}
\toprule
Seed & Price $\RelL$ & Near price p95 & Delta p95 & Gamma p95 & PDE residual RMS\\
\midrule
2026 &0.00584792&1.90532&0.126435&0.0330828&0.0108313\\
3407 &0.00700464&2.25867&0.135171&0.0357557&0.0109860\\
5201 &0.00808553&2.42720&0.164153&0.0375136&0.0149560\\
7713 &0.00893305&2.53432&0.156070&0.0385162&0.0147449\\
9109 &0.00577322&1.88157&0.132980&0.0321622&0.0105338\\
\bottomrule
\end{tabular}
\end{table}

\FloatBarrier
\subsection{Complete parameter-support OOD population and secondary selected subset}
\label{app:ood}
\begin{table}[htbp]
\caption{Proposed and realized normalized parameter-support OOD distances beyond the nearest training boundary, divided by the corresponding training width. The joint row is multi-axis.}
\label{tab:realizedsupport}
\centering\small
\begin{tabular}{lrrr}
\toprule
Regime & Proposed near/mid/far & Realized min/median/max & Count\\
\midrule
Low $K$ &.037/.113/.188&.042/.109/.179&32\\
High $K$ &.037/.175/.312&.043/.174/.308&32\\
High $r$ &.077/.231/.385&.079/.247/.383&32\\
High $\sigma_0$ &.086/.257/.429&.098/.247/.428&32\\
Negative $\beta$ &.033/.100/.167&.035/.051/.065&32\\
Positive $\beta$ &.033/.083/.133&.036/.079/.129&32\\
Negative $\gamma$ &.053/.132/.211&.054/.126/.193&32\\
Positive $\gamma$ &.053/.132/.211&.061/.112/.210&32\\
Joint &---&Multi-axis&64\\
\bottomrule
\end{tabular}
\end{table}
The primary parameter-support OOD evaluation compares refitted E0 and FI on the complete canonical-clipped 320-surface population. Across five seeds, mean pooled $\RelL$ $\pm$ sample SD is $0.1974974\pm0.0269031$ for E0 and $0.00151740\pm0.0000751062$ for FI. Table~\ref{tab:classical}C reports these means at reduced display
precision; the sample standard deviations are retained here.

The archived OOD reference uses $S_{\max}=825$, $N_S=1200$, $N_t=320$, and four Rannacher half steps, rather than the principal ID solver domain. Each seed pools all $320\times41\times81$ price points, including terminal and boundary points. FI uses the same selected weights as in the principal study but retains the historical OOD output adapter; the corrected terminal-preserving ID adapter is a different evaluation convention. These results therefore support a matched E0--FI comparison on the OOD population, not a matched OOD/ID degradation ratio. This complete-population comparison does not isolate the OOD contributions of normalization, FAL, or RA; intermediate-architecture OOD values below come from the secondary selected archive and are not substituted for a complete-320 ablation. The full initial OOD implementation used the same \emph{unclipped} coefficient in sensors and reference labels, making inputs/labels internally consistent but different from canonical training/ID clipping. Table~\ref{tab:clipping} restores canonical clipping on identical parameter rows with fixed models. Clipping affects all 32 negative-skew, 15 joint, and one positive-skew surface. FI's negative-skew result changes about $11.369\%$, more than the $1.4553\%$ aggregate change. Whole-population and affected-surface statistics are reported separately.

\begin{table}[htbp]
\caption{Historical complete 320-row clipping sensitivity, fixed pre-refit trained models and identical parameter rows.}
\label{tab:clipping}
\centering\small
\begin{tabular}{lrrr}
\toprule
Model & Unclipped & Canonical clipped & Relative change\\
\midrule
DeepONet (historical E0) & 0.197656594 & 0.197654514 & $-0.0011\%$\\
FI-DeepONet &0.0014956348&0.0015174005&$+1.4553\%$\\
\bottomrule
\end{tabular}
\end{table}

The following 82-surface L3 archive is secondary. Selection depended partly on model errors; its DeepONet aggregate $0.2636$ must not be used as the complete-320 error. It does not establish a representative parameter-support OOD comparison for all intermediate architectures.
Table~\ref{tab:ood82} preserves the regime-level values from this selected archive for traceability.

\begin{table}[htbp]
\caption{Secondary OOD archive: price $\RelL$ on 82 selected higher-resolution surfaces. Columns with FAL are normalized; this is not the complete-population primary evidence.}
\label{tab:ood82}
\centering\small
\begin{tabular}{lrrrrr}
\toprule
Regime & DeepONet & FAL & FAL + PP & FI & FI + PP\\
\midrule
Aggregate &0.2636&0.0042&0.0041&0.0015&0.0014\\
High strike &0.2006&0.0081&0.0077&0.0032&0.0038\\
High rate &0.0074&0.0035&0.0035&0.000748&0.000883\\
High volatility &0.0138&0.0091&0.0089&0.0017&0.0017\\
Joint shift &0.1934&0.0108&0.0106&0.0038&0.0039\\
Low strike &0.3822&0.000590&0.000564&0.000833&0.000381\\
Negative skew &0.0062&0.0025&0.0024&0.0010&0.0011\\
Negative term &0.0050&0.0022&0.0020&0.000789&0.000939\\
Positive skew &0.0060&0.0028&0.0027&0.0014&0.0015\\
Positive term &0.0054&0.0017&0.0017&0.000645&0.000773\\
\bottomrule
\end{tabular}
\end{table}
The proposed shifts are moderate; realized negative-skew excursions are especially smaller than their proposal extreme. Substituting $S=Km$ into the family removes explicit $K$ dependence from its volatility formula, so strike shifts partly test scale handling. Skew/term shifts change spatial/time shape. These shifts provide the primary OOD evidence within the stated coefficient family; off-family directional responses are examined separately in Appendix~\ref{app:functional}. Principal ID, complete OOD, and selected-82 values have distinct model/population/aggregation identities; no cross-study OOD/ID ratio is used.

\FloatBarrier
\subsection{Beyond-family functional-response stress tests and development studies}
\label{app:functional}
At fixed $(K,r,T)$, the generating family is a low-dimensional manifold $\mathcal M=\{\sigma_{\sigma_0,\beta,\gamma}\}$. A direction normal to its sampled tangent span is not produced by an infinitesimal change of its generating parameters. A carrier-null perturbation also satisfies $h(K,t)=0$, so its first-order analytic-reference response is zero.

To distinguish volatility from variance inputs, set $\mathcal G_\sigma[\sigma]=\mathcal G[\sigma^2]$. For a smooth admissible perturbation $\sigma+\varepsilon h$, $Z=D\mathcal G_\sigma[\sigma]h$ satisfies
\begin{equation}
\partial_\tau Z=\mathcal A_{\sigma^2,\tau}Z+\sigma hS^2v_{SS},\qquad Z(S,0)=0.
\label{eq:tangent}
\end{equation}
Both $\sigma,h$ are evaluated at $(S,T-\tau)$. The factor $\sigma h$ is the derivative of $\tfrac12\sigma^2$; this is not a variance perturbation. $Z_S,Z_{SS}$ are mixed input--spot sensitivities, distinct from ordinary Delta/Gamma.

For a numerical step $M_nv_{n+1}=N_nv_n+b_n$, differentiation gives
\begin{equation}
M_nZ_{n+1}=N_nZ_n+\dot N_nv_n+\dot b_n-\dot M_nv_{n+1}.
\label{eq:discretetangent}
\end{equation}
Initial/boundary values, restriction, and Greek stencils are differentiated consistently. Model AD versus its own FD checks implementation; comparison with this reference tangent checks the \emph{learned response}. Derivative-informed operator learning concerns the latter problem \citep{r38,r39,r40,r65,r66}.

The completed final-FI study probes a stricter off-family regime than the parameter-support OOD test. It uses seed 7713, selected update 6600, and 137 identifiable perturbation pairs from 23 bases. Carrier-null median directional $\RelL$ is $1.10659$ for price, $1.01399$ for Delta response, and $1.02228$ for Gamma response, with slightly negative near-zero cosines. The model responds measurably to the perturbations but does not accurately reproduce the off-family reference tangent. This is a single selected-model functional stress test rather than a five-seed comparison.

An earlier normalized FAL-only input ablation compares five information variants over three seeds. Removing all volatility information increases price error by factors $2.27$ before FAL and $8.32$ after FAL. This establishes within-family value of volatility information; off-family directional response is evaluated in the dedicated stress tests. That earlier model's one-shot off-family response has $\RelL=9.47$ and cosine $0.066$. Five subsequent input/fusion architectures on the original family retain small price errors but do not repair the tested function-direction response. Removing generating parameters or delaying fusion is not sufficient in those tests; it is not a reproduction of every late-fusion design \citep{r67}.

\paragraph{Controlled multi-family development study.}
A separate set includes benchmark-family, B-spline, Fourier, Mat\'ern/Karhunen--Lo\`eve, and localized-bump fields, with 512 training, 40 validation, and 40 Wavelet development surfaces. Price/directional labels use \eqref{eq:discretetangent}. An initial derivative-informed study was confounded by a floor on a loss coefficient, output precision, terminal handling, and output-stage ambiguity; those results do not prove intrinsic target incompatibility.

The controlled follow-up uses one fixed-sensor late-fusion architecture and three configurations: value supervision plus preservation; that baseline plus calibrated JVP/mixed-derivative supervision; and a control retaining the coefficient floor. All retain the preservation term, corrected output implementation, three seeds, and 7,000 updates. A 2,000-resample bootstrap clusters the 40 base surfaces, retaining both directions and all three fixed trained seeds in each sampled cluster. Its intervals condition on these models and reused development surfaces, not new seeds or prospective unseen families.
Table~\ref{tab:jvp} reports the resulting conditional directional and value-error ratios.

\begin{table}[htbp]
\caption{Multi-family development paired error ratios, geometric means. Conditional 95\% intervals cluster surfaces, retaining three fixed seeds. Lower than one is improvement.}
\label{tab:jvp}
\centering\small
\begin{tabular}{lrl}
\toprule
Quantity & Error ratio & Conditional 95\% interval\\
\midrule
Composite directional error &0.584911&[0.554260, 0.618680]\\
Price error &2.413665&[2.344490, 2.484836]\\
Composite value error &1.430692&[1.416221, 1.444485]\\
\bottomrule
\end{tabular}
\end{table}
The three within-seed directional ratios are approximately $0.5834$, $0.5017$, and $0.6837$. Nevertheless, benchmark-family, Wavelet, and worst-family price-preservation criteria are not met. Removing the coefficient floor improves pooled results versus that control, with heterogeneous seed/direction effects. No tested configuration meets both price-preservation and directional-response criteria. This development study targets a distinct beyond-family response objective from the held-out ID and parameter-support evaluations. A prospective new-family test has not been run and would address a distinct extrapolation setting.

A separate variable-sensor study has worst 96-versus-231-sensor price-error ratio $2.2235$. Its population is not pooled with the 40-surface JVP experiment. The fixed-sensor follow-up does not retest sensor transfer. Resolution acceptance and numerical accuracy after a discretization change are different properties \citep{r68,r69}. Principal pricing, spot-Greek, and quantum-compatible results continue to refer to the benchmark-family FI model, not these development models.

\FloatBarrier
\subsection{Historical terminal sensitivity, PP, and recorded timing}
\label{app:historical}
\begin{table}[htbp]
\caption{Historical including-terminal Greek sensitivity and PP price/PDE diagnostics, mean $\pm$ sample SD over five seeds. Price uses 32 ID surfaces; Greeks use the 24-surface L3 subset with 3,239 pairs including 79 terminal pairs. These are not Table~\ref{tab:principal}'s 3,160-pair nonterminal Greeks. A dash indicates no independently differentiated continuous PP output.}
\label{tab:historical}
\centering\small
\begin{tabular}{lrrr}
\toprule
Metric & PI & FI & FI--PP\\
\midrule
Global $\RelL$ &$0.007244\pm0.001228$&$0.0009645\pm0.0000488$&$0.001043\pm0.0000455$\\
Global price p95 &$0.9294\pm0.1979$&$0.1331\pm0.0046$&$0.1330\pm0.0037$\\
Near price p95 &$2.3338\pm0.3268$&$0.1579\pm0.0073$&$0.1579\pm0.0073$\\
Delta p95 &$0.1726\pm0.0190$&$0.02535\pm0.00188$&---\\
Gamma p95 &$0.03760\pm0.00075$&$0.006469\pm0.000038$&---\\
PDE residual RMS &$0.01049\pm0.00151$&$0.007025\pm0.000224$&$0.006931\pm0.000220$\\
\bottomrule
\end{tabular}
\end{table}
These historical and primary masks view the same frozen predictions, not different trained models. PP slightly worsens global price error. Identical displayed near-strike p95 before/after PP does not imply pointwise equality. Although the PDE residual RMS omits the terminal evaluation row, its last centered interior time stencil uses terminal prices, so the reported near-maturity residual is not independent of the terminal-value convention. Under the including-terminal mask, carrier-only Delta/Gamma p95 is $0.0945468/0.00874230$, versus FI $0.0253491/0.00646887$, reductions of approximately $73.19\%/26.00\%$. This comparison includes the learned correction and output map jointly; it is not mixed into the primary nonterminal attribution.

Table~\ref{tab:timing} records the historical timers; their unequal scopes are explained below.

\begin{table}[htbp]
\caption{Recorded CPU timing, not an equal-scope efficiency comparison. Training includes validation but PI/FI timer boundaries differ.}
\label{tab:timing}
\centering\small
\begin{tabular}{lrrr}
\toprule
Quantity & PI & FI & FI--PP\\
\midrule
Training time (s) &$4779.8\pm196.8$&$437.6\pm42.9$&Same checkpoint\\
Inference time (s) &$0.1490\pm0.0031$&$0.1266\pm0.0021$&$0.2055\pm0.0055$\\
\bottomrule
\end{tabular}
\end{table}
PI excludes initial dataset/model preparation and later held-out inference, but includes collocation, differentiation, checkpointing, and periodic log rewrites. FI also includes initial data/reference preparation and final selected-checkpoint held-out inference/prediction writing. PI has 7,000 updates and 72 validation calls; differentiation time is not separately recorded. Historical mean training totals for DeepONet, Normalized DeepONet, normalized soft penalties, normalized FAL, and FI are $41.78$, $475.89$, $536.49$, $499.24$, and $437.61$ seconds, respectively, with 258,305 parameters each. Neither equal timer scope nor architecture-only efficiency follows from these values. No amortization crossover is inferred without label generation and training costs.
\FloatBarrier
\subsection{Upper-boundary target sensitivity}
\label{app:piboundary}
The canonical PI comparator uses an asymptotic call-price target at the query edge $S=220$, whereas the ID reference solver extends to $S=605$. We test this difference by replacing only the upper-boundary target with the interpolated reference-solver trace at $S=220$. The architecture, optimizer schedule, collocation protocol, loss weights, validation checkpoint rule, and five seeds remain fixed. The sensitivity was predeclared before training; all five runs completed the 7,000-update protocol. Per-seed configurations, training logs, and selected-checkpoint hashes are retained.

Table~\ref{tab:pi_boundary} shows small changes in PI errors and a persistent gap to FI. The tested boundary-target difference therefore does not account for the principal FI--PI separation.

\begin{table}[htbp]
\caption{Upper-boundary target sensitivity: five-seed means under the principal evaluation protocol. Trace-matched PI changes only the upper-boundary target. Lower is better; bold marks the lowest mean in each row.}
\label{tab:pi_boundary}
\centering\small
\begin{tabular}{lrrr}
\toprule
Metric & Canonical PI & Trace-matched PI & FI\\
\midrule
Price $\RelL$ &0.00724&0.00723&$\mathbf{0.000965}$\\
Near price p95 &2.3338&2.3051&$\mathbf{0.1579}$\\
Delta p95 &0.1547&0.1528&$\mathbf{0.0245}$\\
Gamma p95 &0.0355&0.0358&$\mathbf{0.00434}$\\
PDE residual RMS &0.01049&0.01045&$\mathbf{0.00703}$\\
\bottomrule
\end{tabular}
\end{table}
\FloatBarrier
\section{Quantum-compatible realizations and resource accounting}
\label{app:quantum}
This appendix studies two quantum-compatible realizations of a pretrained dense FI model. QO exactly compiles selected frozen linear maps through signed-core orthogonal factorizations, whereas QDO replaces eligible dense maps by restricted diagonal--orthogonal ($DQ$) parameterizations followed by supervised adaptation. These constructions differ from Quantum DeepONet, which trains an orthogonally parameterized classical analogue and transfers the learned angles for quantum-circuit evaluation \citep{r15}, and from recent quantum neural operators based on alternative structured parameterizations or spectral embeddings \citep{r16,r71}. All reported results here are classical or idealized algebraic evaluations: no physical routed depth, finite-shot quantum execution, or measured end-to-end hardware speedup is claimed.

\FloatBarrier
\subsection{Exact frozen-model compilation}
\label{app:qo}
For $\mathbf W\in\R^{m\times n}$, take a full SVD $\mathbf W=U\Sigma V^\top$. Select diagonal sign matrices $D_U,D_V$ with $\det D_U=\det U$, $\det D_V=\det V$. Then
\begin{equation}
U_+=UD_U\in SO(m),\quad V_+=VD_V\in SO(n),\quad
\Sigma_\pm=D_U\Sigma D_V,\qquad \mathbf W=U_+\Sigma_\pm V_+^\top.
\label{eq:qo}
\end{equation}
Because $D_U^2=D_V^2=I$, multiplication gives $(UD_U)(D_U\Sigma D_V)(VD_V)^\top=U\Sigma V^\top$. Determinants place the factors in the special-orthogonal groups; plane-rotation elimination gives a finite Givens decomposition. Reflection signs must remain in the signed core. Unary amplitude encoding and Givens/RBS circuits implement these orthogonal actions in the unary subspace \citep{r15,r53,r54}. A single replacement $\mathbf W\mapsto Q$ would not preserve a general matrix.

QO's eligible set consists of all 15 dense maps. The realized compiler selects \texttt{trunk.2}, a $128\times128$ map evaluated at 3,321 queries, and \texttt{fusion.0}, a $128\times384$ map evaluated once:
\begin{equation}
128^2\cdot3321+128\cdot384=54,460,416
\end{equation}
semantic MACs, or 49.5563\% of the whole dense-map census $109,896,064$. In this all-dense-map census, eligible and whole-map coverage coincide. Reconstructed blocks and ideal orthogonal actions pass the recorded checks, with compilation-induced relative price discrepancy $9.06\times10^{-17}$. This measures an idealized selected computation, not finite-shot accuracy. The signed singular core, biases, nonlinearities, carrier, FAL, and other retained operations remain classical; hardware additionally needs preparation, readout, and precision control.

\FloatBarrier
\subsection{Structured representation, limitations, and initialization}
\label{app:qdo}
For a square block, classical Procrustes gives \citep{r60}
\begin{equation}
\min_{Q^\top Q=I}\norm{\mathbf W-Q}_F^2=\sum_i(\varsigma_i-1)^2,\qquad Q_*=UV^\top.
\end{equation}
Here $\varsigma_i$ are the singular values of $\mathbf W$. The restricted diagonal--orthogonal family used by QDO is
\begin{equation}
\mathbf W\approx DQ,\qquad D=\diag(d_1,\ldots,d_n),\quad Q=Q_{\rm circ}\in\mathcal Q_{\rm circ}\subseteq O(n).
\end{equation}
Positive scales are parameterized by $d_i=e^{s_i}$; determinant/reflection signs require separate treatment. For a fixed $Q$, rowwise least squares gives $d_i=\langle w_i,q_i\rangle$ without sign constraints, or $\max(\langle w_i,q_i\rangle,0)$ for nonnegative scales. A zero optimum is approached but not attained by a strictly positive exponential parameterization.

\begin{proposition}[Representability and approximation obstruction]
\label{prop:dq}
For square $\mathbf W$, unrestricted $Q\in O(n)$ and nonnegative diagonal $D$ represent $\mathbf W=DQ$ exactly iff its nonzero rows are mutually orthogonal, equivalently $\mathbf W\mathbf W^\top$ is diagonal. Put $\gamma_{\rm row}(\mathbf W)=\norm{\offdiag(\mathbf W\mathbf W^\top)}_F$ and $\mu_W=\norm{\mathbf W}_{\rm op}$. Every such $A=DQ$ satisfies
\begin{equation}
\norm{\mathbf W-A}_F\ge\sqrt{\mu_W^2+\gamma_{\rm row}(\mathbf W)}-\mu_W.
\label{eq:dqlower}
\end{equation}
Restricting $Q$ to $\mathcal Q_{\rm circ}$ cannot reduce the optimum error.
\end{proposition}
\begin{proof}
$\mathbf W=DQ$ implies $\mathbf W\mathbf W^\top=D^2$. Conversely, normalize the nonzero orthogonal rows and complete an orthonormal row set at zero rows; row norms form $D$. For the bound, put $E=\mathbf W-A$. Since $AA^\top=D^2$,
\begin{align}
\gamma_{\rm row}(\mathbf W)&\le\norm{\mathbf W\mathbf W^\top-AA^\top}_F
=\norm{E\mathbf W^\top+AE^\top}_F\nonumber\\
&\le\norm E_F(\norm{\mathbf W}_{\rm op}+\norm A_{\rm op})
\le\varepsilon(2\mu_W+\varepsilon),\qquad \varepsilon=\norm E_F.
\end{align}
Solving the quadratic inequality proves \eqref{eq:dqlower}. A smaller circuit set cannot improve its minimum.
\end{proof}
Row orthogonality is not sufficient for a restricted-depth family. An $SO(n)$-only, positive-scale model needs the appropriate determinant sign unless an external sign operation is retained. Row geometry, circuit restrictions, and fitting error can all contribute; high-coverage error is not automatically an optimization failure.

Initial substitutions fit dense activations $X$ and outputs $Y$ using the nonnegative-clamped normalized quadratic objective
\begin{equation}
\frac{\operatorname{tr}(YY^\top)-2\langle M,YX^\top\rangle+
\operatorname{tr}(MXX^\top M^\top)}{\operatorname{tr}(YY^\top)},
\qquad M=\diag(e^s)Q_{\rm BB}.
\label{eq:activationfit}
\end{equation}
Here $Q_{\rm BB}=Q_{\rm BB}(\theta)\in\mathcal Q_{\rm circ}$ denotes the implementation-specific parameterized orthogonal factor used in the activation fit. The retained records identify the logical stage count and aggregate parameterization but do not recover the within-stage pairing/permutation schedule, so no more specific circuit topology is inferred from the angle count. Angles and log scales start at zero. LBFGS uses learning rate 0.8, maximum 120 iterations/240 evaluations, history size 30, strong-Wolfe line search, gradient tolerance $10^{-10}$, and change tolerance $10^{-13}$. The selected \texttt{trunk.2} initialization is reused; \texttt{trunk.4} is fitted independently of coverage. This initialization is followed by supervised adaptation; the matched attribution study below separates structured-map-only from full-network updates.

\FloatBarrier
\subsection{Supervised adaptation, eligible coverage, and implementation cost}
\label{app:qadapt}
The parent is FI-DeepONet (archive alias E5), seed 7713, update 6600. QPF is the historical protocol identifier for QDO supervised adaptation, not an additional model family. The historical QPF protocol adapts all registered trainable parameters after structured initialization; its selected C94 instance for seed 4262703989 is update 1900. For each full 128-dimensional map, the archived structured parameterization contains eight logical Givens stages, totaling 3,584 rotation angles, together with 128 diagonal scale parameters and bias. The retained manuscript package does not recover the within-stage pairing/permutation schedule or any stage-to-stage reordering; accordingly, ``stage'' denotes a logical parameterization stage rather than routed circuit depth. QPF fits reconstructed pre-FAL normalized-price MSE on 410 original-training surfaces, holding out a disjoint training-only 102-surface subset for selection. Batches are uniform 32 surfaces by 512 queries. Learning rate stays $10^{-3}$ through update 800, then follows cosine decay to $10^{-5}$ at update 2500. The matched attribution experiment in Appendix~\ref{app:qmatched} reuses the same data/update/schedule semantics but freezes unreplaced parameters in its local-only arm. C100 is a separate 3,500-update endpoint diagnostic, not part of the four-candidate knee selection.

QDO admits only ten $128\times128$ hidden maps: $C_{\rm elig}=108,953,600$ versus $C_{\rm whole}=109,896,064$ dense-map MACs. C94 substitutes \texttt{trunk.2} and a selected subset of \texttt{trunk.4}, recorded in the archived summary as fourteen output blocks, with
\begin{align}
C_{\rm repl}^{94}&=102,021,120,&\rho_{94}&=0.9363721804511278,&\omega_{94}&=0.9283418922082596,\nonumber\\
C_{\rm repl}^{99}&=108,822,528,&\rho_{99}&=0.998796992481203,&\omega_{99}&=0.9902313516888103,\label{eq:coverages}\\
C_{\rm repl}^{100}&=108,953,600,&\rho_{100}&=1,&\omega_{100}&=0.9914240422659724.\nonumber
\end{align}
Here $f_{\rm elig}=C_{\rm elig}/C_{\rm whole}$, $\rho=C_{\rm repl}/C_{\rm elig}$, and $\omega=f_{\rm elig}\rho$. The symbol $\alpha$ remains reserved for the maturity-rescaling exponent. Every count uses the same semantic MAC convention. They are not percentages of all end-to-end computation.

The saved partial implementation evaluates both the full classical dense map and the structured map, then overwrites selected outputs. Replaced-MAC coverage therefore does not measure current runtime reduction. The denominator excludes reference-price evaluation, nonlinearities, latent contraction, FAL, PP, data movement, state preparation, readout, and other quantum overhead. Materialized $DQ$ remains an ordinary dense matrix when evaluated classically.

Figure~\ref{fig:app_quantum} separates exact matrix factorization,
the restricted QDO family, and the denominators used for semantic MAC
coverage. No arrow or area in this schematic represents measured latency.
\begin{figure}[!htbp]
\centering
\begingroup\fontfamily{ptm}\selectfont\mathversion{figuretimes}

\begin{tikzpicture}[x=1cm,y=1cm]
\node[anchor=west,font=\footnotesize\bfseries] at (0,0) {(a) What each realization preserves};
\node[appbox,text width=5.7cm,minimum height=1.75cm] at (3.15,-1.25)
{\textbf{QO: exact factorization}\\[4pt]
$\mathbf W=U_+\Sigma_\pm V_+^\top$\\
$U_+,V_+$: rotations; $\Sigma_\pm$: classical signed core\\
No supervised adaptation};
\node[appbox,text width=5.7cm,minimum height=1.75cm,draw=appRed] at (10,-1.25)
{\textbf{QDO: restricted approximation}\\[4pt]
$\mathbf W\approx DQ_{\rm circ}$\\
Unrestricted exact $DQ$: diagonal $\mathbf W\mathbf W^\top$\\
Restricted family + supervised adaptation};
\node[anchor=west,font=\footnotesize\bfseries] at (0,-2.75) {(b) C94 semantic MAC census};
\fill[appRed!85] (0,-3.5) rectangle (12.06844,-3.1);
\fill[appGray!55] (12.06844,-3.5) rectangle (12.8885,-3.1);
\fill[appNavy] (12.8885,-3.5) rectangle (13,-3.1);
\node[font=\scriptsize,anchor=west] at (0,-3.9) {Replaced: $102{,}021{,}120$};
\node[font=\scriptsize,anchor=east] at (13,-3.9) {Remaining: $7{,}874{,}944$};
\node[appbox,text width=5.7cm,minimum height=1cm] at (3.15,-4.9)
{Eligible denominator\\$108{,}953{,}600$\quad$\rho=93.6372\%$};
\node[appbox,text width=5.7cm,minimum height=1cm] at (10,-4.9)
{Whole dense-map denominator\\$109{,}896{,}064$\quad$\omega=92.8342\%$};
\node[appnote,text width=12.7cm] at (6.5,-5.85)
{Red: replaced eligible work. Gray: retained eligible work. Navy: ineligible dense maps.\\The saved partial implementation evaluates both dense and structured outputs; this bar is not runtime.};
\end{tikzpicture}

\endgroup
\caption{Quantum-compatible realization scope. (a) QO preserves the selected
matrix through two orthogonal factors and a classical signed core; QDO
restricts the matrix family and then adapts it. The diagonal-Gram condition
is necessary and sufficient only for unrestricted $Q\in O(n)$ with
nonnegative scales, and is merely necessary for the restricted circuit family.
(b) C94 replaced-MAC accounting: the small navy portion is ineligible
dense-map work, so eligible coverage $\rho$ exceeds whole-map coverage $\omega$.
Counts refer to the same semantic dense-map workload, excluding preparation,
measurement, carrier evaluation, nonlinearities, and other overhead.
Neither the proportional bar nor QO exactness demonstrates hardware speedup.}
\label{fig:app_quantum}
\end{figure}

\FloatBarrier
\subsection{Matched adaptation attribution}
\label{app:qmatched}
To separate recovery by the structured maps from compensation elsewhere in the network, we rerun C94 with the same parent, replacement support, 410/102 training-only fitting/selection split, minibatch schedule, optimizer schedule, update budget, and checkpoint rule. The only intervention is the trainable parameter set. The frozen arm applies the fitted DQ substitution without supervised adaptation; the local arm updates only the substituted structured parameters (and their replaced biases where applicable); the full arm follows the historical QPF scope and updates all registered trainable parameters. Three paired QPF seeds are used for the trainable arms.

Table~\ref{tab:qadaptattr} compares fidelity recovery and unreplaced-parameter drift under this matched intervention.

\begin{table}[htbp]
\caption{C94 matched adaptation attribution on the locked 32-surface synthetic evaluation. Local and full rows are three-seed mean $\pm$ sample SD. The recovery fraction is computed on global price $\RelL$ from the frozen substitution toward the full-adaptation result.}
\label{tab:qadaptattr}
\centering\small
\begin{tabular}{lrr}
\toprule
Arm & Global price $\RelL$ & Unreplaced-parameter drift\\
\midrule
Dense FI parent & 0.000836816 & ---\\
DQ frozen & 0.00808648 & 0\\
DQ local & $0.00180191\pm0.0000673$ & 0\\
DQ full & $0.00178871\pm0.0000160$ & $0.08698\pm0.00200$\\
\bottomrule
\end{tabular}
\end{table}

Local-only adaptation accounts for $99.79\%$ of the frozen-to-full global-$\RelL$ recovery. Full adaptation has a slightly lower mean error, and its unreplaced parameters move measurably (relative drift $0.08698\pm0.00200$; replaced-parameter drift $0.02010\pm0.00091$), but the extra full-over-local error reduction is only about $0.73\%$ on the global-$\RelL$ scale. Parameter drift alone is not causal evidence; these results therefore support recoverability within the structured maps but do not support full-network compensation as necessary for the bulk of fidelity recovery. A historical block-sparse proxy lacks recoverable matched adaptation semantics, so no equal-adaptation DQ-versus-block-sparse claim is made.

\FloatBarrier
\subsection{Training-only score and knee selection}
\label{app:qselection}
With $\varepsilon=10^{-12}$, the six-component selection score is
\begin{align}
J={}&0.35\log(\varepsilon+L_g)+0.20\log(\varepsilon+L_n)
+0.15\log(\varepsilon+L_\Delta)+0.15\log(\varepsilon+L_\Gamma)\nonumber\\
&+0.10\log(\varepsilon+L_w)+0.05\log(\varepsilon+L_p),
\qquad R_{\rm comp}=e^{J-J_{\rm E5}}.
\label{eq:qscore}
\end{align}
$L_g$ is pooled price $\RelL$, $L_n$ near-strike price p95, and $L_\Delta,L_\Gamma$ global finite-point AD absolute p95. For post-FAL, pre-PP price $P$, its repaired grid $P_{\rm PP}$, and $C=\Cref$,
$L_p=(\norm{P_{\rm PP}-P}_2/(\norm{P-C}_2+\varepsilon))^2$.
$L_w$ is the maximum of six surface-averaged normalized weak-condition components described below. The earliest minimum $(J,\mathrm{step})$ among evaluations every 100 updates selects the checkpoint. Primary $R_{\rm comp}=1.305077932190532$ is a training-holdout composite, not a price-error ratio, a formal-validation statistic, or a hardware metric.

Four candidates C49, C87, C94, and C99 are represented by median selected $R_{\rm comp}$ over three paired seeds. Nondominance minimizes this cost while maximizing eligible coverage and $S_i=G_{\rm ref}(\eta=c=0.01)$. With $\ell_i=\log S_i$, define
\begin{equation}
d_i=\left[\left(\frac{R_i-R_{\min}}{R_{\max}-R_{\min}+10^{-30}}\right)^2+
\left(1-\frac{\ell_i-\ell_{\min}}{\ell_{\max}-\ell_{\min}+10^{-30}}\right)^2\right]^{1/2}.
\label{eq:knee}
\end{equation}
Ranges use all four candidates. Select the nondominated point minimizing $d_i$, breaking ties toward smaller $R_i$: C94 is selected. Accuracy/financial flags are evaluated afterward, not used as a general feasibility filter. C94 does not pass the frozen ID/OOD accuracy-preservation gates; it must not be described as a successful lossless migration. Independent replay of those gates requires their numerical tolerance/configuration records in addition to these summaries.

\paragraph{Weak-condition components are surrogate diagnostics.}
Projected prices are reversed in time and transferred by tensor-product linear interpolation from uniform auxiliary coordinates $x_{\rm w}\in[-1,1]$, $\tau_{\rm w}\in[0,1]$ to grids G1/G2/G4 of sizes $(17,9)$, $(33,17)$, and $(65,33)$. With $b_{\rm w}=(1-x_{\rm w}^2)_+^2[\tau_{\rm w}(1-\tau_{\rm w})]_+^2$, the three test functions are
\begin{equation}
\begin{gathered}
b_{\rm w}(1+0.2x_{\rm w}+0.1\tau_{\rm w}),\qquad b_{\rm w}(1-0.15x_{\rm w}+0.3\tau_{\rm w}+0.1x_{\rm w}\tau_{\rm w}),\\
b_{\rm w}(1+0.25x_{\rm w}^2-0.2\tau_{\rm w}+0.1\tau_{\rm w}^2).
\end{gathered}
\end{equation}
Each grid component averages squared weak residuals integrating
$\phi\,\partial_{\tau_{\rm w}}v+0.02(\partial_{x_{\rm w}}\phi)(\partial_{x_{\rm w}}v)+0.03\phi v$, with zero source/drift. On every grid, $\partial_{\tau_{\rm w}}v$, $\partial_{x_{\rm w}}v$, and $\partial_{x_{\rm w}}\phi$ use custom five-point fourth-order centered differences where two neighbors on each side are available. At the remaining points, \texttt{np.gradient(edge\_order=2)} supplies second-order centered differences adjacent to the endpoints and second-order one-sided endpoint differences. Nested Simpson quadrature integrates first in $x_{\rm w}$, then in $\tau_{\rm w}$. Three additional components integrate squared coarse/fine differences by Simpson quadrature on the finer grid after interpolating the coarser field to that grid, for G1--G2, G1--G4, and G2--G4. Surface means are normalized as $(\varepsilon+\text{candidate mean})/(\varepsilon+\text{E5 mean})$; their maximum is $L_w$. This is a frozen surrogate weak-condition diagnostic, \emph{not the
local-volatility PDE residual}.

\FloatBarrier
\subsection{Fidelity--coverage frontier}
\label{app:qfrontier}
The frozen candidate grid is evaluated on the same 32-surface formal population. C49/C87/C94/C99 use the original three paired QPF seeds; C100 uses three different historical seeds and is descriptive only. The mean sequence is nonmonotone, so increasing replacement coverage is not identified with either monotone fidelity loss or monotone improvement.

Table~\ref{tab:qfrontier} lists both coverage denominators alongside the paired synthetic means.

\begin{table}[htbp]
\caption{Synthetic fidelity across the frozen coverage grid. C49--C99 are three-seed paired-QPF means $\pm$ sample SD. C100 uses a non-paired historical seed registry and is descriptive only. Coverage is analytical dense-map accounting, not measured runtime saving.}
\label{tab:qfrontier}
\centering\small
\begin{tabular}{lrrrr}
\toprule
Candidate & $\rho$ & $\omega$ & Global price $\RelL$ & Near-price p95\\
\midrule
C49 & 0.4994 & 0.4951 & $0.0018105\pm0.0000123$ & $0.1137\pm0.0062$\\
C87 & 0.8739 & 0.8665 & $0.0017746\pm0.0000267$ & $0.1427\pm0.0201$\\
C94 & 0.9364 & 0.9283 & $0.0017887\pm0.0000160$ & $0.1318\pm0.0160$\\
C99 & 0.9988 & 0.9902 & $0.0017780\pm0.0000278$ & $0.1350\pm0.0204$\\
C100$^{\dagger}$ & 1.0000 & 0.9914 & $0.0017791\pm0.0000543$ & $0.1401\pm0.0122$\\
\bottomrule
\end{tabular}
\end{table}
\noindent$^{\dagger}$Different historical training seeds; excluded from paired-seed inference. C94 remains the knee selected by the frozen training-only rule in Appendix~\ref{app:qselection}, not a universal accuracy-optimal coverage.

\FloatBarrier
\subsection{Selected-C94 fidelity and unmatched proxy controls}
\label{app:qfidelity}

\begin{table}[htbp]
\caption{
Paired synthetic fidelity on a separate 32-surface population;
selected metrics are summarized in Table~\ref{tab:quantum}B.
The adapted C94 QDO instance is compared with its own dense parent,
not the principal five-seed mean.
Delta and Gamma use AD on the formal grid.
Ratios are QDO/dense; bold marks the lower error in each row.
}
\label{tab:qfidelity}
\centering
\small

\begin{tabular}{lrrr}
\toprule
Metric & Dense FI & QDO-FI & Ratio\\
\midrule
Global price $\RelL$
& $\mathbf{0.0008368}$ & 0.0017702 & 2.115\\
Price RMSE
& $\mathbf{0.05189}$ & 0.10978 & 2.116\\
Price MAE
& $\mathbf{0.03049}$ & 0.09600 & 3.149\\
Near-strike price p95 error
& $\mathbf{0.13147}$ & 0.13238 & 1.007\\
Global AD Delta p95 error
& $\mathbf{0.01267}$ & 0.01481 & 1.169\\
Global AD Gamma p95 error
& 0.002441 & $\mathbf{0.002237}$ & 0.916\\
\bottomrule
\end{tabular}
\end{table}
Table~\ref{tab:qfidelity} shows a fidelity trade-off: global price error roughly doubles, near-strike p95 changes little, and Gamma p95 decreases slightly. The selected C94 instance does not meet the fixed ID/OOD accuracy-preservation criteria. Formal synthetic Greeks use float64 AD on the formal grid; unlike training-only score evaluation, the formal evaluator does not discard nonfinite values, potentially giving nonfinite OOD summaries.

Table~\ref{tab:qcontrols} retains the historical proxy-budget controls; their adaptation protocols are not matched.

\begin{table}[htbp]
\caption{Structured-approximation price $\RelL$ on the matched 32-surface population. Coverage columns use the QDO eligible denominator. QO is an exact algebraic reference, not a new measured hardware realization at each column. Budgets match declared arithmetic/storage proxies, not QPU latency. Lower is better; bold marks the lowest error in each coverage column, including the tied dense and exact-compilation references.}
\label{tab:qcontrols}
\centering\small
\begin{tabular}{lrrr}
\toprule
Realization / control &93.6\%&99.9\%&100\%\\
\midrule
Dense FI
&$\mathbf{0.000836816}$
&$\mathbf{0.000836816}$
&$\mathbf{0.000836816}$\\
QO exact control
&$\mathbf{0.000836816}$
&$\mathbf{0.000836816}$
&$\mathbf{0.000836816}$\\
QDO matrix, classical evaluation &0.001770246&0.001747465&0.001716484\\
Block sparse, arithmetic matched &0.0224863&0.0269211&0.0398960\\
Truncated SVD, arithmetic matched &0.137667&0.141756&0.169582\\
Truncated SVD, storage matched &0.153499&0.177486&0.219893\\
Diagonal + low rank, matched &0.160663&0.180123&0.221246\\
\bottomrule
\end{tabular}
\end{table}
The controls use the selected blocks and nested coverage levels, with no market rows in fitting or selection. QDO has lower listed price errors than these structured classical controls at the declared proxy budgets. Because equal trainable parameter sets, adaptation data, update budgets, and selection rules are not recoverable for these historical controls, this comparison is descriptive and does not isolate the effect of the DQ parameterization. The matched experiment in Appendix~\ref{app:qmatched} addresses adaptation scope within QDO but does not retrofit matched training semantics onto these controls. Coverage trends can be nonmonotone because layer perturbations reinforce or cancel; the trend alone cannot identify whether fitting was independent or joint.

Representative whole-model proxy budgets at these three coverages are $(59,682,544,1,879,944)$, $(56,334,976,1,863,688)$, and $(56,270,464,1,052,680)$ in MACs and bytes. They define the matching envelope, not actual dense-materialized runtime/storage gains. The measured host-side materialization median for a full FI surface is $0.0155$ ms, not evidence that readout dominates current CPU latency or determines QPU costs.

\FloatBarrier
\subsection{Perturbation propagation and conditional resource scenarios}
\label{app:qresources}
For a sequential network, let $\kappa_j$ bound activation Lipschitz constants, $B_j$ bound the hybrid input to layer $j$, and $M_j$ bound old/new matrix operator norms. For a \emph{matrix-substitution comparison with the biases held fixed}, a one-layer-at-a-time hybrid-network argument gives
\begin{equation}
\norm{\widetilde h_L-h_L}\le\sum_{j=0}^{L-1}\kappa_jB_j\norm{\widetilde{\mathbf W}_j-\mathbf W_j}_{\rm op}
\prod_{k=j+1}^{L-1}\kappa_kM_k.
\label{eq:networkbound}
\end{equation}
At a substituted node the immediate error is then bounded by $\kappa_jB_j\norm{\widetilde{\mathbf W}_j-\mathbf W_j}_{\rm op}$; propagating downstream and summing proves the result. If biases also differ, the immediate term becomes
$\kappa_j\!\left(B_j\norm{\widetilde{\mathbf W}_j-\mathbf W_j}_{\rm op}+\norm{\widetilde b_j-b_j}\right)$
before the same downstream product is applied. Thus \eqref{eq:networkbound} is a fixed-bias perturbation bound and is \emph{not} a bound for the full QDO adaptation, where all registered trainable parameters may change. The same construction applies nodewise to the branch--trunk graph. FAL with fixed bounds does not amplify the final scalar price perturbation, but Delta/Gamma require separate differentiation.

Let $\eta\in[0,1)$ denote an assumed cost fraction for replaced work and $c\ge0$ overhead, normalized to the same reference cost as $\omega$. Define
\begin{equation}
G_{\rm ref}=\frac1{1-\omega(1-\eta)+c}.
\label{eq:gain}
\end{equation}
For positive denominator, $k>1$, and $\omega>0$, $G_{\rm ref}\ge k$ iff
\begin{equation}
\omega\ge1-1/k,\quad
0\le c\le\omega-(1-1/k),\quad
0\le\eta\le1-\frac{1-1/k+c}{\omega}.
\label{eq:gainfeasible}
\end{equation}
At $\omega=1-1/k$, only $c=\eta=0$ reaches equality. Thus $\omega=0.9$ permits a tenfold gain only at the zero-cost/zero-overhead boundary.
Table~\ref{tab:qresources} evaluates the conditional cost formula and tenfold-gain feasibility intercepts.

\begin{table}[htbp]
\caption{Analytical scenarios, not measured speedups. $G_{\rm ref}$ uses $\eta=c=0.01$; intercepts describe $G_{\rm ref}\ge10$. Eligible coverage $\rho$ and whole dense-map coverage $\omega$ differ.}
\label{tab:qresources}
\centering\small
\begin{tabular}{rrrrr}
\toprule
$\rho$ & $\omega$ & $G_{\rm ref}$ & $c_{\max}$ at $\eta=0$ & $\eta_{\max}$ at $c=0$\\
\midrule
93.6\% &0.9283419&10.996&0.0283419&0.0305296\\
99.9\% &0.9902314&33.703&0.0902314&0.0911215\\
100\% &0.9914240&35.100&0.0914240&0.0922149\\
\bottomrule
\end{tabular}
\end{table}
C94's census reconciles $102,021,120$ replaced and $7,874,944$ remaining operations to $109,896,064$. This is a dense-linear-map decomposition, not complete model work. A full query returns 3,321 classical prices; selected market quotes are a different output workload.

An end-to-end model must include target device, preparation, communication, precision, routing, shots, recovery, and unreplaced work \citep{r09,r10,r11,r12,r13,r14}. With quantum workload $\mathcal W_Q(H)$ and compatible throughput $\Theta_H$,
\begin{equation}
T_{\rm hybrid}(H)=T_{\rm nonQ}(H)+\mathcal W_Q(H)/\Theta_H,\qquad
S_{\rm e2e}(H)=T_{\rm classical}(H)/T_{\rm hybrid}(H).
\label{eq:e2e}
\end{equation}
For positive quantum work, $S_{\rm e2e}\ge k$ requires $T_{\rm nonQ}<T_{\rm classical}/k$ and
\begin{equation}
\Theta_H\ge\frac{\mathcal W_Q(H)}{T_{\rm classical}(H)/k-T_{\rm nonQ}(H)}.
\end{equation}
A device throughput can be substituted only when workload, accuracy, and execution conventions match. Returning an explicit $M$-point classical surface needs $\Omega(M)$ scalar slots/materialization operations. A stronger $\Omega(M\log(B/\varepsilon))$ bit bound needs a full output cube $[-B,B]^M$, which is not asserted for financial surfaces. If a per-scalar output cost $c_{\rm out}$ is measured for the explicit-list task,
$T_{\rm hybrid}\ge T_{\rm fixed}+c_{\rm out}M$.
Returning an amplitude-encoded state, an expectation estimate, or a few prices is a different task; such quantum pricing bounds do not directly certify acceleration for explicit full-surface output.

\FloatBarrier
\section{Real-Market Evaluation with Constant-Volatility Inputs}
\label{app:market}

\FloatBarrier
\subsection{Data identity, quote provenance, and input mapping}
The study uses Nasdaq-100 index option records identified as NDX/.NDX: an OptionMetrics extract through WRDS with LSEG-linked underlying/identifier fields. The observation window is 13 May--7 August 2026, covering 60 trading dates. The own-IV protocol contains 13,131 calls. The primary held-out-IV protocol begins on 26 May and contains 2,586 targets from 101 date--expiry chains on 52 dates. All 2,586 targets are uniquely recoverable by source-row identifier and a composite quote key. Raw option fields include \texttt{best\_bid} and \texttt{best\_offer}; bid, offer, mid, IV, and target identity are carried from the same daily source row, and frozen model predictions align to that row. Quote time is missing for all rows, so intraday timestamp synchrony is not established. The data are market-derived \emph{constant-volatility inputs}, not reconstructed local-volatility fields.

Cached $q$ is a decimal dividend-yield feature and cached $r$ a decimal one-year zero-coupon-rate proxy shared across 7--90-day contracts on each date. For quote $i$, let
\begin{equation}
\tau_i=\mathrm{DTE}_i/365,\qquad F_i=S_i e^{(r_i-q_i)\tau_i},\qquad S_i^{\rm eff}=S_i e^{-q_i\tau_i}.
\end{equation}
The one-year zero-dividend model is queried via
\begin{equation}
K_i^{\rm model}=100,\quad S_i^{\rm model}=100S_i^{\rm eff}/K_i^{\rm mkt},\quad
\bar t_i=1-\tau_i,\quad [K,r,\sigma_0,\beta,\gamma]=[100,r_i,\widehat\sigma_i,0,0].
\label{eq:marketmap}
\end{equation}
All 231 volatility sensors equal $\widehat\sigma_i$; the 101 payoff sensors use strike 100. Own-IV input sets $\widehat\sigma_i=\mathrm{IV}_i/100$; held-out input reconstructs it as below. Since $\beta=\gamma=0$, the corresponding continuous model has spatially and temporally constant variance at the supplied rate/volatility pair. Under the effective-spot zero-dividend mapping in \eqref{eq:marketmap}, the exact same-input solution is therefore the Black--Scholes carrier itself:
\begin{equation}
V=\Cref,\qquad W=V-\Cref=0,\qquad R^*=0\quad(\tau>0).
\label{eq:marketconstant}
\end{equation}
Thus this experiment checks whether the learned correction remains consistent on a known constant-volatility subfamily; it is not evidence of general local-volatility market transfer. Since $T_{\rm model}=1$, $\bar t_i=1-\tau_i$ retains the actual remaining maturity, not a stretched one-year expiry. Training standardizers remain unchanged; no market scaler is fitted. For dimensionless network output $\widehat c_i=\widehat V_i^{\rm model}/K_i^{\rm model}$, recover
\begin{equation}
\widehat C_i^{\rm mkt}=K_i^{\rm mkt}\widehat c_i.
\end{equation}
The dividend substitution is a constant-volatility quote mapping, not a general dividend-removal identity for every state-dependent local-volatility field.

\FloatBarrier
\subsection{Held-out IV, spread-aware error measures, and neural baselines}
Within a date--expiry chain, strikes and source-row identifiers determine order; duplicate strikes retain the smallest source identifier. Endpoints stay in the fit. Internal positions 4, 9, 14, and so forth, under recorded one-based indexing, are withheld jointly. A natural cubic spline fits $(\log(K/F),\mathrm{IV}/100)$ using retained quotes, with extrapolation disabled. Recorded checks find no target in its own fit, no extrapolated target, and no invalid interpolation. The original eligibility rule used the target's own-IV domain flag before fitting, so predictions were target-value blind but eligibility was not logically fully target-selection blind. A sensitivity based only on reconstructed IV and target-independent metadata gives the \emph{identical realized population}: 2,586 targets, 101 chains, 52 dates, zero additions/removals, and unchanged learned-model ordering.

For market mid $\mathbf C^{\rm mkt}$, same-input Black--Scholes values $\mathbf C^{\rm BS}$, and model $\widehat{\mathbf C}$, distinguish
\begin{equation}
\operatorname{rRMSE}_{\rm mkt}=\frac{\norm{\widehat{\mathbf C}-\mathbf C^{\rm mkt}}_2}{\norm{\mathbf C^{\rm mkt}}_2},\qquad
\RelL_{\rm op}=\frac{\norm{\widehat{\mathbf C}-\mathbf C^{\rm BS}}_2}{\norm{\mathbf C^{\rm BS}}_2}.
\label{eq:marketmetrics}
\end{equation}
With bid $B_i$, offer $A_i$, mid $M_i=(A_i+B_i)/2$, and half-spread $H_i=(A_i-B_i)/2$, define
\begin{equation}
E_i^{\rm spr}=\frac{|\widehat C_i-M_i|}{H_i},\qquad
D_i^{\rm BS}=\frac{|\widehat C_i-C_i^{\rm BS}|}{H_i},\qquad
I_i=\mathbf 1\{B_i\le\widehat C_i\le A_i\}.
\label{eq:spreadmetrics}
\end{equation}
All 2,586 held-out-IV rows have finite positive spreads and pass this spread-eligibility gate. ATM-short p95 is measured against the same-input Black--Scholes operator, not market mid, using $|\log(K/F)|\le0.03$ and 7--30 DTE. All quote predictions are before PP.

\begin{table}[htbp]
\caption{Market-derived evaluation. Neural entries are five-seed mean $\pm$ sample SD. Operator error and ATM-short p95 use the same-input Black--Scholes reference; their Black--Scholes entries are zero by definition. ``Normalized'' denotes Normalized DeepONet. Lower is better.}
\label{tab:market}
\centering\small
\begin{tabular}{llrrr}
\toprule
Input & Configuration & Market rRMSE & Operator $\RelL$ & ATM-short p95\\
\midrule
Own IV &Black--Scholes &\textbf{0.54\%}&\textbf{0}&\textbf{0}\\
&Normalized&$(19.23\pm4.21)\%$&$0.1912\pm0.0416$&$725.14\pm109.87$\\
&Normalized + FAL&$(4.14\pm0.62)\%$&$0.0419\pm0.0056$&$80.55\pm17.85$\\
&FI&$(1.23\pm0.15)\%$&$0.0143\pm0.0017$&$14.42\pm1.98$\\
\midrule
Held-out IV &Black--Scholes &\textbf{0.98\%}&\textbf{0}&\textbf{0}\\
&Normalized&$(19.64\pm4.29)\%$&$0.1952\pm0.0424$&$722.56\pm110.59$\\
&Normalized + FAL&$(4.26\pm0.61)\%$&$0.0417\pm0.0056$&$79.52\pm18.02$\\
&FI&$(1.53\pm0.13)\%$&$0.0145\pm0.0017$&$14.55\pm2.02$\\
\bottomrule
\end{tabular}
\end{table}

\begin{table}[htbp]
\caption{Spread-aware held-out-IV consistency on the unchanged 2,586-target population. Neural predictions are arithmetic means of the five frozen seed predictions, evaluated after averaging prices; they are not means of five separately computed metrics. Market error is relative to quote mid; BS deviation is relative to the exact same-input Black--Scholes operator. Spread quantities use half-spread $H_i$, so a market-mid error at most one half-spread is necessary and sufficient for lying inside the bid--offer interval.}
\label{tab:marketspread}
\centering\small
\begin{tabular}{lrrrrr}
\toprule
Model & Market rRMSE & Operator $\RelL$ & Inside spread & Median $E^{\rm spr}$ & Median $D^{\rm BS}$\\
\midrule
Black--Scholes & 0.980\% & 0 & 93.74\% & 0.215 & 0\\
Normalized & 19.56\% & 0.1943 & 18.72\% & 11.391 & 10.997\\
Normalized + FAL & 4.05\% & 0.03958 & 39.48\% & 1.612 & 1.471\\
FI & 1.519\% & 0.01444 & 50.04\% & 0.998 & 1.201\\
\bottomrule
\end{tabular}
\end{table}
The ordering of learned models survives both input protocols. Black--Scholes remains more accurate against quoted prices with the same volatility information. FI nevertheless moves the learned approximation into a market-relevant scale: its median market-mid error is about one half-spread, while its median deviation from the exact same-input Black--Scholes operator is $1.201$ half-spreads. The corresponding p95 values are $4.800$ and $4.805$ half-spreads, so an unqualified ``small relative to spread'' claim is not supported. This is approximation consistency, not market local-volatility generalization or calibration.

Figure~\ref{fig:app_market} separates the five-seed metric summary
from the diagnostic computed on averaged seed predictions. This distinction
explains the slightly different FI market rRMSE values $1.53\%$ and $1.519\%$:
the population is unchanged, but the aggregation is different.
\begin{figure}[!htbp]
\centering
\begingroup\fontfamily{ptm}\selectfont\mathversion{figuretimes}

\begin{tikzpicture}
\begin{groupplot}[apphalf,group style={group size=2 by 1,horizontal sep=1.5cm},
 xmin=.55,xmax=4.45,xtick={1,2,3,4},xticklabels={BS,Norm.,+FAL,FI},xlabel={Constant-volatility input model}]
\nextgroupplot[title={(a) Held-out-IV market error},ymode=log,ymin=.5,ymax=30,ylabel={Market rRMSE (\%)},
 ytick={1,2,5,10,20},yticklabels={1,2,5,10,20}]
\addplot[only marks,appNavy,mark=triangle*] coordinates {(1,.98)};
\addplot[only marks,appGray,mark=*,error bars/.cd,y dir=both,y explicit] coordinates {(2,19.64)+-(0,4.29)};
\addplot[only marks,appBlue,mark=square,error bars/.cd,y dir=both,y explicit] coordinates {(3,4.26)+-(0,.61)};
\addplot[only marks,appRed,mark=diamond*,error bars/.cd,y dir=both,y explicit] coordinates {(4,1.53)+-(0,.13)};
\nextgroupplot[title={(b) Averaged-price spread diagnostic},ymin=0,ymax=100,ytick={0,25,50,75,100},ylabel={Predictions inside bid--offer (\%)}]
\addplot[only marks,appNavy,mark=triangle*] coordinates {(1,93.74)};
\addplot[only marks,appGray,mark=*] coordinates {(2,18.72)};
\addplot[only marks,appBlue,mark=square] coordinates {(3,39.48)};
\addplot[only marks,appRed,mark=diamond*] coordinates {(4,50.04)};
\node[font=\scriptsize,align=left,anchor=north west] at (rel axis cs:.03,.94)
{Five seed prices averaged\\before computing this metric};
\end{groupplot}
\end{tikzpicture}

\endgroup
\caption{Two aggregation conventions on the same 2,586 held-out-IV quotes.
(a) Per-seed market rRMSE, summarized by five-seed mean and sample SD,
from Table~\ref{tab:market}; Black--Scholes is deterministic and has no
seed error bar. The vertical axis is logarithmic. (b) Fraction inside the
bid--offer interval, computed from the arithmetic average of five seed
prices, from Table~\ref{tab:marketspread}. No interval is supplied for (b).
``Norm.'' denotes Normalized DeepONet and ``+FAL'' its FAL-only extension.
FI improves on both neural baselines, while same-input Black--Scholes remains
more accurate and more frequently inside the spread. The panels do not
constitute local-volatility calibration or a trading-performance test.}
\label{fig:app_market}
\end{figure}
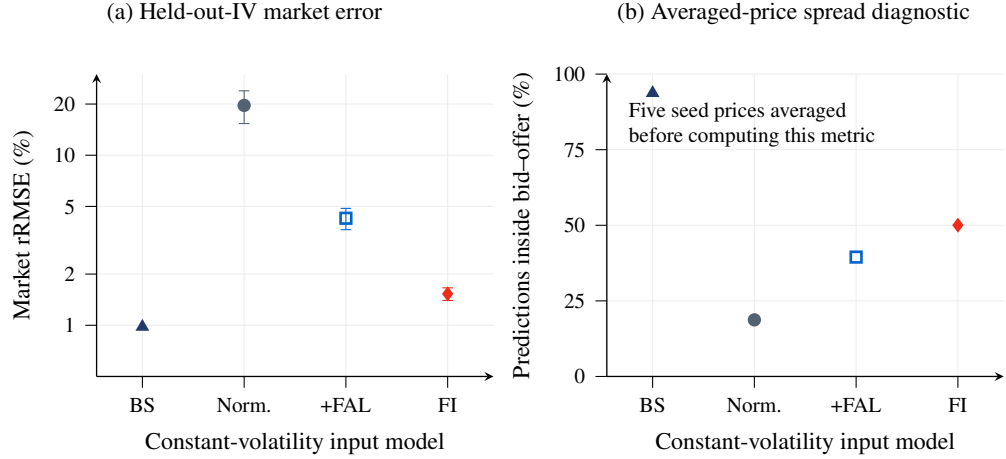

\FloatBarrier
\subsection{Temporal and IV-reconstruction robustness}
\label{app:marketrobustness}
The following temporal and IV-reconstruction diagnostics use the arithmetic mean of the five frozen FI seed predictions, as in Table~\ref{tab:marketspread}. They are distinct from both the per-seed statistics in Table~\ref{tab:market} and the single-parent QDO comparisons. The 52 held-out-IV dates are frozen chronologically before evaluation: the final 16 dates (1,376 quotes) form a sealed robustness subset. No model retraining or checkpoint selection uses either period. On the sealed dates, FI market rRMSE is $1.452\%$, operator $\RelL$ is $1.444\%$, and inside-spread fraction is $55.60\%$; same-input Black--Scholes market rRMSE is $0.893\%$ and inside-spread fraction $95.13\%$. The direction and scale of the operator-consistency result therefore persist late in the sample.

Table~\ref{tab:marketrobustness} holds the target population fixed while changing only the target-excluded IV interpolator.

\begin{table}[htbp]
\caption{Held-out-IV reconstruction sensitivity on the identical 2,586 targets. Natural cubic is the frozen primary protocol; PCHIP and linear interpolation use the same withheld positions, target exclusion, endpoint treatment, and no extrapolation.}
\label{tab:marketrobustness}
\centering\small
\begin{tabular}{lrrrr}
\toprule
IV method & IV RMSE (decimal) & FI market rRMSE & FI operator $\RelL$ & FI inside spread\\
\midrule
Natural cubic & 0.003206 & 1.519\% & 0.0144449 & 50.04\%\\
PCHIP & 0.002708 & 1.437\% & 0.0144411 & 50.08\%\\
Linear & 0.002651 & 1.427\% & 0.0144395 & 50.62\%\\
\bottomrule
\end{tabular}
\end{table}
Across these reconstruction methods the FI operator $\RelL$ range is only about $0.037\%$ relative, supporting stability to reasonable target-excluded IV interpolation choices rather than a method-selection advantage.

\FloatBarrier
\subsection{Multi-coverage QDO/FI market bridge and error alignment}
\label{app:marketpaired}
All market-bridge evaluations are frozen inference: no training, fitting, calibration, coverage selection, seed selection, checkpoint selection, filter change, IV-method change, or temporal re-splitting uses market outcomes. C49/C87/C94/C99 each use the same three paired QPF seeds and descend from the same dense FI parent. C100 uses a different historical seed registry and is descriptive only. A frozen C94 replay reproduces the historical 2,586-row predictions to maximum absolute discrepancy $4.55\times10^{-13}$ and exactly reproduces the reported market and operator metrics within floating-point tolerance.

Table~\ref{tab:marketpaired} summarizes the paired coverage comparisons and their conditional chain/date intervals.

\begin{table}[htbp]
\caption{Frozen quantum-to-market bridge. Full-sample ratios are geometric means over three paired QPF seeds; all 12 seed-level ratios are below one. Chain/date intervals use 5,000 cluster resamples conditional on the fixed trained realizations. Sealed ratios use the frozen 16-date subset. Perturbation scales are $|\mathrm{QDO}-\mathrm{FI}|$ divided by half-spread.}
\label{tab:marketpaired}
\centering\small
\resizebox{\linewidth}{!}{
\begin{tabular}{lrrrrrr}
\toprule
Coverage & Full RMSE ratio & Chain 95\% CI & Date 95\% CI & Sealed ratio & Wins & Median / p95 scale\\
\midrule
C49 & 0.7999 & [0.7507,0.8471] & [0.7433,0.8495] & 0.8720 & 3/3 & 1.215 / 5.373\\
C87 & 0.7617 & [0.7111,0.8101] & [0.7082,0.8074] & 0.8192 & 3/3 & 1.137 / 4.877\\
C94 & 0.7694 & [0.7172,0.8181] & [0.7147,0.8174] & 0.8345 & 3/3 & 1.154 / 5.139\\
C99 & 0.7460 & [0.6931,0.7938] & [0.6878,0.7934] & 0.8018 & 3/3 & 1.045 / 4.793\\
\bottomrule
\end{tabular}}
\end{table}
The C94 market improvement is therefore not unique: all four paired coverages improve market RMSE in all three seeds, and all chain/date conditional intervals lie below one. The sequence is not strictly monotone in coverage, so these data do not support a claim that higher quantum-compatible coverage systematically improves market fit. C100's non-paired historical seeds give a descriptive geometric-mean RMSE ratio $0.7756$, broadly consistent with the paired candidates but excluded from paired inference.

For an algebraic decomposition of the market change, write $f_Q=f_{\rm FI}+\delta_m$ and $e=f_{\rm FI}-y$, where $y$ is the quote midpoint and $\mathbb E$ below is the equal-weight empirical average over the fixed quote population. Then
\begin{equation}
\operatorname{MSE}(f_Q)-\operatorname{MSE}(f_{\rm FI})
=2\E[e\delta_m]+\E[\delta_m^2].
\label{eq:alignment}
\end{equation}
Across C49/C87/C94/C99, the alignment term is negative in all 12 paired seed-level comparisons and its magnitude exceeds the perturbation energy in all 12. The identity closes to floating-point tolerance. Thus the observed market reduction is consistently associated with favorable alignment of the frozen QDO perturbation with the parent error on this market population. This is sample-dependent error cancellation, not causal evidence of quantum regularization. Synthetic global fidelity is only a partial descriptive predictor of the market ratio (12 paired points); synthetic near-price p95 has the opposite correlation direction, so no single synthetic metric is claimed to predict market performance.

\FloatBarrier
\subsection{Feature-cache provenance and release limits}
\label{app:marketprovenance}
The fixed cache uses a decimal dividend feature $q$ and decimal one-year rate $r$, without a second division by 100. The associated rate series is FRED \texttt{THREEFY1}; the associated dividend source is \texttt{NASDAQNDXDIV}, a dividend-point series, not directly a yield. Original dividend-point observations and historical reset, rolling, point-to-yield, annualization, and compounding transformations are unavailable. Rate export/fill implementation is only partly documented. These gaps are not filled using plausible but unverified formulas.

Downstream quote matching uses exact cache dates and drops rows missing spot, rate, or dividend features. A source check reported a 31 July 2026 rate carried into 3--7 August, a 3--7-calendar-day lag, without recovering the fill code. Exact-date cache matching is therefore not proof of contemporaneous source observations. The raw contract multiplier and currency convention are not populated in the supplied extract; this does not block normalized comparisons because bid, offer, mid, Black--Scholes, and neural predictions share the same quoted price units, but it blocks an independently verified multiplier/currency claim. A replay of ten fixed rows and 90 scalars has maximum absolute difference $9.37\times10^{-17}$ and maximum relative difference $1.68\times10^{-14}$. Quote time is unavailable, so same-row daily alignment is verified but intraday synchronization is not.

Selected exercise, settlement, correction, and IV-inversion metadata also remain incomplete. Accordingly, the mapping is reported from the fixed cache, not as a fully reconstructed source-to-model pipeline. Data licenses and anonymization must be checked before release; quote-level licensed records are excluded from lightweight manifests.

\FloatBarrier
\section{Related Work and Reproducibility}
\label{app:context}
\paragraph{Operator learning and financial surrogates.}
Universal operator approximation \citep{r17}, DeepONet \citep{r18}, Fourier and integral-kernel operators \citep{r19,r20}, and multiple-input constructions \citep{r21} provide ways to amortize families of PDE solutions. Neural pricing and hedging methods \citep{r22}, deep PDE solvers \citep{r23,r24}, Black--Scholes approximation analyses \citep{r25}, and volatility-calibration surrogates \citep{r26} establish the broader financial context. Our contribution concerns the representation of the coefficient-to-price map through input-dependent strike matching. Its empirical scope distinguishes held-out ID prediction, parameter-support extrapolation within the benchmark family, and off-family directional responses.

\paragraph{PDE loss and derivative accuracy.}
Physics-informed learning \citep{r27} and its operator variants \citep{r28,r29} incorporate known equations into training, with outcomes affected by loss balancing and optimization \citep{r30,r31,r32,r36,r37}. Derivative-informed methods explicitly target sensitivities \citep{r38,r39,r40,r65,r66}, while EquiNO and MD-PNOP illustrate other ways to incorporate equation structure \citep{r51,r52}. Our representation uses an analytic carrier independently of a PDE
loss. The controlled factorial separately evaluates how adding the
PDE loss affects price error, PDE residual RMS, and derivative errors.

\paragraph{Financial constraints and analytic representations.}
Option-surface restrictions couple contracts \citep{r02,r33,r34}, and learning-based calibration must account for their implications \citep{r35}. Second-order financial knowledge and smooth or arbitrage-constrained surface models address this structure \citep{r41,r42,r43,r44,r45,r46}. Residual-learning methods refine analytic stochastic-volatility prices \citep{r72}; ETCNN uses terminal masks and moneyness normalization \citep{r47}; and a recent 0DTE model learns a maturity-gated variance correction within a Black--Scholes representation \citep{r70}. Analytic output transforms and hard-constraint constructions provide further precedents \citep{r48,r49,r50}. Building on these ideas, the distinction here is strike-line variance matching for a local-volatility operator, together with its correction order, critical scaling, and output-stage analysis. FAL imposes pointwise bounds; cross-contract shape restrictions remain a separate property.

\paragraph{Analysis and quantum-compatible realizations.}
The correction analysis uses diffusion, parametrix, and short-maturity methods \citep{r55,r56,r57,r58}, together with interior regularity \citep{r62,r63}. Exact compilation uses established Givens/RBS and orthogonal-network constructions \citep{r53,r54}; the restricted-approximation analysis uses Procrustes geometry \citep{r60}. Quantum DeepONet transfers a trained orthogonal model to circuit evaluation \citep{r15}, while other quantum neural operators use different structured parameterizations or spectral embeddings \citep{r16,r71}. Our starting point is an unconstrained dense FI parent, with frozen exact compilation separated from structured approximation and adaptation. Quantum derivative-pricing algorithms \citep{r09,r10,r11,r12,r13,r14} have their own input, output, and hardware assumptions, which must be included when comparing computational costs.

\paragraph{Evaluation scope.}
The principal five-seed ID comparison, matched five-seed carrier--rescaling ablation, complete parameter-support OOD evaluation, three-seed kink intervention, objective factorial, single-checkpoint off-family test, selected-model realization comparison, and selected-model market pairing use distinct evaluation protocols. Each metric is specified by its model, population, output stage, mask, reference, and aggregation. Surface-bootstrap intervals condition on the trained models; they do not measure variation over newly trained seeds. The separate sensor-count study is likewise distinct from the multi-family JVP development study. The carrier--rescaling study uses paired initialization and sampling schedules across exponents and reports pre-FAL mechanism metrics separately from post-FAL inference-stage metrics.

\paragraph{Reproducibility scope.}
The retained synthetic-benchmark records support replay of selected-state
evaluations through fixed data, selected checkpoints, saved predictions,
evaluation masks, configurations, and checkpoint metadata.
The current principal DeepONet baseline is a separate five-seed refit,
and the complete parameter-support OOD study is maintained as a separate
frozen-state replay.
A fresh retraining of every historical principal run has not been performed,
and the original split-generation stream remains incomplete.
Accordingly, our reproducibility claims concern the fixed realized datasets,
audited selected states, and the explicitly identified refits and replays.
Licensed quote-level market records are not redistributed.
\end{document}